\documentclass[reqno]{amsart}

\usepackage[utf8]{inputenc}
\usepackage[english]{babel}
\usepackage{etoolbox,needspace}
\usepackage{enumitem,float,graphicx,leftidx,mathrsfs,txfonts}
\AtBeginEnvironment{thm}{\Needspace{5\baselineskip}}
\AtBeginEnvironment{defn}{\Needspace{5\baselineskip}}

\def\dbar{{\mathchar'26\mkern-12mu d}} 
\newcommand{\Smooth}{\ensuremath{{\mathcal C}^{\infty}}}
\newcommand{\G}{\ensuremath{\mathcal{G}}}
\newcommand{\Greg}{\ensuremath{\mathcal{G}}^{\infty}}
\newcommand{\GLtwo}{\ensuremath{\mathcal{G}}_{L^2}}
\newcommand{\GregLtwo}{\ensuremath{\mathcal{G}}^{\infty}_{L^2}}%

\numberwithin{equation}{section}
\newtheorem{thm}{Theorem}[section]
\newtheorem{defn}[thm]{Definition}
\newtheorem{lem}[thm]{Lemma}
\newtheorem{prop}[thm]{Proposition}
\newtheorem{cor}[thm]{Corollary}

\theoremstyle{definition}
\newtheorem{rem}[thm]{Remark}
\newtheorem{ex}[thm]{Example}

\begin{document}

\title[Factorization of hyperbolic operators]{Factorization of second-order strictly hyperbolic operators with logarithmic slow scale coefficients and generalized microlocal approximations}

\author[M. Glogowatz]{Martina Glogowatz}

\address{Faculty of Mathematics, University of Vienna, Austria}
\email{martina.glogowatz@univie.ac.at}

\date{\today}

\thanks{Partially supported by FWF grant P24420 'Spectral invariants: Index and Noncommutative Residue'}

\begin{abstract}
We give a factorization procedure for a strictly hyperbolic partial differential operator of second order with logarithmic slow scale coefficients. From this we can microlocally diagonalize the full wave operator which results in a coupled system of two first-order pseudodifferential equations in a microlocal sense. Under the assumption that the full wave equation is microlocal regular in a fixed domain of the phase space, we can approximate the problem by two one-way wave equations where a dissipative term is added to suppress singularities outside the given domain. We obtain well-posedness of the corresponding Cauchy problem for the approximated one-way wave equation with a dissipative term.
\end{abstract}
\maketitle 
\setcounter{tocdepth}{1}
\tableofcontents
\setlength{\parskip}{10pt plus 2pt minus 1pt}
\section{Basic notions}
In this section we specify the basic notions that will be needed for our constructions. As the problem is treated within the framework of Colombeau algebras we refer to the literature \cite{Colombeau:85, Moe:92, NPS:98, GKOS:01, GarettoGramchevMoe:05, Garetto:08} for a systematic treatment in this field. 

One of the main objects in our setting are Colombeau generalized functions based on $\GLtwo$ which were first introduced in \cite{BO:92}. The elements in this algebra are given by equivalence classes $u := [(u_{\varepsilon})_{\varepsilon\in (0,1]}]$ of nets of regularizing functions $u_\varepsilon$ in the Sobolev space $H^{\infty}=\cap_{k \in \mathbb{Z}} H^k$ corresponding to certain asymptotic seminorm estimates. More precisely, we denote by $\mathcal{M}_{H^{\infty}}$ the nets of moderate growth whose elements are characterized by the property
\begin{equation*}
\forall \alpha \in \mathbb{N}^n \ \exists N \in \mathbb{N}: \lVert \partial^{\alpha} u_{\varepsilon} \rVert_{L^{2}(\mathbb{R}^n)} = \mathcal{O}({\varepsilon}^{-N}) \ \mbox{as} \ {\varepsilon} \to 0.
\end{equation*}
Negligible nets will be denoted by $\mathcal{N}_{H^{\infty}}$ and are nets in $\mathcal{M}_{H^{\infty}}$ whose elements satisfy the following additional condition
\begin{equation*}
\forall q \in \mathbb{N}: \lVert u_{\varepsilon} \rVert_{L^{2}(\mathbb{R}^n)} = \mathcal{O}({\varepsilon}^{q}) \ \mbox{as} \ {\varepsilon} \to 0.
\end{equation*}
Then the algebra of generalized functions based on $L^{2}$-norm estimates is defined as the factor space $\mathcal{G}_{H^{\infty}} = \mathcal{M}_{H^{\infty}} / \mathcal{N}_{H^{\infty}}$. By abuse of notation, we continue to write $\GLtwo (\mathbb{R}^n)$ for $\G_{H^{\infty}}$. For simplicity, we shall also use the notation $(u_\varepsilon)_{\varepsilon}$ instead of $(u_\varepsilon)_{\varepsilon \in (0,1]}$ throughout the paper. 

Using \cite[Theorem 2.7]{BO:92}, we first note that the distributions $H^{-\infty} = \cup_{k \in \mathbb{Z}} H^k$ are linearly embedded in $\GLtwo(\mathbb{R}^n)$ by convolution with a mollifier $\varphi_\varepsilon(x) = \varepsilon^{-n} \varphi(\varepsilon^{-1}x)$ where $\varphi \in \mathcal{S}(\mathbb{R}^n)$ is a Schwartz function such that
\begin{eqnarray}\label{mollifier}
\int \! \varphi(x) \, dx  =  1, \qquad \int \! x^{\alpha}  \varphi(x) \, dx  =  0 \quad \mbox{for all} \ \alpha \in \mathbb{N}^n, \ |\alpha| \geq 1.
\end{eqnarray}
Further, by the same result, $H^{\infty}(\mathbb{R}^n)$ is embedded as a subalgebra of $\GLtwo (\mathbb{R}^n)$. 

More generally, we introduce Colombeau algebras based on a locally convex vector space $E$ topologized through a family of seminorms $\{ p_i \}_{i \in I}$ as in \cite[Section 1]{GarMoe:2011a}. To continue, the elements
\begin{equation}\label{eqn:loc_convex_VS}
\begin{split}
\ensuremath{\mathcal{M}}_{E} &:= \{ (u_{\varepsilon})_{\varepsilon} \in E^{(0,1]} \hspace{2pt} | \hspace{2pt} \forall i \in I  \ \exists N \in \mathbb{N}: p_i( u_{\varepsilon}) = \mathcal{O} ({\varepsilon}^{-N}) \ \text{as } \varepsilon \to 0 \} \\
\ensuremath{\mathcal{M}}_E^{\infty} &:= \{ (u_{\varepsilon})_{\varepsilon} \in E^{(0,1]} \hspace{2pt} | \hspace{2pt} \exists N \in \mathbb{N} \ \forall i \in I:  p_i( u_{\varepsilon}) = \mathcal{O} ({\varepsilon}^{-N}) \ \text{as} \ {\varepsilon} \to 0\} \\
\ensuremath{\mathcal{M}}_E^{lsc} &:= \{ (u_{\varepsilon})_{\varepsilon} \in E^{(0,1]} \hspace{2pt} | \hspace{2pt} \forall i \in I  \ \exists (\omega_\varepsilon)_\varepsilon \in \Pi_{lsc}:  p_i( u_{\varepsilon}) = \mathcal{O} (\omega_\varepsilon) \ \text{as} \ {\varepsilon} \to 0\} \\
\ensuremath{\mathcal N}_{E} &:= \{ (u_{\varepsilon})_{\varepsilon} \in E^{(0,1]} \hspace{2pt} | \hspace{2pt} \forall i \in I  \ \forall q \in \mathbb{N}: p_i( u_{\varepsilon}) = \mathcal{O} ({\varepsilon}^{q}) \ \text{as} \ {\varepsilon} \to 0\}
\end{split}
\end{equation}
are said to be $E$-moderate, $E$-regular, $E$-moderate of logarithmic slow scale type and $E$-negligible, respectively. 

Then, $\ensuremath{\mathcal N}_{E}$ is an ideal in $\mathcal{M}_{E}$ and the space of Colombeau algebra based on $E$ is defined by the factor space $\ensuremath{{\mathcal G}_{E}} = \ensuremath{\mathcal{M}}_{E} / \ensuremath{\mathcal N}_{E}$ and possesses the structure of a $\widetilde{\mathbb{C}}$-module. The space of the regular Colombeau algebra $\ensuremath{{\mathcal G}_{E}}^{\infty} = \ensuremath{\mathcal{M}}_{E}^{\infty} / \ensuremath{\mathcal N}_{E}$ is again a $\widetilde{\mathbb{C}}$-module whereas the space of the logarithmic slow scale algebra $\ensuremath{{\mathcal G}_{E}}^{lsc} = \ensuremath{\mathcal{M}}_{E}^{lsc} / \ensuremath{\mathcal N}_{E}$ is a $\mathbb{C}$-module. 
\begin{ex}
Setting $E = \mathbb{C}$ one gets the ring of complex generalized numbers $\widetilde{\mathbb{C}} = \G_{\mathbb{C}}$ with the absolute value as the corresponding seminorm. Furthermore, we denote by $\widetilde{\mathbb{R}} = \G_{\mathbb{R}}$ the ring of real generalized numbers.

Further, let  $\Omega$ be an open subset of $\mathbb{R}^n$. Then the Colombeau algebra $\G(\Omega) = \mathcal{E}_M(\Omega) / \ensuremath{{\mathcal N}}(\Omega)$ is obtained by taking $E=\mathcal{C}^{\infty}(\Omega)$ endowed with the topology induced by the family of seminorms $p_{K,i}(f)=\sup \{ |\partial^{\alpha}f(x)| : x \in K, |\alpha| \le i \}$ with varying $K \Subset \Omega$ and $i \in \mathbb{N}$ for some fixed $f \in \mathcal{C}^{\infty}(\Omega)$.

Another example is the Colombeau algebra $\G_{p,p}(\Omega)$, $1 \le p \le \infty$, when $E$ is set to $W^{\infty,p}(\Omega)$ and the topology is determined by the collection of seminorms $p_i(f) = \sup\{ || \partial^{\alpha} f ||_p : |\alpha| \le i \}$, $f \in W^{\infty,p}(\Omega)$, as $i$ varies over $\mathbb{N}$. We note that $\GLtwo = \G_{2,2}$. For more general Colombeau algebras based on Sobolev spaces we refer to \cite[Section 2]{BO:92}. Using the notation there we have the following equivalences: $\mathcal{E}_{M,p}(\Omega) = \mathcal{M}_{W^{\infty,p}(\Omega)}$ and $\mathcal{N}_{p,p}(\Omega) = \mathcal{N}_{W^{\infty,p}(\Omega)}$. 
\end{ex}
As different growth types are crucial in regularity theory we introduce the set of logarithmic slow scale nets by
\begin{equation*}
\begin{split}
\Pi_{lsc} := \Big\{ (\omega_\varepsilon)_\varepsilon \in \mathbb{R}^{(0,1]} \ \big| & \ \exists \eta \in (0,1] \ \exists c > 0 \ \forall \varepsilon \in (0,\eta]: c \le \omega_\varepsilon,  \\ 
\ & \ \exists \eta \in (0,1] \ \forall p \ge 0 \ \exists c_p > 0: |\omega_\varepsilon|^p \le c_p \log \Big(\frac{1}{\varepsilon} \Big) \ \ \ \varepsilon \in (0,\eta] \Big\}.
\end{split}
\end{equation*}
Further, we call a net $(\omega_\varepsilon)_\varepsilon \in \Pi_{lsc}$ logarithmic slow scale strictly nonzero if there exists $(r_\varepsilon)_\varepsilon \in \Pi_{lsc}$ and an $\eta \in (0,1]$ such that $|\omega_\varepsilon| \ge 1/r_\varepsilon$ for $\varepsilon \in (0,\eta]$. 
\section{Pseudodifferential calculus}

In this section we introduce a general calculus for pseudodifferential operators which are standard quantizations of generalized symbols. Since most of the techniques are similar to the usual theory of pseudodifferential operators we also refer to \cite{Hoermander:3,Taylor:81}. A detailed discussion on pseudodifferential operators with Colombeau generalized symbols can be found in \cite{GarettoGramchevMoe:05,Garetto:08}. As usual, we write $D_{x_j}=-i\partial_{x_j} = -i\frac{\partial}{\partial_{x_j}}$.

We shall initially discuss the main notions of generalized symbols. As already indicated above we will study symbols which satisfy asymptotic growth conditions with respect to $(\omega_\varepsilon)_\varepsilon \in \Pi_{lsc}$. As usual, we use the notation $\langle \xi \rangle := (1 + |\xi|^2)^{1/2}$.

We let $\rho, \delta \in [0,1]$ and $m \in \mathbb{R}$. We denote by $S^m_{\rho,\delta} = S^m_{\rho,\delta}(\mathbb{R}^n \times \mathbb{R}^n)$ the set of symbols of order $m$ and type $(\rho,\delta)$ as first introduced by H\"ormander in \cite[Definition 2.1]{Hoermander:67}. Since the symbol class $S^m_{\rho,\delta}$ satisfies global estimates we remark that our space $\mathcal{M}_{S^m_{\rho,\delta}}$ is different to that in \cite[Section 1.4]{GHO:09}. 

As in the following, we will typically encounter subspaces of $\mathcal{M}_{S^m_{\rho,\delta}}$ subjected to logarithmic slow scale asymptotics we can choose in (\ref{eqn:loc_convex_VS}) $E = S^m_{\rho,\delta}$ and obtain symbols of the form:

\begin{defn}\label{defn:lscSymbol}
Let $m \in \mathbb{R}$. The set of moderate logarithmic slow scale symbols $S^m_{\rho,\delta,lsc} (\mathbb{R}^n \times \mathbb{R}^n)$ of order $m$ and type $(\rho,\delta)$ consists of all $(a_\varepsilon)_\varepsilon \in S^m_{\rho,\delta}(\mathbb{R}^{n} \times \mathbb{R}^n)^{(0,1]}$ such that
\begin{equation*}
\begin{split}
\forall \alpha, \beta \in \mathbb{N}^n & \ \exists (\omega_\varepsilon)_\varepsilon \in \Pi_{lsc} : \\
& q^{(m)}_{\alpha, \beta} (a_{\varepsilon}) \hspace{-1.5pt} := \hspace{-1.5pt} \sup_{(x,\xi) \in  \mathbb{R}^{2n}} | \partial_{\xi}^{\alpha} \partial_x^{\beta} a_\varepsilon (x,\xi) | \langle \xi \rangle^{-m + \rho |\alpha| - \delta |\beta|} = \mathcal{O}( \omega_{\varepsilon}) \quad \text{ as } \varepsilon \to 0.
\end{split}
\end{equation*}
An element of $N^m_{\rho,\delta} (\mathbb{R}^n \times \mathbb{R}^n)$ is said to be negligible in $S^m_{\rho,\delta,lsc} (\mathbb{R}^n \times \mathbb{R}^n)$ if it fulfills the following condition:
\begin{equation*}
\forall \alpha, \beta \in \mathbb{N}^n \ \forall q \in \mathbb{N} : \quad q^{(m)}_{\alpha, \beta} (a_{\varepsilon}) = \mathcal{O}(\varepsilon^q) \quad \text{ as } \varepsilon \to 0 .
\end{equation*}
Then, logarithmically slow scaled symbols of order $m$ are defined as the factor space
\begin{equation*}
\widetilde{S}^m_{\rho,\delta,lsc} (\mathbb{R}^n \times \mathbb{R}^n) := S^m_{\rho,\delta,lsc} (\mathbb{R}^n \times \mathbb{R}^n) / N^m_{\rho,\delta} (\mathbb{R}^n \times \mathbb{R}^n).
\end{equation*}
and in the following we will assume that $\delta < \rho$. In the case that $\rho=1$ and $\delta=0$ we use the abbreviation $\widetilde{S}^m_{lsc}$ for $\widetilde{S}^m_{1,0,lsc}$.
\end{defn}
\begin{rem}
Moreover, the space $\widetilde{S}^{-\infty}_{lsc}$ of logarithmic slow scale symbols of order $-\infty$ consists of equivalence classes $a$ whose representatives $(a_\varepsilon)_\varepsilon$ have the property that
\begin{equation*}
\forall m \in \mathbb{R} \ \forall \alpha, \beta \in \mathbb{N}^n \ \exists (\omega_\varepsilon)_\varepsilon \in \Pi_{lsc}: \quad  q^{(m)}_{\alpha, \beta} (a_{\varepsilon}) = \mathcal{O}( \omega_{\varepsilon}) \quad \text{ as } \varepsilon \to 0.
\end{equation*}
\end{rem}
Since $(\omega_\varepsilon)_\varepsilon$ is a logarithmic slow scale net, we note that the net of a symbol $(a_\varepsilon)_\varepsilon$ in $S^m_{lsc}$ can always be estimated as follows
\begin{equation*}
\forall \alpha, \beta \in \mathbb{N}^n : \quad q^{(m)}_{\alpha, \beta} (a_{\varepsilon}) = \mathcal{O}\Bigl( \log\Big(\frac{1}{\varepsilon}\Big) \Bigr) \quad \text{ as } \varepsilon \to 0.
\end{equation*}
To give an example, let $P(x,D_x) = \sum_{|\alpha| \leq m} a_{\alpha}(x) D^{\alpha}_x$ be a partial differential operator with bounded and measurable coefficients. Then the logarithmically slow scaled regularization of the symbol is given by $p_\varepsilon (x,\xi) := (p(.,\xi) \mathbin{*} \varphi_{\omega_\varepsilon^{-1}})(x)$ and $(p_\varepsilon)_\varepsilon$ is contained in $S^m_{lsc} (\mathbb{R}^n \times \mathbb{R}^n)$.

Also, we will make use of the following symbol class:
\begin{defn}\label{defn:SymbolOpenSet}
Let $U \subset \mathbb{R}^n \times \mathbb{R}^n$ be open and conic with respect to the second variable. We say that a generalized symbol $a$ is in $\widetilde{S}^m_{\rho,\delta,lsc}(U)$ if it has a representative $\left(a_\varepsilon\right)_\varepsilon$ with $a_\varepsilon \in \mathcal C^{\infty} (U)$ for fixed $\varepsilon \in (0,1]$ and for any compact set $K \subset \text{pr}_2(U)$ (independent of $\varepsilon$) and $V_K := \{ (x,\xi) \in U \ | \ \xi \in K \}$ we have
\begin{equation*}
\begin{split}
\forall \alpha,\beta \in \mathbb{N}^n \ \exists (\omega_\varepsilon)_\varepsilon & \in \Pi_{lsc}  \ \exists C>0: \quad \\
& |\partial_{\xi}^{\alpha} \partial_x^{\beta} a_\varepsilon(x,\xi)| \leq C \omega_\varepsilon \langle \xi\rangle^{m-\rho|\alpha|+\delta|\beta|} \qquad (x,\xi) \in V_{K}^c 
\end{split}
\end{equation*}
for $\varepsilon$ sufficiently small and where $V_{K}^c := \{ (x,\lambda\xi) \ | \ (x,\xi) \in V_K, \lambda \ge 1 \}$.
\end{defn}
Note that Definition~\ref{defn:lscSymbol} is equivalent to Definition~\ref{defn:SymbolOpenSet} in the case that $U=\mathbb{R}^n \times \mathbb{R}^n$. For completeness, we recall the definition for global symbols $S^m_{\rho,\delta}(U)$ which is similar to that for local symbols, see \cite[Definition 2.3, page 141]{CP:82}: For $U$ being an open subset of $\mathbb{R}^n \, \! \times \! \, \mathbb{R}^n$, which is conic in the second variable, we say that $a \in S^m_{\rho,\delta}(U)$ if $a \in \mathcal{C}^{\infty}(U)$ and for each compact $K \subset \text{pr}_2(U)$ and $V_K:=\{ (x,\xi) \in U \ | \ \xi \in K \}$ we have
\begin{equation*}
\forall \alpha, \beta \in \mathbb{N}^n: \quad \sup_{(x,\xi) \in V_K^{c}} | \partial_{\xi}^{\alpha} \partial_x^{\beta} a (x,\xi) | \langle \xi \rangle^{-m + \rho|\alpha|-\delta|\beta|} < \infty
\end{equation*}
where $V_K^{c}:= \{(x,\lambda \xi) \ | \ (x,\xi) \in V_K, \lambda \geq 1 \}$.

We choose the following convention for defining the Fourier transform $\mathcal{F}$ of a function $u \in L^2(\mathbb{R}^n)$:
\begin{align*}
\mathcal{F} u(\xi) &:= \hat{u}(\xi) :=  \int_{\mathbb{R}^{n}} e^{-i x \xi} u(x) \,dx := \lim_{\sigma \to 0_+} \int_{\mathbb{R}^{n}} e^{-i x \xi - \sigma \langle x \rangle} u(x) \,dx.
\end{align*}
Then the Fourier transform is an isomorphism on $L^2$ and the inverse Fourier transform of $u \in L^2$ is given by the formula
\begin{align*}
\mathcal{F}^{-1} u(x) &:= \int_{\mathbb{R}^{n}} e^{i x \xi} u(\xi) \,\ \dbar \xi := \lim_{\sigma \to 0_+} \int_{\mathbb{R}^{n}} e^{i x \xi - \sigma \langle \xi \rangle} u(\xi) \,\ \dbar \xi.
\end{align*}
where we have set $\ \dbar \xi := (2 \pi)^{-n} \, d  \xi$. More information of the Fourier transform acting on $\GLtwo$ can be found in \cite{Anton:99}.
As already mentioned above, we will focus on generalized pseudodifferential operators having the following phase-amplitude representation:
\begin{defn}
Let $(a_\varepsilon)_\varepsilon  \in a \in \widetilde{S}^m_{\rho,\delta,lsc}(\mathbb{R}^n \times \mathbb{R}^{n})$ and let $(u_\varepsilon)_\varepsilon$ be a representative of $u \in \GLtwo$. We define the corresponding linear operator $A: \GLtwo \to \GLtwo$ by
\begin{equation*}
A(x,D)u(x) := \int e^{i(x-y) \xi} a(x,\xi) u(y) \,d y \,\ \dbar \xi := (A_\varepsilon (x,D) u_\varepsilon (x))_\varepsilon + \mathcal{N}_{H^{\infty}} (\mathbb{R}^n)
\end{equation*}
where 
\begin{equation*}
\begin{split}
A_{\varepsilon}(x,D_x) u_\varepsilon(x) &:= \int e^{i (x-y) \xi} a_\varepsilon(x,\xi) u_\varepsilon (y) \,dy \,\ \dbar \xi = \int e^{i x \xi} a_\varepsilon(x, \xi) \hat{u}_\varepsilon (\xi) \,\ \dbar \xi =\\
& = \mathcal{F}^{-1}_{\xi \to x} \Bigl( a_\varepsilon(x,\xi) \hat{u}_\varepsilon (\xi) \Bigr) 
\end{split}
\end{equation*}
where the last integral is interpreted as an oscillatory integral. The operator $A$ is called the generalized pseudodifferential operator with generalized symbol $a$. Later on, we will sometimes write $A \in \Psi^m_{\rho,\delta,lsc}(\mathbb{R}^n)$ to denote that $A$ is a generalized pseudodifferential operator with symbol in $\widetilde{S}^m_{\rho,\delta,lsc}(\mathbb{R}^n \times \mathbb{R}^{n})$. In the case that $(\rho,\delta)=(1,0)$, we will write $A \in \Psi^m_{lsc}(\mathbb{R}^n)$ for short.
\end{defn}
\subsection{Generalized point values of a generalized symbol}\label{subsec:GenPointValues}
As above, let $U \times \Gamma \subseteq \mathbb{R}^n \times \mathbb{R}^n$ be open and conic with respect to the second variable and $(a_\varepsilon)_\varepsilon$ a generalized symbol in $S^m_{\rho,\delta,lsc}(U \times \Gamma)$. 
\begin{defn}
Let $\Gamma \subseteq \mathbb{R}^n$ be an open cone. On 
\begin{equation*}
\Gamma_{M} := \{ (\zeta_\varepsilon)_\varepsilon \in \Gamma^{(0,1]} \ | \ \exists N \in \mathbb{N}: |\zeta_\varepsilon| = \mathcal{O}(\varepsilon^{-N}) \text{ as } \varepsilon \to 0 \}
\end{equation*}
we introduce the equivalence relation 
\begin{equation*}
(\zeta^1_\varepsilon)_\varepsilon \sim (\zeta^2_\varepsilon)_\varepsilon \Leftrightarrow \forall m \in \mathbb{N}: |\zeta^1_\varepsilon - \zeta^2_\varepsilon| = \mathcal{O}(\varepsilon^m) \text{ as } \varepsilon \to 0
\end{equation*}
and denote by $\widetilde{\Gamma} := \Gamma_{M} / \sim$ the generalized cone with respect to $\Gamma$.
\end{defn}
\begin{lem}
Let $a \in \widetilde{S}^m_{\rho,\delta,lsc}(U \times \Gamma)$ and $\widetilde{\zeta} \in \widetilde{\Gamma}$. Then the generalized point value of $a$ at $\widetilde{\zeta} = [(\zeta_\varepsilon)_\varepsilon]$, 
\begin{equation*}
a(z,\widetilde{\zeta}) := [(a_\varepsilon(z,\zeta_\varepsilon))_\varepsilon]
\end{equation*} 
is a well-defined element in $\widetilde{\mathbb{C}}$.
\end{lem}
\begin{proof}
The proof is similar to the proof of \cite[Proposition 1.2.45]{GKOS:01}. We let $(z,\zeta_\varepsilon)_\varepsilon \in U \times \Gamma_M$ and $(a_\varepsilon)_\varepsilon \in S^m_{\rho,\delta,lsc}(U \times \Gamma)$. 
For any index $i\in \mathbb{N}$ $C_i$ denotes a positive constant, $(\omega_{i,\varepsilon})_\varepsilon$ a logarithmic slow scale net and $N_i$ a natural number.
Then
\begin{equation*}
|a_\varepsilon(z,\zeta_\varepsilon)| \le C_1 \omega_{1,\varepsilon} (1+|\zeta_\varepsilon|)^m \le C_1 \omega_{1,\varepsilon} (1+ C_2 \varepsilon^{-N_2})^{\max(m,0)} \le C_3 \varepsilon^{-N_3}
\end{equation*}
so $(a_\varepsilon(z,\zeta_\varepsilon))_\varepsilon \in \mathbb{R}_M = \mathcal{M}_\mathbb{R}$.

We now show that $\widetilde{\zeta}^1 \sim \widetilde{\zeta}^2$ implies $a(z,\widetilde{\zeta}^1) \sim a(z,\widetilde{\zeta}^2)$. So let $\widetilde{\zeta}^1 \sim \widetilde{\zeta}^2$. Then
\begin{equation*}
\begin{split}
|a_\varepsilon(z,\zeta_\varepsilon^1) - a_\varepsilon(z,\zeta_\varepsilon^2)| &\le |\zeta_\varepsilon^1-\zeta_\varepsilon^2| \int_{0}^{1} |D_\zeta a_\varepsilon(z, \zeta_\varepsilon^1+\sigma(\zeta_\varepsilon^2-\zeta_\varepsilon^1))| \, d\sigma\\
&\le C_4 \omega_{4,\varepsilon} (1+ C_5 \varepsilon^{-N_5})^{\max(m,0)} |\zeta_\varepsilon^1-\zeta_\varepsilon^2| \le C_6 \varepsilon^{-N_6} \varepsilon^p
\end{split}
\end{equation*}
for all $p \ge 0$ and $\varepsilon$ small enough. Hence $(a_\varepsilon(z,\zeta_\varepsilon^1))_\varepsilon - (a_\varepsilon(z,\zeta_\varepsilon^2))_\varepsilon \sim 0$. 

Finally, if $(a_\varepsilon)_\varepsilon \in \mathcal{N}_{S^m}(U \times \Gamma)$ we have
\begin{equation*}
|a_\varepsilon(z,\zeta_\varepsilon)| \le C \varepsilon^p (1+|\zeta_\varepsilon|)^m \le C \varepsilon^p (1+ C_7 \varepsilon^{-N_7})^{\max(m,0)} 
\end{equation*}
for any $p \ge 0$ as $\varepsilon \to 0$. So $(a_\varepsilon(z,\zeta_\varepsilon))_\varepsilon \sim 0$.
\end{proof}
In the case of generalized logarithmic slow scale cones we get a similar result.
\begin{defn}
Let $\Gamma \subseteq \mathbb{R}^n$ be an open cone. On 
\begin{equation*}
\Gamma_{M,lsc} := \{ (\zeta_\varepsilon)_\varepsilon \in \Gamma^{(0,1]} \ | \ \exists (\omega_\varepsilon)_\varepsilon \in \Pi_{lsc}: |\zeta_\varepsilon| = \mathcal{O}(\omega_\varepsilon) \text{ as } \varepsilon \to 0 \}
\end{equation*}
we introduce the equivalence relation 
\begin{equation*}
(\zeta^1_\varepsilon)_\varepsilon \sim (\zeta^2_\varepsilon)_\varepsilon \Leftrightarrow \forall m \in \mathbb{N}: |\zeta^1_\varepsilon - \zeta^2_\varepsilon| = \mathcal{O}(\varepsilon^m) \text{ as } \varepsilon \to 0
\end{equation*}
and denote by $\widetilde{\Gamma}_{lsc} := \Gamma_{M,lsc} / \sim$ the generalized logarithmic slow scale cone of $\Gamma$.
\end{defn}

\begin{lem}
Let $a \in \widetilde{S}^m_{\rho,\delta,lsc}(U \times \Gamma)$ and $\widetilde{\zeta} \in \widetilde{\Gamma}_{lsc}$. Then the generalized point value of $a$ at $\widetilde{\zeta} = [(\zeta_\varepsilon)_\varepsilon]$, 
\begin{equation*}
a(z,\widetilde{\zeta}) := [(a_\varepsilon(z,\zeta_\varepsilon))_\varepsilon]
\end{equation*} 
is a well-defined element of $\widetilde{\mathbb{C}}_{lsc}$.
\end{lem}
\subsection{Asymptotic Expansion}
To prepare the factorization theorem of section~\ref{sec:Factorization} we will have to consider products of pseudodifferential operators. In the sequel we will construct a complete symbolic calculus for generalized pseudodifferential operators. 

We therefore start with some general observations concerning the notion of asymptotic expansion of a generalized symbol in $\widetilde{S}^m_{\rho,\delta,lsc}$. The presentation of the results are inspired by the results given in \cite[Section 2.5]{Garetto:08} and \cite{GarettoGramchevMoe:05}. There, the techniques are described by means of generalized symbols as well as for slow scale symbols. The modification to logarithmic slow scale symbols is evident.

The definition of the asymptotic expansion is now the following:
\begin{defn}\label{defn:AE_rep}
Let $\{m_j\}_j$ be a strictly decreasing sequence of real numbers such that $m_j \to -\infty$ as $j \to \infty$, $m_0 = m$. 
\begin{itemize}[leftmargin=.25in]
\item[(i)]
Further, let $\{ (a_{j,{\varepsilon}})_{\varepsilon} \}_{j}$ be a sequence with $( a_{j,{\varepsilon}} )_{\varepsilon} \in \mathcal{M}_{S^{m_j}_{\rho,\delta}}$. We say that the formal series $\sum_{j=0}^{\infty} (a_{j,\varepsilon})_\varepsilon$ is the asymptotic expansion for $(a_\varepsilon)_\varepsilon \in \mathcal{C}^{\infty}(\mathbb{R}^n \times \mathbb{R}^n)^{(0,1]}$, denoted by $(a_\varepsilon)_\varepsilon \sim \sum_{j} (a_{j,\varepsilon})_\varepsilon$, if and only if
\begin{equation*}
(a_\varepsilon - \sum_{j=0}^{N-1} a_{j,\varepsilon})_\varepsilon \in \mathcal{M}_{S^{m_N}_{\rho,\delta}} \qquad \forall N \geq 1.
\end{equation*}
\item[(ii)]
Further, let $\{ (a_{j,{\varepsilon}})_{\varepsilon} \}_{j}$ be a sequence with $( a_{j,{\varepsilon}} )_{\varepsilon} \in S^{m_j}_{\rho,\delta,lsc}$. We say that the formal series $\sum_{j=0}^{\infty} (a_{j,\varepsilon})_\varepsilon$ is the asymptotic expansion for $(a_\varepsilon)_\varepsilon \in \mathcal{C}^{\infty}(\mathbb{R}^n \times \mathbb{R}^n)^{(0,1]}$, denoted by $(a_\varepsilon)_\varepsilon \sim \sum_{j} (a_{j,\varepsilon})_\varepsilon$, if and only if
\begin{equation*}
(a_\varepsilon - \sum_{j=0}^{N-1} a_{j,\varepsilon})_\varepsilon \in S^{m_N}_{\rho,\delta,lsc} \qquad \forall N \geq 1.
\end{equation*}
\end{itemize}
In both cases, $(a_{0,\varepsilon})_\varepsilon$ is said to be the principal symbol of $(a_\varepsilon)_\varepsilon$.
\end{defn}
Also, we introduce the restriction $\mathcal{M}_{S^m_{\rho,\delta,cl}}(\mathbb{R}^n \times \mathbb{R}^n) \subseteq \mathcal{M}_{S^m_{\rho,\delta}}$ of classical generalized symbols. 
\begin{defn}
We say that a generalized symbol $(a_\varepsilon)_\varepsilon \in \mathcal{M}_{S^m_{\rho,\delta}}$ is classical, denoted by $(a_\varepsilon)_\varepsilon \in \mathcal{M}_{S^m_{\rho,\delta,cl}}$, if there exists a sequence $\{(a_{j,\varepsilon})_\varepsilon \}_j$ with symbols $(a_{j,\varepsilon})_\varepsilon$ in $\mathcal{M}_{S^{m-j}_{\rho,\delta}} (\mathbb{R}^n \times (\mathbb{R}^n \setminus 0))$ homogeneous of degree $m-j$ in $|\xi| \ge 1$, $j \in \mathbb{N}$, such that for any cut-off function $\varphi \in \mathcal{C}^{\infty}_0(\mathbb{R}^n)$ equal to 1 near the origin we have
\begin{equation}\label{AE_ph}
\Big(a_\varepsilon(x,\xi) - \sum_{j=0}^{N-1} (1-\varphi(\xi))a_{j,\varepsilon}(x,\xi)\Big)_\varepsilon \in \mathcal{M}_{S^{m-N}_{\rho,\delta}} \qquad \forall N \geq 1.
\end{equation}
As above, we will write $(a_\varepsilon)_\varepsilon \sim \sum_{j} (a_{j,\varepsilon})_\varepsilon$ if (\ref{AE_ph}) holds and we call $(a_{0,\varepsilon})_\varepsilon$ the principal symbol of $(a_\varepsilon)_\varepsilon$.
\end{defn}
Replacing $\mathcal{M}_{S^{m}_{\rho,\delta}}$ by $\mathcal{M}_{S^{m}_{\rho,\delta}}^{lsc}$, one obtains classical symbols of logarithmic slow scale type.

As a result, we obtain that any infinite sum of symbols of strictly decreasing orders can be summarized as follows.
\begin{thm}\label{thm:AE}
\mbox{}
\begin{itemize}[leftmargin=.25in,itemsep=5pt]
\item[(i)]
Let $( a_{j,{\varepsilon}} )_{\varepsilon} \in \mathcal{M}_{S^{m_j}_{\rho,\delta}}$ with $m_j \searrow -\infty$ as $j \to \infty$, $m_0 = m$ as in (i) of Definition~\ref{defn:AE_rep}. Then there exists $(a_\varepsilon)_\varepsilon \in \mathcal{M}_{S^{m}_{\rho,\delta}}$ such that $(a_\varepsilon)_\varepsilon \sim \sum_{j} (a_{j,\varepsilon})_\varepsilon$. Moreover, if $(a'_\varepsilon)_\varepsilon  \sim \sum_{j} (a_{j,\varepsilon})_\varepsilon$, then $(a_\varepsilon - a'_\varepsilon)_\varepsilon \in \mathcal{M}_{S^{-\infty}}$.
\item[(ii)]
Let $( a_{j,{\varepsilon}} )_{\varepsilon} \in S^{m_j}_{\rho,\delta,lsc}$ with $m_j \searrow -\infty$ as $j \to \infty$, $m_0 = m$ as in (ii) of Definition~\ref{defn:AE_rep}. Then there exists $(a_\varepsilon)_\varepsilon \in S^{m}_{\rho,\delta,lsc}$ such that $(a_\varepsilon)_\varepsilon \sim \sum_{j} (a_{j,\varepsilon})_\varepsilon$. Moreover, if $(a'_\varepsilon)_\varepsilon  \sim \sum_{j} (a_{j,\varepsilon})_\varepsilon$, then $(a_\varepsilon - a'_\varepsilon)_\varepsilon \in S^{-\infty}_{lsc}$.
\end{itemize}
\end{thm}
\begin{rem}
The proof of this theorem can be found in \cite[Theorem 2.2]{Garetto:08}. The same proof remains valid for classical symbol classes. We recall the construction of the symbol $(a_\varepsilon)_\varepsilon$ of Theorem~\ref{thm:AE}:

Let $\psi \in \mathcal{C}^{\infty}(\mathbb{R}^n)$, $0 \le \psi(\xi) \le 1$ such that $\psi(\xi)=0$ for $|\xi| \le 1$ and $\psi(\xi)=1$ for $|\xi| \ge 2$. As in \cite[Theorem 2.2]{Garetto:08} one can define
\begin{equation*}
a_\varepsilon(x,\xi) := \sum_{j \in \mathbb{N}} \psi(\lambda_{j,\varepsilon} \xi) a_{j,\varepsilon}(x,\xi)
\end{equation*}
where $\lambda_{j,\varepsilon}$ are appropriate positive constants with $\lambda_{j+1,\varepsilon} < \lambda_{j,\varepsilon} < 1$, $\lambda_{j,\varepsilon} \to 0$ as $j \to \infty$. In case (i) of Theorem~\ref{thm:AE}, $\lambda_{j,\varepsilon}$ are taken to be the inverse of an appropriate strict positive net. In Theorem~\ref{thm:AE} (ii), it turns out that $(\lambda_{j,\varepsilon})$ can be chosen to be an inverse of a logarithmic slow scale net.

We also remark that Theorem~\ref{thm:AE} can be carried over to classical symbols in a corresponding way.
\end{rem}
Noting that in Theorem~\ref{thm:AE} (i) one can replace the moderateness by negligibility and we introduce as in \cite[Definition 2.6 (ii)]{Garetto:08}:
\begin{defn}
Let $\{m_j\}_j$ be as in Definition~\ref{defn:AE_rep}. Further let $\{ a_{j} \}_{j}$ be a sequence in $\widetilde{S}^{m_j}_{\rho,\delta,lsc}$. We say that the formal series $\sum_{j=0}^{\infty} a_j$ is the asymptotic expansion of $a \in \widetilde{S}^{m}_{\rho,\delta,lsc}$, denoted by $a \sim \sum_{j} a_{j}$, if and only if there exists a representative $(a_\varepsilon)_\varepsilon$ of $a$ and, for every $j$ representatives $(a_{j,\varepsilon})_\varepsilon$ of $a_j$, such that $(a_\varepsilon)_\varepsilon \sim \sum_{j} (a_{j,\varepsilon})_\varepsilon$.
\end{defn}
Similarly, we introduce the following definition for classical symbols.
\begin{defn}\label{defn:AE_ph}
Let $\{ a_{j} \}_{j \in \mathbb{N}}$ be a sequence with $a_j \in \widetilde{S}^{m-j}_{\rho,\delta,lsc}(\mathbb{R}^n \times (\mathbb{R}^n \setminus 0))$ homogeneous of degree $m-j$. We say that the formal series $\sum_{j=0}^{\infty} a_j$ is the asymptotic expansion of $a \in \widetilde{S}^{m}_{\rho,\delta,cl,lsc}$, denoted by $a \sim \sum_{j} a_{j}$, if and only if there exists a representative $(a_\varepsilon)_\varepsilon$ of $a$ and, for every $j$ representatives $(a_{j,\varepsilon})_\varepsilon$ of $a_j$, such that $(a_\varepsilon)_\varepsilon \sim \sum_{j} (a_{j,\varepsilon})_\varepsilon$.
\end{defn}
\subsection{Composition and Adjoint of Pseudodifferential Operators}
In this subsection we briefly recall the composition law of two pseudodifferential operators and adjoint operators. A detailed discussion for slow scale regular generalized symbols can be found in \cite[Section 5]{GarettoGramchevMoe:05}. Again, the adaptation to logarithmic slow scale symbols is evident.

Therefore, for logarithmic slow scale symbols, we get the following result:
\begin{thm}\label{thm:composition}
Let $A(x,D_x)$ and $B(x,D_x)$ be two pseudodifferential operators with generalized symbols $a \in \widetilde{S}^{m_1}_{\rho,\delta,lsc}$ and $b \in \widetilde{S}^{m_2}_{\rho,\delta,lsc}$ respectively. Then the product $A B$ is well-defined and maps $\GLtwo$ into itself. Moreover $A B$ is a pseudodifferential operator with generalized symbol $a \# b$ in $\widetilde{S}^{m_1+m_2}_{\rho,\delta,lsc}$ having the representation
\begin{equation*}
a \# b (x,\xi) \sim \sum_{|\alpha| \ge 0} \frac{1}{\alpha!} D^{\alpha}_{\xi} a(x,\xi) \partial^{\alpha}_{x} b(x,\xi).
\end{equation*}
\end{thm}
For the proof we refer to \cite[Theorem 5.15]{GarettoGramchevMoe:05}. We note that Theorem~\ref{thm:composition} can also be stated for classical symbols. In detail, if $a \in \widetilde{S}^{m_1}_{\rho,\delta,cl,lsc}$ and $b \in \widetilde{S}^{m_2}_{\rho,\delta,cl,lsc}$, then $a \# b$ in $\widetilde{S}^{m_1+m_2}_{\rho,\delta,cl,lsc}$.

Also useful is the following theorem:
\begin{thm}
Let $A(x,D_x)$ be a pseudodifferential operators with generalized symbol $a \in \widetilde{S}^{m}_{\rho,\delta,lsc}$. Then the adjoint operator $A^*$ is in $\Psi^{m}_{\rho,\delta,lsc}$ and its symbol $a^*$ is given by the asymptotic expansion
\begin{equation*}
a^* (x,\xi) \sim \sum_{|\alpha| \ge 0} \frac{i^{|\alpha|}}{\alpha!} D^{\alpha}_{x} D^{\alpha}_{\xi} \overline{a(x,\xi)}.
\end{equation*}
\end{thm}
For the proof see \cite[Theorem 5.12]{GarettoGramchevMoe:05}.
\section{The governing equation in the Colombeau setting}\label{sec:GovEqn}

The present survey is devoted to wave propagation in inhomogeneous acoustic media in the case of non-smooth background data. The operator structure is motivated from wave propagation phenomena occuring in a number of physical applications, such as underwater acoustic or seismology. Hence, the evolution direction is the depth coordinate.  
\subsection{Derivation of the inhomogeneous wave equation}

The acoustic wave equation characterizes fluid motions and can be derived from the conservation laws of mass and momentum together with the equation of state of thermodynamical equilibrium. This system of acoustic equations can be linearized and therefore describe small perturbations from  a state of rest of the pressure $U(t,x) \in \mathbb{R}$, the density $\rho'(t,x) \in \mathbb{R}$ and the velocity $\vec{v'}(t,x) \in \mathbb{R}^n$ that moves in a fluid with given wave speed $c(x)$ and density $\rho(x)$, ($x \in \mathbb{R}^n, t \in \mathbb{R}$). To explain how these waves are generated or added to the fluid one can introduce so-called sources by adding source terms in the equation of mass, momentum and energy. Assuming an isentropic process the linearized acoustic system can be written as (cf. \cite{Howe:98, KinslerFrey:99, SO:09, BrekhovskikhGodin:99})
\begin{align*}
\frac{\partial \rho'}{\partial t} &= \text{div}(\rho \vec{v'}) + m & (\text{mass conservation})&\\
\frac{\partial \vec{v'}}{\partial t} &= \frac{1}{\rho}\nabla U + \vec{f} & (\text{momentum conservation})&\\
\frac{\partial U}{\partial t} &= c^2 \Bigl( \frac{\partial \rho'}{\partial t} + \vec{v'} \nabla{\rho} \Bigr) & (\text{equation of state})&
\end{align*}
where the mass source term $m$ and the force source term $\vec{f}$ are supposed to vanish in the undisturbed state. Here $m$ represents a volume injection of a source such as, for example, bodies whose volume is oscillating. An example for a body force $\vec{f}$ of a source is an oscillating rigid body of constant volume. Note that these equations also hold for the fluid at rest.

Substituting the state equation into the equation of mass and using the conservation of momentum to eliminate the velocity term gives
\begin{equation*}
\text{div}( \frac{1}{\rho} \nabla U) - \frac{1}{c^2 \rho} \frac{\partial^2}{\partial t^2} U = F
\end{equation*}
where $F:= - \frac{\partial m/\rho}{\partial t} + \text{div}(\vec{f}/\rho)$ denotes the source function. We remark that in the absence of sources one derives the homogeneous wave equation.

In the following we will study the Colombeau generalized partial differential equation of the form
\begin{equation}\label{eqn:GovEqn}
LU := \big( \partial_z \frac{1}{\rho} \partial_z + \sum_{j=1}^{n-1} \partial_{x_j} \frac{1}{\rho} \partial_{x_j} - \frac{1}{\rho} \frac{1}{c^2} \partial^2_t \big) U = F
\end{equation}
with the (pressure) wave field $U \in \GLtwo(\mathbb{R}^{n+1})$ and $F \in \GLtwo(\mathbb{R}^{n+1})$ a source term. For the space coordinates we will allocate the vertical direction $z$, which we call the depth, the lateral directions are denoted by $x$. 

We assume that the coefficients $\frac{1}{\rho}=\frac{1}{\rho}(x,z)$, $\frac{1}{c}=\frac{1}{c}(x,z)$ are Colombeau generalized functions in $\G^{lsc}_{\infty,\infty} := \ensuremath{\mathcal{M}}_{W^{\infty,\infty}}^{lsc} / \ensuremath{\mathcal{N}}_{W^{\infty,\infty}}^{}$ and meet the following requirements: there are representatives $\Big(\frac{1}{\rho_\varepsilon}\Big)_\varepsilon \in \frac{1}{\rho}$, $\Big(\frac{1}{c^2_\varepsilon}\Big)_\varepsilon \in \frac{1}{c^2}$ such that
\begin{itemize}[leftmargin=.35in,topsep=-5pt, itemsep=5pt]
\setlength\itemsep{1em}
\item[(i)] there exist H\"older continuous functions $\frac{1}{\rho^*}, \frac{1}{c^*} \in \mathcal{C}^{0,\mu}(\mathbb{R}^n)$ for some $\mu \in (0,1)$ such that for all $(x,z) \in \mathbb{R}^n$ we have $\frac{1}{\rho_\varepsilon} = \frac{1}{\rho^*} \ast \varphi_{\omega_\varepsilon^{-1}}, \frac{1}{c_\varepsilon} = \frac{1}{c^*} \ast \varphi_{\omega_\varepsilon^{-1}}$ where $\varphi$ is as in (\ref{mollifier}), $(\omega_\varepsilon)_\varepsilon \in \Pi_{lsc}$ and the convolution is taken with respect to $x$ and $z$
\item[(ii)] $\exists \eta \in (0,1] \ \exists$ constants $c_0,\ c_1, \ \rho_0$ and $\rho_1$ such that $0 < c_0 \le c_\varepsilon(x,z) \le c_1 < \infty$ and $0 < \rho_0 \le \rho_\varepsilon(x,z) \le \rho_1 < \infty$ for all $(x,z) \in \mathbb{R}^n$ and $\varepsilon \in (0,\eta]$.
\end{itemize}
The lower bound assumption is sometimes referred to as strong positivity (cf. \cite[Section 1]{GarettoHoermann:05} or \cite[Section 2]{Garetto:06}). The upper bound means uniform boundedness of the $0$-th derivative of the representatives. For more details on the assumption (i) we refer to the remark given below.

Given a regularized operator as above, we carry out all transformations within algebras of generalized functions from now on. More explicitly, we will study the action of the linear operator $L$ from $\GLtwo$ into itself in the following sense: on the level of representatives $L$ acts as
\begin{equation*}
(u_\varepsilon)_\varepsilon \mapsto \Big( \big( \partial_z \frac{1}{\rho_\varepsilon} \partial_z + \sum_{j=1}^{n-1} \partial_{x_j} \frac{1}{\rho_\varepsilon} \partial_{x_j} - \frac{1}{\rho_\varepsilon} \frac{1}{c^2_\varepsilon} \partial^2_t \big) u_\varepsilon \Big)_\varepsilon \qquad \forall (u_\varepsilon)_\varepsilon \in \mathcal{M}_{H^{\infty}}.
\end{equation*}
This explains our governing equation in (\ref{eqn:GovEqn}) where $U=[(u_\varepsilon)_\varepsilon]$.  

\begin{rem}
In order to realize the logarithmic slow scale type conditions on the coefficients of a partial differential operator, one usually uses a rescaling in the mollification. 
In our case, the regularization is obtained by convolution with the logarithmically scaled mollifier $\varphi_{\omega_\varepsilon^{-1}}(.) := \omega_\varepsilon^n \varphi(\omega_\varepsilon .)$ with $\varphi \in \mathcal{S}(\mathbb{R}^n)$ as in (\ref{mollifier}) and $(\omega_\varepsilon)_\varepsilon \in \Pi_{lsc}$. 

For completeness, we give the following implications for our type of coefficients:
Let $u \in \mathcal{C}^{0,\mu}(\mathbb{R}^n)$ be a real-valued H\"older continuous function with exponent $\mu \in (0,1)$ and such that $\inf(u) \ge c$ for some positive constant $c$. Further $\varphi$ is a mollifier as above and $(\omega_\varepsilon)_\varepsilon \in \Pi_{lsc}$. 
One may think of $\omega_\varepsilon = \log ( \log (\frac{1}{\varepsilon} ) )$ as an explicit example. Then $\exists c>0: \omega_\varepsilon \ge c$ for all $\varepsilon \in (0,\frac{1}{11}]$ and $\forall p \ge 0, \exists c_p > 0: \omega_\varepsilon^p \le c_p \log(\frac{1}{\varepsilon})$ for all $\varepsilon \in (0,1]$. So given some $(\omega_\varepsilon)_\varepsilon \in \Pi_{lsc}$ the logarithmic slow scale regularization of $u$ is given by $u_\varepsilon (x) = u \ast \varphi_{\omega_\varepsilon^{-1}} (x)$ and satisfies the following estimates (see \cite[Theorem 7]{GH:04a}, or \cite[Section 5]{Pilipovic:13}):
\begin{equation*}
\lVert \partial^{\alpha} u_\varepsilon \rVert_{L^{\infty}} = \begin{cases} \mathcal{O}(1) & |\alpha|=0 \\ \mathcal{O}\big(\omega_\varepsilon^{|\alpha|-\mu}\big) & |\alpha|>0 \end{cases} \qquad  (\varepsilon \to 0)
\end{equation*}
and 
\begin{equation*}
\exists u_0>0 \ \exists \eta \in (0,1]: \quad | u_\varepsilon(x) | \ge u_0 \qquad x \in \mathbb{R}^n, \ \varepsilon \in (0,\eta].
\end{equation*}
Here, the latter inequality is the strong positivity and the first estimate guarantees that the net $(u_\varepsilon)_\varepsilon$ is in $\ensuremath{\mathcal{M}}_{W^{\infty,\infty}}^{lsc}$. Moreover, note that $u \in \mathcal{C}^{0,\mu}$ implies that for every $x \in \mathbb{R}^n$ we have
\begin{equation*}
|u_\varepsilon(x) - u(x)| \le \int |u(x-\omega_\varepsilon^{-1} y)-u(x)| \ |\varphi(y)|  \,d y  = \mathcal{O} (\omega_\varepsilon^{-\mu}) \qquad \varepsilon \to 0.
\end{equation*}
Again, with a slight abuse of notation we write $u$ for the equivalence class of the Colombeau regularization of $u \in \mathcal{C}^{0,\mu}$, i.e: 
\begin{equation*}
u:= (u_\varepsilon)_\varepsilon + \ensuremath{\mathcal{N}}_{W^{\infty,\infty}}^{} \in \ensuremath{\mathcal{G}}_{\infty,\infty}^{lsc}.
\end{equation*}
\end{rem}
\subsection{Time extrapolation of the wave field}\label{subsec:TimeMigration}

In a paper of Garetto and Oberguggenberger, see \cite{GarMoe:2011a}, the authors studied well-posedness of Cauchy problems with respect to the time variable for strictly hyperbolic systems and higher order equations in the Colombeau setting using symmetrisers and the theory of generalized pseudodifferential operators. There they imposed conditions concerning the asymptotic scale of the regularization parameter of the operator. Concretely, the authors proved existence, uniqueness and regularity of generalized solutions in the case that the regularization parameter is chosen logarithmic slow scale. More details can be found in \cite[Theorem 4.2]{GarMoe:2011a}.
\subsection{Depth evolution processes}

The concept of one-way wave equations, also known as paraxial wave equations, was first introduced by \cite{Cl:85} and has become a standard tool in depth migration processes due to ill-posedness of the full wave equation. In fact, they are expressions for the first depth derivative of a wave field and thus of the form
\begin{equation}\label{op:one-way}
\partial_z + iB_{\pm}(x,z,D_t,D_x).
\end{equation}
where $B_{\pm}$ are (microlocal) pseudodifferential operators.
Our derivation of one-way wave equations starts with a factorization of the operator $L$ in (\ref{eqn:GovEqn}) into terms of the form (\ref{op:one-way}).
\begin{rem}
As already noted in subsection \ref{subsec:TimeMigration}, one obtains well-posed problems in the model of time migration for strictly hyperbolic operators. We therefore give the following link.

Note that the operator $L$ in (\ref{eqn:GovEqn}) is strictly hyperbolic in the following sense:
Let $\theta \in (0,\pi/2)$ be a fixed angle and
\begin{align*}
\begin{split}
I'_{\theta} &:= \Bigl\{ (x,z,\tau,\xi) \in \mathbb{R}^n \times \mathbb{R}^n \ | \ \tau \neq 0, \ |c^*(x,z) \tau^{-1} \xi | < \sin \theta \Bigr\}.
\end{split}
\end{align*}
\end{rem}
For completeness we recall the following definition.
\begin{defn}
The operator $L$ in (\ref{eqn:GovEqn}) is called generalized strictly hyperbolic in $I'_{\theta}$ if there exists a choice of representatives of the coefficients in $\ensuremath{\mathcal{G}}_{\infty,\infty}^{lsc}(\mathbb{R}^n)$ such that the corresponding principal symbol of $L_\varepsilon(\zeta; x,z,\tau,\xi)$ has $2$ distinct real-valued roots $(\zeta_{j,\varepsilon})_\varepsilon$, $j=1,2$ such that
\begin{equation}\label{eqn:CharRoots}
|\zeta_{1,\varepsilon}(x,z,\tau,\xi) - \zeta_{2,\varepsilon}(x,z,\tau,\xi)| \ge \zeta_\varepsilon \langle (\tau, \xi) \rangle
\end{equation}
holds for some strictly nonzero net $(\zeta_\varepsilon)_\varepsilon$, for all $(x,z,\tau,\xi) \in I'_{\theta}$ and $\varepsilon$ sufficiently small.
\end{defn}
\section{Ellipticity of pseudodifferential operators}\label{sec:ell}
In this section, we introduce the notion of elliptic operators which enables us to construct a parametrix for such operators. Furthermore such operators will play an important role in the characterization of the generalized wave front set (cf. section~\ref{sec:lscWaveFrontSet}).
\subsection{Ellipticity}

In \cite[Definition 1.2]{GarettoHoermann:05} and \cite[Proposition 2.7]{Garetto:08} the authors introduced the concept of slow scale ellipticity which can be carried over to logarithmic slow scale ellipticity easily. With this in mind, we define the property of logarithmic slow scale ellipticity for our class of operators as follows.
\begin{defn}\label{defn:lsc-ellipticity}
We say that a generalized symbol $a \in \widetilde{S}^{m}_{lsc}$ is logarithmic slow scale micro-elliptic at $(x_0,\xi_0) \in T^*\mathbb{R}^n \setminus 0$ if it has a representative $(a_\varepsilon)_\varepsilon$ such that the following is satisfied: $\exists$ relatively compact open neighborhood $U$ of $x_0$, $\exists$ conic neighborhood $\Gamma$ of $\xi_0$, $\exists$ nets $(r_\varepsilon)_\varepsilon \in \Pi_{lsc}$,  $(s_\varepsilon)_\varepsilon \in \Pi_{lsc}$ and a constant $\eta \in (0,1]$ such that
\begin{equation}\label{eqn:lsc-ell}
|a_\varepsilon(x,\xi)| \geq \frac{1}{s_\varepsilon} \langle \xi \rangle^m \quad (x,\xi) \in U \times \Gamma, \ |\xi| \geq r_\varepsilon, \ \varepsilon \in (0,\eta].
\end{equation}
In the sequel, we will sometimes use the abbreviation lsc for logarithmic slow scale.
We denote by $\text{Ell}_{lsc}(a)$ the set of points $(x_0,\xi_0) \in T^* \mathbb{R}^n \setminus 0$ where $a$ is logarithmic slow scale micro-elliptic. 
If there exists a representative $(a_\varepsilon)_\varepsilon \in a$ such that (\ref{eqn:lsc-ell}) holds for any $(x_0,\xi_0) \in T^* \mathbb{R}^n \setminus 0$ then the symbol $a$ is called logarithmic slow scale elliptic. 
\end{defn}
\begin{defn}
More generally, we say that a generalized symbol $a \in \widetilde{S}^{m}_{sc}$ is slow scale micro-elliptic at $(x_0,\xi_0) \in T^*\mathbb{R}^n \setminus 0$ if (\ref{eqn:lsc-ell}) holds for some nets $(r_\varepsilon)_\varepsilon, (s_\varepsilon)_\varepsilon \in \Pi_{sc}$.
\end{defn}
Note that a logarithmic slow scale elliptic symbol is also slow scale elliptic but not vice versa.
\begin{rem}
We first observe that condition $(\ref{eqn:lsc-ell})$ is independent of the choice of representative. Indeed, let $(a_\varepsilon)_\varepsilon \in a$ as in Definition~\ref{defn:lsc-ellipticity} and $(a'_\varepsilon)_\varepsilon$ another representative of $a$. Then for some arbitrary but fixed $p>0$ there $\exists c_1 > 0, \ \exists \eta_1 \in (0,1]$ such that on $U \times \Gamma$ we have
\begin{equation*}
|a'_\varepsilon(x,\xi)| \ge |a_\varepsilon(x,\xi)| - |(a'_\varepsilon-a_\varepsilon)(x,\xi)| \ge \frac{1}{s_\varepsilon} \langle \xi \rangle^m \big( 1- c_1 s_\varepsilon \varepsilon^{2p} \big) \qquad |\xi| \ge r_\varepsilon, \ \varepsilon \in (0,\eta_1].
\end{equation*}
Since $\exists \eta_2 \in (0,1], \ \exists c_2 >0$ so that $s_\varepsilon \le c_2 \varepsilon^{-p}$ for $\varepsilon \in (0,\eta_2]$ we get
\begin{equation*}
1-c_1 s_\varepsilon \varepsilon^{2p} \ge 1- c_1 c_2 \varepsilon^p \ge 1/2 \qquad \forall \varepsilon \in (0, \min(\eta_2, \big(2c_1 c_2\big)^{-1/p}) ].
\end{equation*}
Redefining $\eta:= \min(\eta_1, \eta_2, \big(2c_1 c_2\big)^{-1/p})$ we obtain
\begin{equation*}
|a'_\varepsilon(x,\xi)| \geq \frac{1}{2s_\varepsilon} \langle \xi \rangle^m \quad (x,\xi) \in U \times \Gamma, \ |\xi| \geq r_\varepsilon, \ \varepsilon \in (0,\eta].
\end{equation*}
\end{rem}

Moreover, we notice that the notion of (logarithmic) slow scale ellipticity is stable under lower order (logarithmic) slow scale perturbations. For a proof one reworks essentially the lines in \cite[Proposition 1.3]{GarettoHoermann:05}. For completeness we give the proof of the following proposition.
\begin{prop}\label{prop:lsc-ell}
Let $(a_\varepsilon)_\varepsilon \in S^m_{lsc}(\mathbb{R}^n \times \mathbb{R}^n)$ be logarithmic slow scale micro-elliptic at the point $(x_0,\xi_0)$. Then
\begin{itemize}[leftmargin=.35in, topsep=-5pt]
\setlength\itemsep{1em}
\item[(i)]
$
\!
\begin{aligned}[t]
&\forall \alpha, \beta \in \mathbb{N}^n \ \exists (\lambda_\varepsilon)_\varepsilon \in \Pi_{lsc} \ \exists \eta \in (0,1]: \\
&\phantom{\forall \alpha, \beta \in } |\partial_{\xi}^{\alpha} \partial_x^{\beta} a_\varepsilon(x,\xi)| \le \lambda_\varepsilon |a_\varepsilon(x,\xi)| \langle \xi \rangle^{-|\alpha|} \quad U \times \Gamma, \ |\xi| \ge r_\varepsilon, \ \varepsilon \in (0,\eta]
\end{aligned}
$ 
\item[(ii)] if $(b_\varepsilon)_\varepsilon \in S^{m'}_{lsc}(\mathbb{R}^n \times \mathbb{R}^n)$ with $m' < m$, then (\ref{eqn:lsc-ell}) is valid for the net $(a_\varepsilon+b_\varepsilon)_\varepsilon$.
\end{itemize}
\end{prop}
\begin{proof}
We first show (i). Since $(a_\varepsilon)_\varepsilon \in S^m_{lsc}$ is logarithmic slow scale elliptic we obtain
\begin{equation*}
\begin{split}
\forall \alpha, \beta \in \mathbb{N}^n \ \exists (\omega_\varepsilon)_\varepsilon \in \Pi_{lsc} & \ \exists c > 0 \ \exists \eta \in (0,1] \ \text{such that}\\
&|\partial_{\xi}^{\alpha} \partial_x^{\beta} a_\varepsilon(x,\xi)| \le c \omega_\varepsilon \langle \xi \rangle^{m-|\alpha|} \le c \omega_\varepsilon s_\varepsilon \langle \xi \rangle^{-|\alpha|} |a_\varepsilon(x,\xi)|
\end{split}
\end{equation*}
on $U \times \Gamma$ for $|\xi| \ge r_\varepsilon, \ \varepsilon \in (0,\eta]$. Setting $\lambda_\varepsilon := \omega_\varepsilon s_\varepsilon$ then $(\lambda_\varepsilon)_\varepsilon \in \Pi_{lsc}$ as desired.

To show the second part let $(b_\varepsilon)_\varepsilon \in S^{m'}_{lsc}$. Then
\begin{equation*}
|a_\varepsilon(x,\xi) + b_\varepsilon(x,\xi)| \ge \frac{1}{s_\varepsilon} \langle \xi \rangle^m - c\omega_\varepsilon\langle \xi \rangle^{m'} = \frac{1}{s_\varepsilon} \langle \xi \rangle^m \Big( 1- c s_\varepsilon \omega_\varepsilon \langle \xi \rangle^{m'-m} \Big) \ge \frac{1}{2s_\varepsilon} \langle \xi \rangle^m
\end{equation*}
on $U \times \Gamma$, $|\xi| \ge r'_\varepsilon := \max(r_\varepsilon, (2c s_\varepsilon \omega_\varepsilon)^{1/(m-m')})$ and $(r'_\varepsilon)_\varepsilon \in \Pi_{lsc}$.
\end{proof}
We note that the previous proposition can also be carried over for symbols that are slow scale micro-ellipticity.
\begin{rem}
In case of smooth symbols the notion of lsc-ellipticity reduces to the classical one and is therefore equivalent to the complement of the characteristic set. Indeed, given $a \in S^m(\mathbb{R}^n \times \mathbb{R}^n)$ and a representative $(a_\varepsilon)_\varepsilon$ of an lsc-elliptic symbol $a \in S^m_{lsc}(\mathbb{R}^n \times \mathbb{R}^n)$ then there exist nets $(s_\varepsilon)_\varepsilon$, $(r_\varepsilon)_\varepsilon \in \Pi_{lsc}$ and constants $c>0$, $\eta \in (0,1]$ such that for every $q \in \mathbb{N}$ we have:
\begin{equation*}
|a(x,\xi)| \geq |a_\varepsilon(x,\xi)| - |a(x,\xi)-a_\varepsilon(x,\xi)| \geq \langle \xi \rangle^m  \Big(\frac{1}{s_{\varepsilon}} - c \varepsilon^q \Big) \qquad \text{for} \ |\xi| \ge r_\varepsilon, \ \varepsilon \in (0,\eta].
\end{equation*}
As we are allowed to fix an $\varepsilon$ small enough the last expression is bounded away from 0. In particular $\mbox{Ell}_{lsc}(a)^c$ is equal to the characteristic set $\text{Char}(a(x,D))$. Note that this result remains valid if one replaces lsc by sc (cf. \cite[Remark 1.4]{GarettoHoermann:05}).
\end{rem}
Before we proceed, we give an equivalent characterization of the notion of lsc-ellipticity concerning the principal symbol. In \cite[Proposition 3.3]{HO:04} it is proved that a generalized partial differential operator with regular coefficients in $\Greg = \Greg_{C^\infty}$ is W-elliptic (weak elliptic) if its principal symbol is S-elliptic (strong elliptic). For the precise meaning of weak and strong ellipticity consult \cite{HO:04}. Referring to this we want to give the following proposition:
\begin{prop}\label{prop:ps-ell}
Given a generalized symbol $(a_\varepsilon)_\varepsilon \in S^m_{cl,lsc}$ then the following properties are equivalent:
\begin{itemize}[topsep=-5pt, itemsep=5pt]
\item[(i)] $(a_\varepsilon)_\varepsilon$ is lsc-elliptic

\item[(ii)] the principal symbol $(a_{m,\varepsilon})_\varepsilon$ satisfies the following condition: $\exists (s_\varepsilon)_\varepsilon \in \Pi_{lsc}$, $\exists \eta \in (0,1]$ and $r > 0$ such that
\begin{equation*}
 |a_{m,\varepsilon}(x,\xi)| \geq \frac{1}{s_\varepsilon} \langle \xi \rangle^m \qquad \forall (x,\xi) \in T^*\mathbb{R}^n \ \text{with} \ |\xi| \geq r \ \text{and} \ \varepsilon \in (0,\eta].
\end{equation*}
\end{itemize}
\end{prop}
\begin{proof}
The proof follows similar arguments to those used in \cite[Proposition 3.3]{HO:04}.
We first assume that condition (ii) holds. Then there exist nets $(s_\varepsilon)_\varepsilon \in \Pi_{lsc}, (\omega_\varepsilon)_\varepsilon \in \Pi_{lsc}$ and constants $c, r > 0$ such that
\begin{equation*}
\begin{split}
|a_\varepsilon(x,\xi)| &= |a_{m,\varepsilon}(x,\xi) + \sum_{j\ge 1} a_{m-j,\varepsilon}(x,\xi)| \ge \frac{1}{s_\varepsilon}\langle \xi \rangle^m - c \omega_\varepsilon\langle \xi \rangle^{m-1} \\
& \geq \frac{1}{s_\varepsilon} \langle \xi \rangle^m \Big( 1- c \omega_\varepsilon s_\varepsilon \langle \xi \rangle^{-1} \Big)
\end{split}
\end{equation*}
for $|\xi| \ge r$ and $\varepsilon$ small enough. Setting $r_\varepsilon = \max(r,2c\omega_\varepsilon s_\varepsilon)$ then $(r_\varepsilon)_\varepsilon \in \Pi_{lsc}$ and we obtain
\begin{equation*}
|a_\varepsilon(x,\xi)| \ge \frac{1}{2s_\varepsilon} \langle \xi \rangle^m \qquad |\xi| \ge r_\varepsilon, \ \varepsilon \in (0,\eta]
\end{equation*}
for some $\eta \in (0,1]$, showing (i).

To prove that (i) implies (ii), suppose that $(a_\varepsilon)_\varepsilon$ satisfies (\ref{eqn:lsc-ell}) on $\mathbb{R}^n \times \mathbb{R}^n$. Let $\zeta \in \mathbb{R}^n, |\zeta|=1$ and choose $\xi \in \mathbb{R}^n$ such that $|\xi| \ge r_\varepsilon$ and $\zeta = \xi/|\xi|$ with $(r_\varepsilon)_\varepsilon \in \Pi_{lsc}$ as in (\ref{eqn:lsc-ell}). Then $\exists (s_\varepsilon')_\varepsilon \in \Pi_{lsc}, \ \exists (r'_\varepsilon)_\varepsilon \in \Pi_{lsc}, \ \exists \eta \in (0,1]$ so that
\begin{equation*}
\begin{split}
|\sum_{j \ge 0} a_{m-j,\varepsilon}(x,\zeta) |\xi|^{-j} | &= |a_\varepsilon(x,\xi)| |\xi|^{-m} \ge \frac{1}{s_\varepsilon} \langle \xi \rangle^{m} \frac{1}{|\xi|^m} = \frac{1}{s_\varepsilon} \langle |\xi|^{-1} \rangle^m \geq \\
&\geq \frac{\min(2^{m/2},1)}{s_\varepsilon} =: \frac{1}{s_\varepsilon'}
\end{split}
\end{equation*}
for $|\xi| \ge r'_\varepsilon := \max(1,r_\varepsilon), \ \varepsilon \in (0,\eta]$.

Now fix $N \in \mathbb{N}, \ N \ge 1$. Since $(a_\varepsilon)_\varepsilon \in S^m_{lsc}$ we obtain $ \forall (t_\varepsilon)_\varepsilon \in \Pi_{lsc}, \ \forall 1 \le j \le N, \ \exists (\omega_{j,\varepsilon})_\varepsilon \in \Pi_{lsc}, \ \exists \eta \in (0,1]$ that 
\begin{equation*}
|a_{m-j,\varepsilon}(x,\zeta)| \frac{1}{|\xi|^j} \le \omega_{j,\varepsilon} \frac{1}{|\xi|^j} \le \frac{1}{t_\varepsilon}  \qquad \text{for} \ |\xi| \ge \max_{1 \le j \le N} (r_\varepsilon,(t_\varepsilon \omega_{j,\varepsilon})^{1/j})
\end{equation*} 
and $\varepsilon \in (0,\eta]$. At this point we are free to choose $(t_\varepsilon)_\varepsilon \in \Pi_{lsc}$ such that $\frac{N}{t_\varepsilon} = \frac{1}{4s'_\varepsilon}$. 

On the other hand, we can use Theorem~\ref{thm:AE} for classical symbols to show that $\exists (\omega_\varepsilon)_\varepsilon \in \Pi_{lsc}, \exists \eta \in (0,1]$ such that
\begin{equation*}
\begin{split}
|\sum_{j \ge N+1} a_{m-j,\varepsilon}(x,\zeta) |\xi|^{-j} | &\le \frac{1}{|\xi|^{N+1}} |\sum_{j \ge N+1} a_{m-j,\varepsilon}(x,\zeta) | \le  \frac{1}{|\xi|^{N+1}} \omega_\varepsilon \langle \zeta \rangle^{m-N-1} \le\\
&\le \frac{2^{m-N-1} \omega_\varepsilon}{|\xi|^{N+1}} \le \frac{1}{4s'_\varepsilon}
\end{split}
\end{equation*}
for any $|\xi| \ge \max(1, r_\varepsilon, (2^{m-N-1} \omega_\varepsilon s'_\varepsilon)^{\frac{1}{N+1}})$ and $\varepsilon \in (0,\eta]$. At this point we reset $r'_\varepsilon:= \max_{1 \le j \le N}(1, r_\varepsilon, (2^{m-N-1} \omega_\varepsilon s'_\varepsilon)^{\frac{1}{N+1}}, (t_\varepsilon \omega_{j,\varepsilon})^{1/j})$. Then $(r'_\varepsilon)_\varepsilon \in \Pi_{lsc}$ and we obtain: $\exists \eta \in (0,1]$ such that
\begin{equation*}
\begin{split}
|a_{m,\varepsilon}(x,\zeta)| & \ge  |\sum_{j = 0}^{\infty} a_{m-j,\varepsilon}(x,\zeta)|\xi|^{-j}| - |\sum_{j = 1}^{\infty} a_{m-j,\varepsilon}(x,\zeta) |\xi|^{-j}| \ge \\
& \ge \frac{1}{s_\varepsilon'} - |\sum_{j = 1}^{N} a_{m-j,\varepsilon}(x,\zeta)|\xi|^{-j} + \sum_{j = N+1}^{\infty} a_{m-j,\varepsilon}(x,\zeta)|\xi|^{-j}| \ge \frac{1}{s'_\varepsilon}-\frac{1}{2s'_\varepsilon} = \frac{1}{2s'_\varepsilon}
\end{split}
\end{equation*}
for $|\xi| \ge r'_\varepsilon$, $\varepsilon \in (0,\eta]$ by the above. 
Thus, since $a_{m,\varepsilon}$ is homogeneous of degree $m$ in the second variable 
\begin{equation*}
|a_{m,\varepsilon}(x,\xi)| \ge \frac{1}{2s'_\varepsilon} |\xi|^m \ge \frac{1}{2s'_\varepsilon} \langle \xi \rangle^m \qquad \text{for all} \ \xi \in \mathbb{R}^n, \ |\xi| \ge 1
\end{equation*}
for $\varepsilon$ sufficiently small. 
\end{proof}
\begin{rem}\label{rem:ps-ell}
Again, Proposition~\ref{prop:ps-ell} remains valid if we replace logarithmic slow scale by slow scale.
\end{rem}
\subsection{Strong ellipticity}
As already announced in section~\ref{sec:GovEqn} we are interested in the depth extrapolation of the wave field $U$ in (\ref{eqn:GovEqn}). We will therefore introduce a notion of (logarithmic) slow scale ellipticity which in addition has a certain behavior with respect to $\zeta$, the dual variable of the depth. 

We first introduce some notation. We denote by $(\tau,\xi,\zeta) \in \mathbb{R} \times \mathbb{R}^{n-1} \times \mathbb{R}$ the dual variables corresponding to $(t,x,z) \in \mathbb{R} \times \mathbb{R}^{n-1} \times \mathbb{R}$. Thus, we will work on a $2(n+1)$-dimensional phase space. We already mentioned in subsection~\ref{subsec:GenPointValues} that given a symbol $a \in \widetilde{S}^m_{lsc}(\mathbb{R}^{n+1} \times \mathbb{R}^{n+1})$ and a generalized point $(t,x,z,\tau,\xi, \widetilde{\zeta})$ in $\mathbb{R}^{n+1} \times \mathbb{R}^n \times \widetilde{\mathbb{R}}$ then the generalized point value of the symbol $a$ at $(t,x,z,\tau,\xi, \widetilde{\zeta})$ is well-defined in $\widetilde{\mathbb{C}}$. 

Moreover let $\Gamma_s \subseteq \mathbb{R}^{n+1}$ be a (classical) conic neighborhood of $(\tau_0,\xi_0,\zeta_0)$ of the following form
\begin{equation*}
\Gamma_s = \{ \lambda (\tau,\xi;\zeta) \in \mathbb{R}^{n+1} \setminus 0 \ | \ (\tau,\xi;\zeta) \in B_s\Bigl( \frac{(\tau_0,\xi_0;\zeta_0)}{|(\tau_0,\xi_0;\zeta_0)|}\Bigr) \cap S^1(0), \ \lambda > 0 \}
\end{equation*}
for some $s>0$ small enough, where $B_s(\eta)$ denotes the open ball with radius $s$ around $\eta \in \mathbb{R}^{n+1}$ and $S^1(0) := \{ \eta \in \mathbb{R}^{n+1} : |\eta| = 1 \}$.

We set
\begin{equation*}
\Gamma_{s,M}:= \{ (\lambda (\tau,\xi;\zeta_\varepsilon))_\varepsilon \in \Gamma_s^{(0,1]} \ | \ \exists N \in \mathbb{N} : |\zeta_\varepsilon| = \mathcal{O} (\varepsilon^{-N} ) \text{ as } \varepsilon \to 0, \ \lambda > 0 \}
\end{equation*}
and introduce an equivalence relation on $\Gamma_s$ by
\begin{equation*}
(\zeta^1_\varepsilon)_\varepsilon \sim (\zeta^2_\varepsilon)_\varepsilon \Leftrightarrow \forall m > 0 : |\zeta^1_\varepsilon - \zeta^2_\varepsilon| = \mathcal{O} (\varepsilon^m) \quad \varepsilon \to 0.
\end{equation*}
We then denote by $\widetilde{\Gamma}_s:= \Gamma_{s,M} / \sim$ the generalized cone with respect to $\zeta$ generated from $\Gamma_s$. So $\widetilde{\Gamma}_s$ is a generalized conic neighborhood of $(\tau_0,\xi_0,\zeta_0)$.

Moreover  we recall that $\Gamma_s \hookrightarrow \widetilde{\Gamma}_s$. The canonical embedding of $\Gamma_s$ into $\widetilde{\Gamma}_s$ is given by $(\tau,\xi,\zeta) \mapsto (\tau,\xi,(\zeta)_\varepsilon + \mathcal{N}_\mathbb{R})$. 

Also note that in the definition of $\Gamma_{s,M}$, $(\tau,\xi;\zeta_\varepsilon)_\varepsilon \in \Gamma_s^{(0,1]}$ means that 
\begin{equation*}
(\tau,\xi;\zeta_\varepsilon) \in B_s\Bigl( \frac{(\tau_0,\xi_0;\zeta_0)}{|(\tau_0,\xi_0;\zeta_0)|}\Bigr) \cap S^1(0) \qquad \text{for every fixed } \varepsilon \in (0,1].
\end{equation*}
With this notation we now give the definition of strong lsc-micro-ellipticity.
\begin{defn}\label{defn:*-lsc-ellipticity}
Let $a$ be a generalized symbol in $\widetilde{S}^{m}_{lsc}(\mathbb{R}^{n+1} \times \mathbb{R}^{n+1})$. We say that the symbol $a$ is strong logarithmic slow scale micro-elliptic at a point $(t_0,x_0,z_0,\tau_0,\xi_0;\zeta_0) \in T^*\mathbb{R}^{n+1} \setminus 0$, and we write $(t_0,x_0,z_0,\tau_0,\xi_0;\zeta_0) \in \text{Ell}^*_{lsc}(a)$, if and only if there exist a relatively compact open neighborhood $U$ of $(t_0,x_0,z_0)$, a (classical) conic neighborhood $\Gamma$ of $(\tau_0,\xi_0,\zeta_{0})$ such that $\Gamma \hookrightarrow \widetilde{\Gamma}$ where $\widetilde{\Gamma} = \Gamma_M / \sim$ is the generalized conic neighborhood of $(\tau_0,\xi_0,\zeta_{0})$ generated from $\Gamma$ as above, nets $(r_\varepsilon)_\varepsilon , (s_\varepsilon)_\varepsilon \in \Pi_{lsc}$ and an $\eta \in (0,1]$ such that
\begin{equation}\label{eqn:star-lsc-ell}
|a_\varepsilon(t,x,z,\tau,\xi;\zeta_\varepsilon)| \geq \frac{1}{s_\varepsilon} \langle (\tau,\xi;\zeta_\varepsilon) \rangle^m \qquad U \times \Gamma_M, \ |(\tau,\xi;\zeta_\varepsilon)| \geq r_\varepsilon, \ \varepsilon \in (0,\eta].
\end{equation}
\end{defn}
Figure 1 below shall illustrate the relation of the sets $\Gamma_s$ and $\Gamma_{s,M}$ in the following special situation. We fix a point $(x_0,z_0) \in \mathbb{R}^{n}$, let $\zeta_0^2 :=  \frac{1}{c^{\ast 2}}(x_0,z_0) \tau_0^2 - |\xi_0|^2$ and let $\Gamma_s$ be a conic neighborhood around $(\tau_0,\xi_0,\zeta_0)$. Then  for any $s > 0$, however small, we can define
\begin{equation*}
\zeta_{0,\varepsilon}^{s 2} := \begin{cases}
\frac{1}{c^{\ast 2}}(x_0,z_0) \tau_0^2 - |\xi_0|^2 & \varepsilon \in (C_s, 1] \\
\frac{1}{c^{2}_\varepsilon}(x_0,z_0) \tau_0^2 - |\xi_0|^2 & \varepsilon \in (0, C_s]
\end{cases}
\end{equation*}
for some positive constant $C_s$ depending on $s>0$ and will be specified below. Then 
\begin{equation*}
|\zeta_{0}^{2} - \zeta_{0,\varepsilon}^{s \ 2}| = \Big| \Bigl(\frac{1}{c^{\ast 2}} - \frac{1}{c^{2}_\varepsilon}\Bigr)(x_0,z_0) \tau_0^2 \Big| \le C(\tau_0) \omega_\varepsilon^{-\mu} \qquad \forall \varepsilon \in (0,1].
\end{equation*}
Now let $r(s)$ be a function depending on $s$ such that $r(s) \to 0$ as $s \to 0$. 
Then, since $C(\tau_0) \omega_\varepsilon^{-\mu} \le r(s)/2$ for all $\varepsilon$ small enough, i.e. for $\varepsilon$ less or equal a $C_s >0$ (with $C_s \to 0$ as $s \to 0$) we get
\begin{equation*}
|\zeta_{0}^{2} - \zeta_{0,\varepsilon}^{s 2}| \le r(s)/2 \quad \forall \varepsilon \in (0, C_s].
\end{equation*}
By the definition of $\zeta_{0,\varepsilon}^{s 2}$ we even get
\begin{equation*}
|\zeta_{0}^{2} - \zeta_{0,\varepsilon}^{s 2}| \le r(s)/2 \quad \forall \varepsilon \in (0, 1].
\end{equation*}
So in particular we have $(\tau_0,\xi_0,\zeta_{0,\varepsilon}^{s 2}) \in \Gamma_s$ for all $\varepsilon \in (0,1]$ and $\zeta_{0,\varepsilon}^{s 2} \to \zeta_{0}^{2}$ as $\varepsilon \to 0$.

Therefore, if we working in an open ball around $(\tau_0,\xi_0,\zeta_{0})$ with radius $r(s)>0$ we can find a generalized point $(\tau_0,\xi_0,\zeta_{0,\varepsilon}^s)_\varepsilon$ such that $(\tau_0,\xi_0,\zeta_{0,\varepsilon}^s)$ is contained in $B_{r(s)} \bigl((\tau_0,\xi_0,\zeta_{0}) \bigr)$ for all $\varepsilon \in (0,1]$.
\begin{figure}[H]
\includegraphics{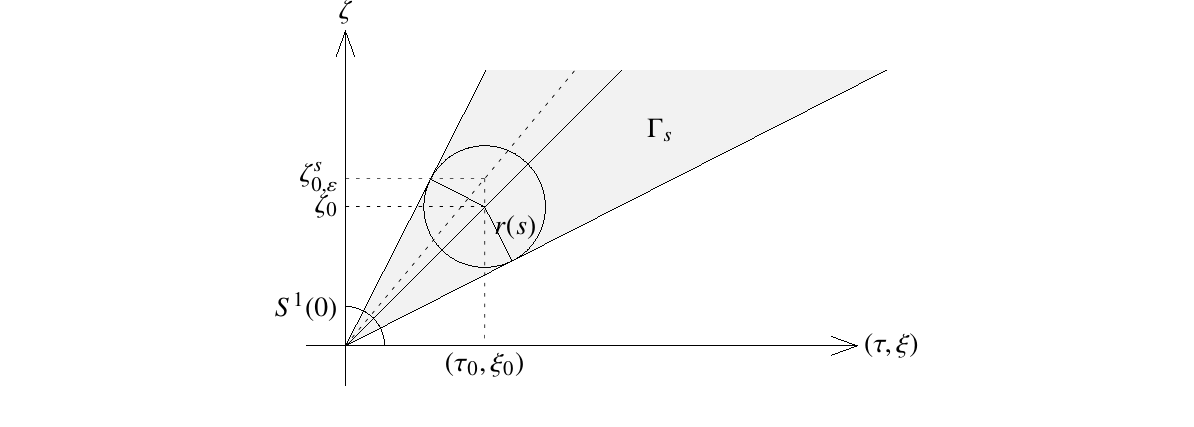}
\caption{Here the shaded area corresponds to the set $\Gamma_s$. The dashed line through $(\tau_0,\xi_0,\zeta^s_{0,\varepsilon})$ and the origin moves to to the line through $(\tau_0,\xi_0,\zeta_{0})$ and the origin as $\varepsilon \to 0$.}
\end{figure}
\begin{rem}\label{rem:strongell->ell}
In particular, since $\Gamma \hookrightarrow \widetilde{\Gamma}$ condition (\ref{eqn:star-lsc-ell}) implies
\begin{equation}\label{eqn:lsc-ell-1}
|a_\varepsilon(t,x,z,\tau,\xi;\zeta)| \geq \frac{1}{s_\varepsilon} \langle (\tau,\xi;\zeta) \rangle^m \qquad U \times \Gamma, \ |(\tau,\xi;\zeta)| \geq r_\varepsilon, \ \varepsilon \in (0,\eta].
\end{equation}
so the symbol $a$ is logarithmic slow scale micro-elliptic at $(t_0,x_0,z_0,\tau_0,\xi_0;\zeta_0)$.
\end{rem}
\begin{rem}
The definition of strong slow scale micro-ellipticity is straightforward. Again, one simply replaces lsc by sc in Definition~\ref{defn:*-lsc-ellipticity}.
\end{rem}
\begin{lem}\label{lem:equiv-lsc-ellipticity}
Let $a \in \widetilde{S}^m_{lsc}(\mathbb{R}^{n+1} \times \mathbb{R}^{n+1})$. Then $a$ is strong logarithmic slow scale micro-elliptic at $(t_0,x_0,z_0,\tau_0,\xi_0;\zeta_0)$ if and only if it is logarithmic slow scale micro-elliptic there.
So if $a$ satisfies (\ref{eqn:lsc-ell-1}), then it also satisfies (\ref{eqn:star-lsc-ell}) and vice versa.
\end{lem}
\begin{proof}
The first direction is already shown by Remark~\ref{rem:strongell->ell}. 

On the other hand suppose that the symbol $a$ is logarithmic slow scale micro-elliptic at $(t_0,x_0,z_0,\tau_0,\xi_0;\zeta_0) \in T^*\mathbb{R}^{n+1} \setminus 0$. Then there exists a representative $(a_\varepsilon)_\varepsilon$ of $a$ such that the following is satisfied: $\exists$ relatively compact open neighborhood $U$ of $(t_0,x_0,z_0)$, $\exists$ (classical) conic neighborhood $\Gamma$ of $(\tau_0,\xi_0;\zeta_0)$, $\exists$ nets $(r_\varepsilon)_\varepsilon, (s_\varepsilon)_\varepsilon \in \Pi_{lsc}$, $\exists \eta \in (0,1]$ such that
\begin{equation*}
|a_\varepsilon(t,x,z,\tau,\xi;\zeta)| \geq \frac{1}{s_\varepsilon} \langle (\tau,\xi;\zeta) \rangle^m \quad \text{ on } U \times \Gamma, \ |(\tau,\xi;\zeta)| \geq r_\varepsilon, \ \varepsilon \in (0,\eta].
\end{equation*}
In particular, since $\Gamma$ is conic neighborhood of $(\tau_0,\xi_0;\zeta_0)$, it is of the form
\begin{equation*}
\Gamma = \{ \lambda (\tau,\xi;\zeta) \in \mathbb{R}^{n+1} \setminus 0 \ | \ (\tau,\xi;\zeta) \in B_s\Bigl( \frac{(\tau_0,\xi_0;\zeta_0)}{|(\tau_0,\xi_0;\zeta_0)|}\Bigr) \cap S^1(0), \ \lambda > 0 \}
\end{equation*}
for some $s>0$ small enough.

We define $\widetilde{\Gamma}:= \Gamma_{M} / \sim$ the generalized cone with respect to $\zeta$ constructed from $\Gamma$ as above.

Take $(\tau,\xi;\zeta_\varepsilon)_\varepsilon \in \Gamma_M$. Then $(\tau,\xi;\zeta_\varepsilon) \in \Gamma$ for any fixed $\varepsilon \in (0,1]$ and $\exists N \in \mathbb{N}$ such that $|\zeta_\varepsilon| = \mathcal{O}(\varepsilon^{-N})$ as $\varepsilon \to 0$. Hence $\exists \eta \in (0,1]$: 
\begin{equation*}
|a_\varepsilon(t,x,z,\tau,\xi;\zeta_\varepsilon)| \geq \frac{1}{s_\varepsilon} \langle (\tau,\xi;\zeta_\varepsilon) \rangle^m \quad \text{ on } U \times \Gamma_M, \ |(\tau,\xi;\zeta_\varepsilon)| \geq r_\varepsilon, \ \varepsilon \in (0,\eta].
\end{equation*}
As a conclusion we obtain that $(t_0,x_0,z_0,\tau_0,\xi_0;\zeta_0) \in \text{Ell}^*_{lsc}(a)$, i.e., the symbol $a$ is strong logarithmic slow scale micro-elliptic at the point $(t_0,x_0,z_0,\tau_0,\xi_0;\zeta_0)$.

We have therefore shown that logarithmic slow scale micro-ellipticity at a particular point is equivalent to strong logarithmic slow scale micro-ellipticity at that point.
\end{proof}
An inspection of the proof of Lemma~\ref{lem:equiv-lsc-ellipticity} shows the following result.
\begin{cor}\label{cor:equiv-sc-ellipticity}
Given $a \in \widetilde{S}^m_{sc}$ then $a$ is strong slow scale elliptic at a phase space point is equivalent to saying that $a$ is slow scale elliptic at that point. 
\end{cor}
\begin{lem}\label{lem:Sigma-Ell}
Let $L$ be the generalized differential operator given in (\ref{eqn:GovEqn}). Then $\text{Ell}_{sc}(L)$ is a subset of the complement of $\Sigma$, where
\begin{equation*}
\Sigma := \{ (t,x,z,\tau,\xi,\zeta) \in T^*\mathbb{R}^{n+1} \setminus 0 \ | \ \zeta^2-\frac{1}{c^{* 2}}(x,z)\tau^2 + |\xi|^2 = 0 \}
\end{equation*}
and $\frac{1}{c^{* 2}} \in \mathcal{C}^{0,\mu}(\mathbb{R}^n)$ is as in the beginning of section~\ref{sec:GovEqn}. Note that $\Sigma$ is a conic set and independent of the regularization parameter $\varepsilon \in (0,1]$. Moreover we have the following inclusion relations:
\begin{equation*}
\Sigma \subseteq \text{Ell}_{sc}(L)^c \subseteq \text{Ell}_{lsc}(L)^c.
\end{equation*}
\end{lem}
\begin{proof}
Let $(t_0,x_0,z_0,\tau_0,\xi_0,\zeta_0) \in T^*\mathbb{R}^{n+1} \setminus 0$ be an arbitrary but fixed point of $\Sigma$. Hence,  $\zeta_0^2=\frac{1}{c^{* 2}}(x_0,z_0)\tau_0^2 - |\xi_0|^2$.

Assume that the operator $L$ with generalized symbol in $ \widetilde{S}^2_{lsc}(\mathbb{R}^{n+1} \times \mathbb{R}^{n+1})$ is slow scale micro-elliptic at the point $(t_0,x_0,z_0,\tau_0,\xi_0,\zeta_0) \in T^*\mathbb{R}^{n+1} \setminus 0$ . Then by Corollary~\ref{cor:equiv-sc-ellipticity} there is a relatively compact open neighborhood $U$ of $(t_0,x_0,z_0)$, a conic neighborhood $\Gamma$ of $(\tau_0,\xi_0,\zeta_0)$ such that for some $(r_\varepsilon)_\varepsilon, (s_\varepsilon)_\varepsilon \in \Pi_{sc}$ and some $\eta \in (0,1]$ we have  
\begin{equation*}
|L_\varepsilon(x,z,\tau,\xi,\zeta_\varepsilon)| \ge \frac{1}{s_\varepsilon} \langle (\tau,\xi,\zeta_\varepsilon) \rangle^2 
\end{equation*}
for all $(t,x,z,\tau,\xi,\zeta_\varepsilon)_\varepsilon \in U \times \Gamma_M$, $|(\tau,\xi,\zeta_\varepsilon)| \ge r_\varepsilon$, $\varepsilon \in (0,\eta]$ and $\widetilde{\Gamma}=\Gamma_{M}/ \sim$ is the generalized conic neighborhood of $(\tau_0,\xi_0,\zeta_{0})$ from above such that $\Gamma \hookrightarrow \widetilde{\Gamma}$.

For $\varepsilon$ small enough we now define $\zeta_{0,\varepsilon}^2 := \frac{1}{c_{\varepsilon}^2}(x_0,z_0)\tau_0^2 - |\xi_0|^2$ which tends to $\zeta_0^2$ (for fixed $\tau_0 \in \mathbb{R}$) as $\varepsilon$ tends to 0. 

Then one can choose $(\tau_0,\xi_0,\zeta_{0,\varepsilon})_\varepsilon \in \Gamma_M$, i.e. there is a positive constant $s$ such that
\begin{equation*}
(\tau_0,\xi_0;\zeta_{0,\varepsilon}) \in B_s\Bigl( \frac{(\tau_0,\xi_0;\zeta_0)}{|(\tau_0,\xi_0;\zeta_0)|}\Bigr) \cap S^1(0) \qquad \text{for every fixed } \varepsilon \in (0,1]
\end{equation*}
since
\begin{equation*}
(\tau_0,\xi_0;\zeta_{0,\varepsilon}) \to (\tau_0,\xi_0;\zeta_0) \qquad \text{ as } \varepsilon \to 0.
\end{equation*}

On the other hand we have that the principal symbol of $L_\varepsilon$ vanishes on $(x_0,z_0,\tau_0,\xi_0,\zeta_{0,\varepsilon})$, i.e.
\begin{equation*}
\sigma(L_\varepsilon)(x_0,z_0,\tau_0,\xi_0,\zeta_{0,\varepsilon}) = 0 
\end{equation*}
for all $\varepsilon$ small enough. So, by Proposition~\ref{prop:ps-ell} and Remark~\ref{rem:ps-ell}, the symbol $L$ is not strong slow scale micro-elliptic and therefore not slow scale micro-elliptic at $(t_0,x_0,z_0,\tau_0,\xi_0,\zeta_0)$ - a contradiction.
\end{proof}
\begin{lem}
Let $L$ and $\Sigma$ be as in Lemma~\ref{lem:Sigma-Ell}. Then $\Sigma = \text{Ell}^c_{lsc}(L)$.
\end{lem}
\begin{proof}
By Lemma~\ref{lem:Sigma-Ell} it suffices to show the inclusion $\Sigma^c \subseteq \text{Ell}_{lsc}(L)$. To do this, we introduce the continuous function
\begin{equation*}
f(x,z,\tau,\xi,\zeta) := \zeta^2 - \frac{1}{c^{* 2}}(x,z) \tau^2 + |\xi|^2,
\end{equation*}
which coincides with the limit of the principal symbol of $\rho_\varepsilon L_\varepsilon$ as $\varepsilon \to 0$. So,
\begin{equation*}
\Sigma^c = \{ (t,x,z,\tau,\xi,\zeta) \in T^* \mathbb{R}^{n+1} \setminus 0 \ :
|f(x,z,\tau,\xi,\zeta)| > 0\}
\end{equation*}
is an open subset of the phase space.

Now given a point $(t_0,x_0,z_0,\tau_0,\xi_0,\zeta_0) \in \Sigma^c$, there exist a relatively compact open neighborhood $U$ of $(t_0,x_0,z_0)$, a conic neighborhood $\Gamma$ of $(\tau_0,\xi_0,\zeta_0)$ and a (small) $\delta > 0$ such that
\begin{equation*}
|f(x,z,\tau,\xi,\zeta)| = \bigl|\zeta^2 - \frac{1}{c^{* 2}}(x,z) \tau^2 + |\xi|^2\bigr| \ge \delta |(\tau,\xi,\zeta)|^2 \qquad U \times \bigl(\Gamma \cap S^1(0) \bigr).
\end{equation*}
Since $f$ is homogeneous of degree 2 in $(\tau,\xi,\zeta)$ it follows that
\begin{equation*}
|f(x,z,\tau,\xi,\zeta)| = \bigl|\zeta^2 - \frac{1}{c^{* 2}}(x,z) \tau^2 + |\xi|^2 \bigr| \ge \delta |(\tau,\xi,\zeta)|^2 \qquad U \times \Gamma, \ |(\tau,\xi,\zeta)| \ge 1.
\end{equation*}
Hence, there exists a constant $C_\delta \in (0,1]$, depending on $\delta > 0$, such that
\begin{equation*}
\begin{split}
|\sigma(L_\varepsilon(t,x,z,\tau,\xi,\zeta))| &\ge \bigl|\zeta^2 - \frac{1}{c^{* 2}}(x,z) \tau^2 + |\xi|^2\bigr| - \bigl|\Bigl(\frac{1}{c_\varepsilon^2} - \frac{1}{c^{* 2}}\Bigr)(x,z) \tau^2 \bigr| \ge\\
&\ge \delta |(\tau,\xi,\zeta)|^2 - C\omega_\varepsilon^{-\mu}|(\tau,\xi,\zeta)|^2 \ge\\
&\ge \frac{\delta}{2} |(\tau,\xi,\zeta)|^2 \qquad U \times \Gamma, \ |(\tau,\xi,\zeta)| \ge 1, \ \varepsilon \in (0,C_\delta].
\end{split}
\end{equation*}
Therefore, $L_\varepsilon$ is lsc-elliptic at $(t_0,x_0,z_0,\tau_0,\xi_0,\zeta_0)$, which is the desired conclusion.
\end{proof}
\begin{rem} 
We define the following set:
\begin{multline*}
\Lambda_M := \{ (t,x,z,\tau,\xi;\zeta_\varepsilon)_\varepsilon \in \mathbb{R}^{n+1} \times ( \mathbb{R}^n \times \mathbb{R}^{(0,1]}) \setminus 0 \ | \\
\ \zeta_\varepsilon^2 = \frac{1}{c_\varepsilon^2}(x,z) \tau^2 - |\xi|^2, 
\exists N \in \mathbb{N} : |\zeta_\varepsilon| = \mathcal{O}(\varepsilon^{-N}) \text{ as } \varepsilon \to 0\}.
\end{multline*}
If a is a symbol on $\Sigma$, then the evaluation on $\Lambda_M / \sim$ is in general not defined.  
If a is a symbol on $\Gamma$, then the generalized point value of $a$ at a point of $\Gamma_M / \sim$ is a well-defined element on $\widetilde{\mathbb{C}}$. 

A fixed generalized point $(t_0,x_0,z_0,\tau_0,\xi_0,\zeta_{0,\varepsilon})_\varepsilon \in \Lambda_M$ satisfies $ \zeta_{0,\varepsilon}^2 = \frac{1}{c_\varepsilon^2}(x_0,z_0) \tau_0^2 - |\xi_0|^2$. Since $\tau_0 \in \mathbb{R}$ is fixed we get
\begin{equation*}
\zeta_{0,\varepsilon}^2 = \frac{1}{c_\varepsilon^2}(x_0,z_0) \tau_0^2 - |\xi_0|^2 \to \zeta_{0}^2 = \frac{1}{c^{*2}}(x_0,z_0) \tau_0^2 - |\xi_0|^2
\end{equation*}
as $\varepsilon \to 0$ and where $(t_0,x_0,z_0,\tau_0,\xi_0,\zeta_{0}) \in \Sigma$. Moreover, 
\begin{equation*}
|\zeta_{0,\varepsilon}^2- \zeta_{0}^2| = |\Bigl(\frac{1}{c_\varepsilon^{2}}-\frac{1}{c^{*2}}\Bigr)(x_0,z_0) \tau_0^2| \le C \omega_\varepsilon^{-\mu} \tau_0^2.
\end{equation*}
\end{rem}
In the next subsection we will construct an approximated inverse for logarithmic slow scale elliptic pseudodifferential operators.
\subsection{Construction of a parametrix}\label{subsec:parametrix} 

In this subsection we will discuss uniqueness and existence results concerning the invertibility of elliptic pseudodifferential operators. Recall that the asymptotic expansion is unique modulo operators of order $-\infty$. As a consequence we will show that one obtains uniqueness of a parametrix modulo operators of infinite order. 

Note that given $(a_\varepsilon)_\varepsilon$ a symbol of class $S^{-\infty}_{lsc}$ the operator defined by $(u_\varepsilon)_\varepsilon \mapsto (a_\varepsilon(x,D)u_\varepsilon)_\varepsilon$, $((u_\varepsilon)_\varepsilon \in \mathcal{M}_{H^{\infty}})$ has a regular kernel representation in the sense that moderate nets are mapped into regular nets. 

In the following theorem we shall work on the level of representatives.
\begin{thm}\label{thm:parametrix}
Let $(a_\varepsilon)_\varepsilon \in S^m_{lsc}$ be logarithmic slow scale elliptic, i.e.
\begin{equation*}
\begin{split}
\exists (r_\varepsilon)_\varepsilon \in \Pi_{lsc}, \ \exists (s_\varepsilon)_\varepsilon \in \Pi_{lsc},& \ \exists \eta \in (0,1]:\\
 &|a_\varepsilon(x,\xi)| \ge \frac{1}{s_\varepsilon} \langle \xi \rangle^m \qquad |\xi| \ge r_\varepsilon, \ \varepsilon \in (0,\eta].
\end{split}
\end{equation*}
Then there exists $(p_\varepsilon)_\varepsilon \in S^{-m}_{lsc}$ such that $(p_\varepsilon\# a_\varepsilon)_\varepsilon = 1$ modulo $S^{-\infty}_{lsc}$.
\end{thm}
\begin{proof}
\textit{Step 1:} As in \cite[Proposition 2.8]{Garetto:08} we show that there exists $(p_{-m,\varepsilon})_\varepsilon \in S^{-m}_{lsc}$ such that
\begin{equation}\label{eqn:3.1}
(p_{-m,\varepsilon} a_\varepsilon - 1)_\varepsilon \in S^{-\infty}_{lsc}.
\end{equation}
For the proof we let $\psi \in \mathcal{C}^{\infty}(\mathbb{R}^n)$, $0 \le \psi \le 1$, $\psi(\xi)=0$ for $|\xi| \le 1$ and $\psi(\xi)=1$ for $|\xi| \ge 2$.
Then
\begin{equation*}
p_{-m,\varepsilon}(x,\xi) := a_{\varepsilon}^{-1}(x,\xi) \psi(\xi / r_\varepsilon)
\end{equation*}
with $(r_\varepsilon)_\varepsilon \in \Pi_{lsc}$ as in the theorem does the job. Indeed one can show that
\begin{equation*}
|\partial_\xi^{\alpha} \partial_x^{\beta} a^{-1}_\varepsilon(x,\xi)| \le \omega_\varepsilon(\alpha,\beta) \langle \xi \rangle^{-|\alpha|} | a^{-1}_\varepsilon(x,\xi)| \qquad |\xi| \ge r_\varepsilon
\end{equation*}
for some $\omega_\varepsilon(\alpha,\beta) \in \Pi_{lsc}$ (cf. \cite[Lemma 6.3]{GarettoGramchevMoe:05}). Then, by the properties of the function $\psi$, we get $(p_{-m,\varepsilon})_\varepsilon \in S^{-m}_{lsc}$ for $|\xi| \ge 2r_\varepsilon$ and $|\xi| \le r_\varepsilon$. On the set $r_\varepsilon \le |\xi| \le 2r_\varepsilon$ we observe that
\begin{equation*}
\begin{split}
|\partial_\xi^{\alpha} \partial_x^{\beta} & p_{-m,\varepsilon}(x,\xi)| 
= |\sum_{\alpha' \le \alpha} {\binom{\alpha}{\alpha'}} \partial_\xi^{\alpha'} \partial_x^{\beta} a^{-1}_\varepsilon(x,\xi) \partial_\xi^{\alpha-\alpha'} \psi(\xi / r_\varepsilon)| \le \\ 
&\le \sum_{\alpha' \le \alpha} {\binom{\alpha}{\alpha'}} \omega_\varepsilon(\alpha',\beta) |a_\varepsilon^{-1}(x,\xi)| \langle \xi \rangle^{-|\alpha'|} r_\varepsilon^{-|\alpha-\alpha'|} \chi_{[r_e,2r_\varepsilon]}(\xi) \sup |(\partial^{\alpha-\alpha'} \psi)(\xi / r_\varepsilon)| \le \\
&\le \sum_{\alpha' \le \alpha} {\binom{\alpha}{\alpha'}} \omega_\varepsilon(\alpha',\beta) r_\varepsilon^{-|\alpha-\alpha'|} s_\varepsilon \langle 2r_\varepsilon \rangle^{|\alpha-\alpha'|}  \langle \xi \rangle^{-m-|\alpha|} \sup |(\partial^{\alpha-\alpha'} \psi)(\xi)|
\end{split}
\end{equation*}
and therefore, $(p_{m,\varepsilon})_\varepsilon$ is in $S^{-m}_{lsc}$. In order show (\ref{eqn:3.1}) we write
\begin{equation*}
p_{-m,\varepsilon} a_\varepsilon(x,\xi) = a_\varepsilon^{-1} a_\varepsilon(x,\xi) + a_\varepsilon^{-1} a_\varepsilon(x,\xi) (\psi(\xi/r_\varepsilon)-1) = 1 + (\psi(\xi/r_\varepsilon)-1).
\end{equation*}
Since for any $l >0$ and $\alpha \in \mathbb{N}^n \setminus 0$
\begin{eqnarray*}
& &\sup_{r_\varepsilon \le |\xi| \le 2 r_\varepsilon} |\langle \xi \rangle^{l} (\psi(\xi/r_\varepsilon)-1)| \le c(\psi) \langle 2 r_\varepsilon \rangle^l\\
& &\sup_{r_\varepsilon \le |\xi| \le 2 r_\varepsilon} |\langle \xi \rangle^{l} \partial_\xi^{\alpha} (\psi(\xi/r_\varepsilon)-1)| = \sup_{r_\varepsilon \le |\xi| \le 2 r_\varepsilon} |\langle \xi \rangle^{l} r_\varepsilon^{-|\alpha|} (\partial^{\alpha} \psi)(\xi/r_\varepsilon)| \le c(\psi) r_\varepsilon^{-|\alpha|} \langle 2 r_\varepsilon \rangle^l
\end{eqnarray*}
we have shown (\ref{eqn:3.1}).

\textit{Step 2:} We recursively define for $k \ge 1$:
\begin{equation*}
p_{-m-k,\varepsilon}(x,\xi) := -\Big\{ \sum_{\substack{ |\gamma| + j = k \\ j < k}} \frac{(-i)^{|\gamma|}}{\gamma!}
 \partial^{\gamma}_{x} a_{\varepsilon}(x,\xi) \partial_{\xi}^{\gamma} p_{-m-j,\varepsilon}(x,\xi) \Big\} a_{\varepsilon}^{-1}(x,\xi) \psi(\xi/r_\varepsilon).
\end{equation*}
By the same arguments as in \cite[Propositions 6.5 and 6.6]{GarettoGramchevMoe:05}, one can show that each $(p_{-m-k,\varepsilon})_\varepsilon$ is in $S^{-m-k}_{lsc}$ and by Theorem~\ref{thm:AE} there exists a net $(p_{\varepsilon})_\varepsilon \in S^{-m}_{lsc}$ with $(p_\varepsilon)_\varepsilon \sim \sum_j (p_{-m-j,\varepsilon})_\varepsilon$. 

\textit{Step 3:} The aim is to show that
$(p_\varepsilon \# a_\varepsilon - 1)_\varepsilon \in S^{-\infty}_{lsc}$.
Therefore, let $(\sigma_\varepsilon)_\varepsilon$ be the generalized symbol with the asymptotic expansion
\begin{equation*}
\sigma_\varepsilon(x,\xi) \sim \sum_{|\gamma| \ge 0} \frac{1}{\gamma!} \big( \partial_\xi^{\gamma} p_\varepsilon D_x^{\gamma} a_\varepsilon \big)(x,\xi).
\end{equation*}
and we will show that $(\sigma_\varepsilon - 1)_\varepsilon$ in $S^{-\infty}_{lsc}$.

We write
\begin{equation*}
\begin{split}
\sigma_\varepsilon - \sum_{|\gamma|<N} \frac{1}{\gamma!} D_x^{\gamma} a_\varepsilon \partial_\xi^{\gamma} p_\varepsilon &= \sigma_\varepsilon - \sum_{|\gamma|<N} \frac{1}{\gamma!} D_x^{\gamma} a_\varepsilon \sum_{l < N} \partial_\xi^{\gamma} p_{-m-l,\varepsilon} +\\
&\phantom{=}+ \sum_{|\gamma| < N} \frac{1}{\gamma!} D_x^{\gamma} a_\varepsilon \sum_{l \ge N} \partial_\xi^{\gamma} p_{-m-l,\varepsilon}.
\end{split}
\end{equation*}
Since the left-hand side and the last sum on the right-hand side are in $S^{-N}_{lsc}$ we obtain
\begin{equation*}
\bigl( \sigma_\varepsilon - \sum_{|\gamma|<N} \frac{1}{\gamma!} D_x^{\gamma} a_\varepsilon \sum_{l < N} \partial_\xi^{\gamma} p_{-m-l,\varepsilon} \bigr)_\varepsilon \in S^{-N}_{lsc}.
\end{equation*}
We now write
\begin{multline*}
\sum_{|\gamma|<N} \frac{1}{\gamma!} D_x^{\gamma} a_\varepsilon \sum_{l < N} \partial_\xi^{\gamma} p_{-m-l,\varepsilon}  = p_{-m,\varepsilon} a_\varepsilon +\\
\sum_{k=1}^{N-1} \Big\{ p_{-m-k,\varepsilon}a_\varepsilon + \sum_{\substack{ |\gamma| + l = k \\ l < k}} \frac{1}{\gamma!} \partial_\xi^{\gamma} p_{-m-l,\varepsilon} D_x^{\gamma} a_\varepsilon \Big\}
+\sum_{{\substack{ |\gamma| + l \ge N \\ |\gamma| < N, l < N}}} \frac{1}{\gamma!} \partial_\xi^{\gamma} p_{-m-l,\varepsilon} D_x^{\gamma} a_\varepsilon,
\end{multline*}
where the last expression is in $S^{-N}_{lsc}$. The second sum on the right-hand side vanishes for $|\xi| \ge 2r_\varepsilon$ and $|\xi| \le r_\varepsilon$ by construction. So it remains to estimate this expression on the set $r_\varepsilon \le |\xi| \le 2r_\varepsilon$. We observe for any $l >0$ and $\alpha \in \mathbb{N}^n$ that 
\begin{equation*}
\sup_{r_\varepsilon \le |\xi| \le 2 r_\varepsilon} |\langle \xi \rangle^{l} \partial_\xi^{\alpha} \psi(\xi/r_\varepsilon)| = \sup_{r_\varepsilon \le |\xi| \le 2 r_\varepsilon} |\langle \xi \rangle^{l} r_\varepsilon^{-|\alpha|} (\partial^{\alpha} \psi)(\xi/r_\varepsilon)| \le c(\psi) r_\varepsilon^{-|\alpha|} \langle 2 r_\varepsilon \rangle^l
\end{equation*}
and we conclude that this term is contained in $S^{-N}_{lsc}$. 

Using Step 1, i.e. $(p_{-m,\varepsilon} a_\varepsilon)_\varepsilon = 1 \mod S^{-\infty}_{lsc}$, we get for every $N \ge 1$
\begin{equation*}
\bigl(\sigma_\varepsilon - \sum_{|\gamma|<N} \frac{1}{\gamma!} (\partial_\xi^{\gamma} p_\varepsilon D_x^{\gamma} a_\varepsilon) \bigr)_\varepsilon = (\sigma_\varepsilon -1)_\varepsilon \quad \mod S^{-N}_{lsc}
\end{equation*}
and the proof is complete.
\end{proof}
Recall that $a \in \widetilde{S}^{m}_{lsc}$ is called lsc-elliptic if one of its representatives $(a_\varepsilon)_\varepsilon$ is lsc-elliptic. Then, the operator $p(x,D)$ has the symbol $(p_\varepsilon)_\varepsilon + \ensuremath{\mathcal N}_{S^{-m}} \in \widetilde{S}^{-m}_{lsc}$ and $(p_\varepsilon)_\varepsilon$ is as in Theorem~\ref{thm:parametrix}. Furthermore, by Step 3 of the same theorem we get
\begin{equation*}
p(x,D) \circ a(x,D)u(x) = \Bigl( \int e^{i (x-y) \xi} u_\varepsilon (y) \,dy \,\ \dbar \xi \Bigr)_\varepsilon + \mathcal{N}_{H^{\infty}}(\mathbb{R}^n).
\end{equation*}
Therefore, we obtain the following result.
\vbadness=1800
\begin{thm}
Let $a \in \widetilde{S}^{m}_{lsc}$ be a logarithmic slow scale elliptic symbol of order $m$. Then there exists a logarithmic slow scale elliptic pseudodifferential operator of order $-m$ with symbol $p \in \widetilde{S}^{-m}_{lsc}$ such that for all $u \in \GLtwo$
\begin{eqnarray*}
a(x,D) \circ p(x,D) u &=& u + r(x,D) u\\
p(x,D) \circ a(x,D) u &=& u + s(x,D) u
\end{eqnarray*}
where $r(x,D)$ and $s(x,D)$ are regularizing operators, that is they map $\GLtwo$ into $\GregLtwo$ (see section~\ref{sec:lscWaveFrontSet}).
\end{thm}
\section{A factorization procedure for the generalized strictly hyperbolic wave operator}\label{sec:Factorization}
Concerning products of logarithmic slow scale pseudodifferential operators that approximate the generalized operator $L$ from (\ref{eqn:GovEqn}), we will follow the ideas in \cite[Appendix II]{Kumano-go:81}, \cite[Chapter 23]{Hoermander:3}. First we recall from (\ref{eqn:GovEqn}) 
\begin{equation}\label{GoverningOperator}
\begin{split}
L(x,z,D_t,D_x,D_z) &=  \partial_{z}\frac{1}{\rho}\partial_z + \sum_{j=1}^{n-1} \partial_{x_j}\frac{1}{\rho}\partial_{x_j} -  \frac{1}{\rho}\frac{1}{c^2}\partial_{t}^2   =\\
& =: \partial_z \frac{1}{\rho}\partial_z + A_{\rho}(x,z,D_t,D_x)
\end{split}
\end{equation}
where the coefficients $1/\rho$ and $1/c^2$ are in $\ensuremath{\mathcal{G}}_{\infty,\infty}^{lsc}(\mathbb{R}^n)$ in the variables $(x,z)$. This implies that $L$ is an operator of class $\Psi^2_{lsc}(\mathbb{R}^n \times \mathbb{R}^{n+1})$ and a representative $(l_\varepsilon)_\varepsilon$ of the symbol of $L$ is in $S^2_{lsc}(\mathbb{R}^n \times \mathbb{R}^{n+1})$. Furthermore, the symbol of the operator $A_{\rho}$ is in $S^2_{lsc}(\mathbb{R}^n \times \mathbb{R}^n)$ and is given by 
\begin{equation*}
A_{\rho}(x,z,D_t,D_x) = - \frac{1}{\rho} \Big(\frac{1}{c^2}\partial_{t}^2 - \sum_{j=1}^{n-1} \partial_{x_j}^2 \Big) + \sum_{j=1}^{n-1} \Big(\partial_{x_j}\frac{1}{\rho}\Big)\partial_{x_j}.
\end{equation*}
Note that this operator is only a pseudodifferential operator in $(x,t)$ depending smoothly on the parameter $z$ on the level of representatives for any fixed $\varepsilon \in (0,1]$. We denote by $a_{\rho} \in \widetilde{S}^2_{lsc}(\mathbb{R}^n \times \mathbb{R}^n)$ the principal symbol of $A_\rho$.

Moreover, we set
\begin{equation*}
A(x,z,D_t,D_x) := - \frac{1}{c^2}(x,z)\partial_{t}^2 + \sum_{j=1}^{n-1} \partial_{x_j}^2.
\end{equation*}
As already mentioned in section~\ref{sec:GovEqn}, the operator $L$ is not globally strictly hyperbolic but we can restrict the analysis to an appropriate space on which the operator becomes strictly hyperbolic. In particular, we follow Stolk in \cite{Stolk:04} and introduce the following set of points on which we will then establish the factorization. Therefore we recall from section~\ref{sec:GovEqn}:
\begin{equation}\label{BasicSet}
I'_{\theta} := \Bigl\{ (x,z,\tau,\xi) \in \mathbb{R}^n \times \mathbb{R}^n \ | \ \tau \neq 0, \ |c^*(x,z) \tau^{-1} \xi | \le \sin \theta \Bigr\}
\end{equation}
for some fixed $\theta \in (0, \pi/2)$ and $c^*(x,z) = \lim_{\varepsilon \to 0} c_\varepsilon(x,z)$ in $\mathcal{C}^{0,\mu}(\mathbb{R}^n)$ as in section~\ref{sec:GovEqn}. Then $I'_{\theta}$ is a subset of $\mathbb{R}^n \times (\mathbb{R}^n \setminus 0)$, independent of the regularization parameter $\varepsilon \in (0,1]$ and conic with respect to $(\tau,\xi)$.

Fix $\theta \in (0, \pi/2)$ and denote by $a$ the (principal) symbol of $A$, i.e.
\begin{equation*}
a_{\varepsilon}(x,z,\tau,\xi):=\frac{\tau^2}{c_{\varepsilon}^2(x,z)} - |\xi|^2 \qquad \varepsilon \in (0,1].
\end{equation*}
Then $a$ is logarithmic slow scale elliptic on $I'_{\theta}$ and the generalized roots $[(\zeta_{j,\varepsilon})_\varepsilon]$ of the principal symbol of $L$ are given by $\zeta_{j,\varepsilon}:=\mp i \sqrt{a_{\varepsilon}(x,z,\tau,\xi)}$ which satisfy (\ref{eqn:CharRoots}) on $I'_{\theta}$, i.e.
\begin{equation*}
|\zeta_{1,\varepsilon}(x,z,\tau,\xi) - \zeta_{2,\varepsilon}(x,z,\tau,\xi)| \geq C |(\tau,\xi)| \qquad \text{on } I'_{\theta}, \text{ for } |(\tau,\xi)|\ge M
\end{equation*}
for some constants $C>0, M>0$ and all $\varepsilon$ sufficiently small.

Indeed, we have on $I'_{\theta}$
\begin{equation*}
\begin{split}
|a_{\varepsilon}(x,z,\tau,\xi)| &= \Bigl|\frac{\tau^2}{c_\varepsilon^2(x,z)}-|\xi|^2\Bigr| \ge \Biggl| \Bigl|\frac{\tau^2}{c^{* 2}(x,z)}-|\xi|^2 \Bigr|  - \Bigl| \Bigl(\frac{1}{c^2_\varepsilon(x,z)} - \frac{1}{c^{* 2}(x,z)}\Bigr) \tau^2 \Bigr| \Biggr| \ge\\
&\ge \frac{\tau^2}{c^{* 2}(x,z)} (1-\sin^2 \theta) - C \omega_\varepsilon^{-\mu} \tau^2 \ge C (\tau^2+|\xi|^2)
\end{split}
\end{equation*}
for some generic constant $C>0$ and $\varepsilon$ small enough. So, in particular $a_{\rho}$ is logarithmic slow scale elliptic, since the generalized function $1/\rho$ is strongly positive (see assumption (ii) in section\ref{sec:GovEqn}).
\begin{rem}
In the classical theory of pseudodifferential operators the characteristic set plays an important role in microlocal analysis. We recall that the characteristic set is defined as the set of points where the homogeneous principal symbol of the operator vanishes. Concerning vanishing properties of the homogeneous principal symbol in the generalized setting, we make the following remarks. We denote by $(l_\varepsilon)_\varepsilon$ the principal symbol of $L$.

(1) First note that
\begin{multline*}
l_\varepsilon(\zeta; x,z,\tau,\xi) = -\frac{1}{\rho_\varepsilon(x,z)}\zeta^2 + a_{\rho,\varepsilon}(x,z,\tau,\xi) = \\
= \Big(i\zeta - i\sqrt{a_{\varepsilon}(x,z,\tau,\xi)}\Big) \frac{1}{\rho_\varepsilon(x,z)} \Big(i\zeta + i\sqrt{a_{\varepsilon}(x,z,\tau,\xi)}\Big) = 0 \quad \text{on } I'_{\theta}, \ |(\tau,\xi)| \ge M
\end{multline*}
for some constant $M>0$ is not (classically, i.e. independent of $\varepsilon$) solvable in $T^* \mathbb{R}^{n+1}$ as $\varepsilon \to 0$.

But for any $(x,z,\tau,\xi) \in I'_{\theta}$ there exists two distinct $\zeta_\varepsilon(x,z,\tau,\xi)$ such that
\begin{equation*}
\begin{split}
l_\varepsilon(\zeta_\varepsilon; x,z,\tau,\xi) &= -\frac{1}{\rho_\varepsilon(x,z)} \zeta_\varepsilon^2 + a_{\rho,\varepsilon}(x,z,\tau,\xi) = \\
&= \Big(i\zeta_\varepsilon - i\sqrt{a_{\varepsilon}(x,z,\tau,\xi)}\Big) \frac{1}{\rho_\varepsilon(x,z)} \Big(i\zeta_\varepsilon + i\sqrt{a_{\varepsilon}(x,z,\tau,\xi)}\Big) = 0
\end{split}
\end{equation*}
for all $|(\tau,\xi)| \ge M > 0$ and $\varepsilon \in (0,1]$. In particular, this is satisfied if we set for fixed $(x,z,\tau,\xi) \in I'_{\theta}$
\begin{equation*}
\zeta_\varepsilon(x,z,\tau,\xi) := \pm \sqrt{\frac{\tau^2}{c_{\varepsilon}^2(x,z)} - |\xi|^2} \to \zeta(x,z,\tau,\xi) := \pm \sqrt{\frac{\tau^2}{c^{* 2}(x,z)} - |\xi|^2} \quad \text{as } \varepsilon \to 0.
\end{equation*}
So, the equation has a generalized solution for fixed $(x,z,\tau,\xi) \in I'_{\theta}$.

Note that the generalized principal symbol can always be factorized, i.e.
\begin{multline*}
l_\varepsilon(\zeta; x,z,\tau,\xi) = -\frac{1}{\rho_\varepsilon(x,z)} \zeta^2 + a_{\rho,\varepsilon}(x,z,\tau,\xi) = \\
= \Big(i\zeta + i\sqrt{a_{\varepsilon}(x,z,\tau,\xi)}\Big) \frac{1}{\rho_\varepsilon(x,z)} \Big(i\zeta - i\sqrt{a_{\varepsilon}(x,z,\tau,\xi)}\Big) \quad \text{on } I'_{\theta}, \ |(\tau,\xi)| \ge M
\end{multline*}
regardless of whether the symbol vanishes or not.

(2) The roots $\zeta_{1,\varepsilon}(x,z,\tau,\xi)$, $\zeta_{2,\varepsilon}(x,z,\tau,\xi)$ of the equation
\begin{equation*}
-\frac{1}{\rho_\varepsilon(x,z)}\zeta_\varepsilon^2 + a_{\rho,\varepsilon}(x,z,\tau,\xi) = 0
\end{equation*}
are called the generalized characteristic roots.
\end{rem}
Before we state the main theorem of this section, we give a few more details about notation. In the following we will study operators of the form
\begin{equation*}
S = \sum_{j=0}^2 S_j (x,z,D_t,D_x) \partial_z^{2-j}
\end{equation*}
on $\mathbb{R}^{n+1}$ where the coefficients are operators with symbols $S_j \in S^{k_j}_{lsc}$ for some real numbers $k_j$, $j=0,1,2$. Further, we write 
\begin{equation*}
S = \sum_{j=0}^2 S_{j} (x,z,D_t,D_x) \partial_z^{2-j} \quad \text{ on } I'_{\theta}
\end{equation*}
when the symbols of the coefficients $S_j$ are restricted to the set $I'_{\theta}$.

We will now establish the factorization procedure in order to write the operator $L$ in terms of two first-order pseudodifferential operators of the form $L_{j} = \partial_z + A_{j}$ on $I'_{\theta}$ where $A_{j}$ are pseudodifferential operators with generalized symbols in $\widetilde{S}^{1}_{lsc}$, $j=1,2$.

We are now in the position to show the following result:
\begin{thm}\label{thm:Fact}
Let $L = \partial_z \frac{1}{\rho} \partial_z +A_\rho(x,z,D_t,D_x)$ and $I'_{\theta}$ be as in (\ref{GoverningOperator}) and (\ref{BasicSet}) respectively. Then, on the set $I'_{\theta}$ there are generalized roots $\{(\zeta_{1,\varepsilon})_\varepsilon, (\zeta_{2,\varepsilon})_\varepsilon \}$ of the principal symbol of $(L_\varepsilon)_\varepsilon$ that fulfill the separation property
\begin{equation*}
|\zeta_{1,\varepsilon}(x,z,\tau,\xi) - \zeta_{2,\varepsilon}(x,z,\tau,\xi)| \geq C |(\tau,\xi)| \qquad \text{on } I'_{\theta}, \text{ for } |(\tau,\xi)|\ge M
\end{equation*}
for some constants $C>0, M>0, \eta \in (0,1]$ and all $\varepsilon \in (0,\eta]$. Furthermore, the operator $L$ can be factorized into
\begin{equation}\label{StateFact}
L = L_{1} \frac{1}{\rho} L_{2} + R \qquad \text{ on } I'_{\theta}
\end{equation}
where $L_{j}$ is written as $L_{j} = \partial_z + A_{j}$ and $A_{j}=A_{j}(x,z,D_t,D_x)$ is a parameter-dependent logarithmic slow scale (classical) generalized pseudodifferential operator of order 1 in $(t,x)$. The principal symbol of $A_j$ can be chosen either equal to $(-i \zeta_{j,\varepsilon})_\varepsilon = (\mp i \sqrt{a_\varepsilon})_\varepsilon$, $j=1,2$ or equal to $(i \zeta_{j,\varepsilon})_\varepsilon = (\mp i \sqrt{a_\varepsilon})_\varepsilon$, $j=1,2$ on $I'_{\theta}$. Furthermore, the remainder is given by $R = \Gamma_{1} + \Gamma_{2}\partial_z$ for some pseudodifferential operators $\Gamma_{j}=\Gamma_{j}(x,z,D_t,D_x)$ with parameter $z$ and generalized symbol $\gamma_{j}$ in $\widetilde{S}^{-\infty}_{lsc}$ on $I'_{\theta}$, $j=1,2$.
\end{thm}
\begin{rem}
Note that the product $L_{1} \frac{1}{\rho} L_{2}$ in (\ref{StateFact}) is not a pseudodifferential operator on $\mathbb{R}^{n+1}$. But one can overcome this by introducing a generalized cut-off $\psi (t,x,z,D_t,D_x,D_z)$ such that the difference $\psi L_{1} \frac{1}{\rho} L_{2} - L_{1} \frac{1}{\rho} L_{2}$ is insignificant on some adequate subdomain of the phase space $T^* \mathbb{R}^{n+1}$ under a microlocal point of view. This will be specified in the next section where we introduce a generalized version of microlocal analysis.
\end{rem}
\subsection{Technical Preliminaries}\label{subsec:TechnicalPreliminaries}
As already mentioned above, the symbol $(a_\varepsilon)_\varepsilon$ is logarithmic slow scale elliptic on $I'_{\theta}$, $\theta \in (0,\pi/2)$, i.e.
for any $(x_0,z_0,\tau_0,\xi_0) \in I'_{\theta}$ there exist a relatively compact open neighborhood $U$ of $(x_0,z_0)$ and a conic neighborhood $\Gamma$ of $(\tau_0,\xi_0)$ such that for some $(r_\varepsilon)_\varepsilon \in \Pi_{lsc}$, $(s_\varepsilon)_\varepsilon \in \Pi_{lsc}$ and a constant $\eta \in (0,1]$ we have
\begin{equation*}
|a_{\varepsilon}(x,z,\tau,\xi)| \geq \frac{1}{s_\varepsilon} \langle (\tau,\xi) \rangle^2 \qquad \text{ for } (x,z,\tau,\xi) \in U \times \Gamma, \ |(\tau,\xi)| \ge r_\varepsilon, \ \varepsilon \in (0,\eta]. 
\end{equation*}
As demonstrated in Proposition~\ref{prop:lsc-ell}, we have stability under lower order logarithmic slow scale perturbations, and therefore the total symbol of the operator $(A_{\rho,\varepsilon})_\varepsilon$ itself is logarithmic slow scale elliptic on $I'_{\theta}$.

In order to describe a factorization for the operator $L$, we have to give a meaning to the square root of the symbol of $A$ on the set $I_{\theta}'$. 

Note that the square root of $(a_{\varepsilon})_\varepsilon$ is prescribed on the set $I_{\theta}'$. Outside of $I_{\theta}'$ it is in general not defined but we choose it without the singularity of the square root. 
We remark that such an extension is equal to $(\sqrt{a_\varepsilon})_\varepsilon$ when restricted to $I_{\theta}'$. 

In order to cut off singularities of the square root of $(a_\varepsilon)_\varepsilon$, we define a generalized symbol $(\chi_\varepsilon(x,z,\tau,\xi))_\varepsilon \in S^{0}_{lsc}(\mathbb{R}^n \times (\mathbb{R} \setminus 0) \times \mathbb{R}^{n-1})$ homogeneous of degree $0$ for $
|(\tau,\xi)| \ge 1$ and such that 
\begin{equation*}
\chi_\varepsilon(x,z,\tau,\xi) = \begin{cases}  1 & \text{ on } I_{\theta_1}' \\ 0 & \text{ outside } I_{\theta_2}'\end{cases}
\end{equation*}
for angles $\theta_1, \theta_2 \in (0,\pi/2)$, $\theta_1 < \theta_2$. Then, $(\chi_\varepsilon\sqrt{a_\varepsilon})_\varepsilon \in  S^{0}_{lsc}(\mathbb{R}^n \times (\mathbb{R} \setminus 0) \times \mathbb{R}^{n-1})$ and its restriction to $I_{\theta_1}'$ is equal to $(\sqrt{a_\varepsilon})_\varepsilon$. Moreover, $(\chi_\varepsilon \sqrt{a_\varepsilon})_\varepsilon$ vanishes outside $I_{\theta_2}'$.
\begin{rem}
For the construction of $\chi_\varepsilon$ we define for every fixed $\varepsilon \in (0,1]$ the function
\begin{equation*}
f_\varepsilon(x,z,\tau,\xi) := c_\varepsilon(x,z) \xi \tau^{-1} \qquad \text{ on } \mathbb{R}^n \times (\mathbb{R} \setminus 0) \times \mathbb{R}^{n-1}.
\end{equation*}
Then, for every fixed $\varepsilon \in (0,1]$ we let $\chi_\varepsilon$ be the smooth function defined on $\mathbb{R}^n \times (\mathbb{R} \setminus 0) \times \mathbb{R}^{n-1}$, $0 \le \chi_\varepsilon \le 1$, which is given by 
\begin{equation}\label{eqn:CutOffFunction}
\chi_\varepsilon(x,z,\tau,\xi) := \begin{cases}
0, & |f_\varepsilon| \geq \sin \gamma_2 \\
1, & |f_\varepsilon| \leq \sin \gamma_1\\
1- \frac{1}{{1 + e^{\frac{1}{|f_\varepsilon|-\sin \gamma_1} - \frac{1}{\sin \gamma_2 - |f_\varepsilon|}}}}, & \sin \gamma_1 < |f_\varepsilon| < \sin \gamma_2
\end{cases}
\end{equation}
for some fixed angles $\gamma_1$, $\gamma_2$ with $0 < \theta_1 < \gamma_1 < \gamma_2 < \theta_2 < \pi/2$.

We therefore get $\chi_\varepsilon = 1$ on $I_{\theta_1}'$ since $\exists \eta \in (0,1]$ such that
\begin{equation*}
I_{\theta_1}' \subseteq \{ (x,z,\tau,\xi) \ | \ \tau \neq 0, \ |c_\varepsilon(x,z) \xi / \tau| \le \sin \gamma_1 \} \qquad \varepsilon \in (0,\eta].
\end{equation*}
To show the inclusion, let $(x,z,\tau,\xi) \in I_{\theta_1}'$ and $f(x,z,\tau,\xi) := \lim_{\varepsilon \to 0} f_\varepsilon(x,z,\tau,\xi)$. Then there exists a constant $C>0$ such that
\begin{equation*}
\begin{split}
|f_\varepsilon(x,z,\tau,\xi)| &\le |(f_\varepsilon - f)(x,z,\tau,\xi)| + |f(x,z,\tau,\xi)| \le\\
&\le |(c_\varepsilon(x,z)-c^*(x,z)) \xi/ \tau| + |c^*(x,z) \xi/ \tau| \le \\
&\le \sin \theta_1 (1-C \omega_\varepsilon^{-\mu}) 
\end{split}
\end{equation*}
Since $\omega_\varepsilon^{-\mu}$ tends to $0$ as $\varepsilon \to 0$ and $\theta_1 < \gamma_1$, there exists $\eta \in (0,1]$ such that the right-hand side of the last equation is bounded from above by $\sin \gamma_1$.

Similarly, one can show that $\exists \eta \in (0,1]$ such that for all $\varepsilon \in (0,\eta]$:
\begin{equation*}
\{ (x,z,\tau,\xi) \ | \ \tau \neq 0, \ |c^*(x,z) \xi / \tau| \ge \sin \theta_2 \} \subseteq \{ (x,z,\tau,\xi) \ | \ \tau \neq 0, \ |c_\varepsilon(x,z) \xi / \tau| \ge \sin  \gamma_2 \} 
\end{equation*}
and hence $\chi_\varepsilon = 0$ outside $I_{\theta_2}'$ (note $\chi_\varepsilon$ is defined only for $\tau \neq 0$). Indeed
\begin{equation*}
\begin{split}
|f_\varepsilon(x,z,\tau,\xi)| &\ge |f(x,z,\tau,\xi)| - |(f_\varepsilon - f)(x,z,\tau,\xi)| \ge\\
&\ge |c^*(x,z) \xi/ \tau| \ \Bigl(1- \Big| \frac{c_\varepsilon(x,z)-c^*(x,z)}{c^*(x,z)} \Big| \Bigr) \ge \\
&\ge \sin \theta_2 (1-C \omega_\varepsilon^{-\mu}) \ge \sin \gamma_2
\end{split}
\end{equation*}
as $\varepsilon \to 0$.
Now, since $\frac{1}{c^2} \in \ensuremath{\mathcal{G}}_{\infty,\infty}^{lsc}$ and for every $\varepsilon \in (0,1]$ the function $\chi_\varepsilon$ is smooth on $\mathbb{R}^n \times (\mathbb{R} \setminus 0) \times \mathbb{R}^{n-1}$ and homogeneous of order $0$ for $|(\tau,\xi)| \ge 1$, we can conclude that $\chi_\varepsilon \in S^{0}_{lsc} (\mathbb{R}^n \times (\mathbb{R} \setminus 0) \times \mathbb{R}^{n-1})$ (\cite[Example 1.2]{GS:94}). 
\end{rem}
\subsection{Factorization procedure}
The aim is to decompose the operator $L$ as announced in (\ref{StateFact}) in Theorem~\ref{thm:Fact}. Therefore, we give a construction scheme for the generalized symbols $(a_{j,\varepsilon})_\varepsilon$ of the operators $A_{j}$, $j=1,2$ by means of their asymptotic expansions (in the sense of Definition~\ref{defn:AE_ph}), that is
\begin{equation}\label{asymptotica}
a_{j}(x,z,\tau,\xi) \ \sim \ \sum_{\mu = 0}^{\infty} b_{j}^{(\mu)}(x,z,\tau,\xi)  \qquad \text{in } \widetilde{S}^{1}_{cl,lsc} (I'_{\theta})
\end{equation}
where the sequence $\{ b_{j}^{(\mu)}\}_{\mu \in \mathbb{N}}$ consists of appropriate elements $b_{j}^{(\mu)} \in S^{1-\mu}_{lsc} (I'_{\theta})$, $j=1,2$ and will be constructed recursively.
Recall that (\ref{asymptotica}) means that there are representatives $(a_{j,\varepsilon})_\varepsilon$ of $a_j$ and $(b_{j,\varepsilon}^{(\mu)})_\varepsilon$ of $b_j^{(\mu)}$ such that for any cut-off $\varphi$ equal to 1 near the origin we have 
\begin{equation}\label{eqn:AE_a_j}
\Bigl(a_{j,\varepsilon} - \sum_{\mu = 0}^{N-1} (1-\varphi) b_{j,\varepsilon}^{(\mu)} \Bigr)_\varepsilon \in S^{1-N}_{lsc} (I_{\theta}'). 
\end{equation}
\begin{proof}[Proof of Theorem~\ref{thm:Fact}]
For the proof we apply a decomposition method similar to \cite[Appendix II]{Kumano-go:81} and \cite[Chapter 23]{Hoermander:3}.
In the following, we will show the desired factorization in the case that the principal symbol of $A_j$ is equal to $(-i \zeta_{j,\varepsilon})_\varepsilon = (\mp i \sqrt{a_\varepsilon})_\varepsilon$, $j=1,2$. The proof for the second possible choice of the principal symbol of $A_j$ is essentially the same.

We will now give a construction scheme for the symbols $a_j$ of the operator $A_j$, $j=1,2$, by means of their asymptotic expansions. To begin with, let $\zeta_{j,\varepsilon} = \pm \sqrt{a_\varepsilon}$, $j=1,2$ and $\varepsilon \in (0,1]$. For $j=1,2$ we set on $I'_{\theta}$
\begin{equation*}
\begin{array}{ccccc}
b_{j,\varepsilon}^{(0)}:= -i\zeta_{j,\varepsilon}, \ && \ a_{j,\varepsilon}^{(1)}:=b_{j,\varepsilon}^{(0)}, \ && \ A_{j,\varepsilon}^{(1)}:= \mbox{Op}(a_{j,\varepsilon}^{(1)})
\end{array}
\end{equation*}
\hfuzz=0.12pt 
where $A_{j,\varepsilon}^{(1)}$ should be thought as the operator whose symbol has the classical asymptotic expansion $a_{j,\varepsilon}^{(1)}$ on $I_{\theta}'$. We emphasize that $A_{j,\varepsilon}^{(1)}$ is intrinsically the restriction of a globally defined operator restricted to the set $I_{\theta}'$ as the symbol can always be multiplied by a generalized cut-off function $(\chi_\varepsilon)_\varepsilon$ of the form as indicated in (\ref{eqn:CutOffFunction}) which is identically 1 on $I_{\theta}'$ (with $\theta_1$ equal $\theta$). Moreover, the function $\varphi$ as in (\ref{eqn:AE_a_j}) cuts off the singularities near $\tau = 0$.

We define $L_{j,\varepsilon}^{(1)}:= \partial_z + A_{j,\varepsilon}^{(1)}$, $j=1,2$, $\varepsilon \in (0,1]$. Using $L_{1,\varepsilon}^{(1)} \frac{1}{\rho_\varepsilon}L_{2,\varepsilon}^{(1)}$ as a first approximation to $L_\varepsilon$ we obtain an error of the following form
\begin{equation}\label{first_approx}
\begin{split}
L_{1,\varepsilon}^{(1)} \frac{1}{\rho_\varepsilon} L_{2,\varepsilon}^{(1)} - L_\varepsilon & = \Big(\partial_z + A_{1,\varepsilon}^{(1)} \Big) \frac{1}{\rho_\varepsilon} \Big(\partial_z + A_{2,\varepsilon}^{(1)}\Big) - \partial_z \frac{1}{\rho_\varepsilon} \partial_z - A_{\rho,\varepsilon} =\\
& = \Big(A_{1,\varepsilon}^{(1)} \frac{1}{\rho_\varepsilon} + \frac{1}{\rho_\varepsilon} A_{2,\varepsilon}^{(1)}\Big) \partial_z + \partial_z \Big(\frac{1}{\rho_\varepsilon} A_{2,\varepsilon}^{(1)}\Big) + A_{1,\varepsilon}^{(1)} \frac{1}{\rho_\varepsilon} A_{2,\varepsilon}^{(1)} - A_{\rho,\varepsilon} =\\
& =: \sum_{j=1}^{2} \Gamma_{j,\varepsilon}^{(1)} \partial_z^{2-j} \qquad \text{ on } I_{\theta}'.
\end{split}
\end{equation}
Since $A_1^{(1)}$ and $-A_2^{(1)}$ have the same principal symbol on $I_{\theta}'$ it follows that the symbols of the operators $\Gamma_{j}^{(1)} = \Gamma_j^{(1)}(x,z,D_t,D_x)$ are in $S^{j-1}_{cl,lsc}$ on $I_{\theta}'$, $j=1,2$. 

To improve this approximation we proceed by induction. For convenience of the reader, we also compute the second order approximation of $L$. Therefore let $\varepsilon \in (0,1]$ be fixed and define
\begin{equation*}
\begin{array}{ccccc}
a_{j,\varepsilon}^{(2)} := a_{j,\varepsilon}^{(1)} + b_{j,\varepsilon}^{(1)}, \ && \ A_{j,\varepsilon}^{(2)} := \mbox{Op}(a_{j,\varepsilon}^{(2)}), \ && \ L_{j,\varepsilon}^{(2)}:= \partial_z + A_{j,\varepsilon}^{(2)}.
\end{array}
\end{equation*}
Here $b_{j,\varepsilon}^{(1)}$ will be specified immediately and in such a way that the symbol of $A_{j,\varepsilon}^{(2)}$ with classical asymptotic expansion $\bigl(a_{j,\varepsilon}^{(2)}\bigr)_\varepsilon$ is in $S^{1}_{cl,lsc}$.
From the above we already know that
\begin{equation*}
L_{1}^{(1)} \frac{1}{\rho} L_{2}^{(1)} - L = \sum_{j=1}^{2} \Gamma_{j}^{(1)} \partial_z^{2-j}.
\end{equation*}
We can formally write on $I_{\theta}'$
\begin{equation*}
\begin{split}
L_{1,\varepsilon}^{(2)} \frac{1}{\rho_\varepsilon} L_{2,\varepsilon}^{(2)} &- L_{\varepsilon} = \Big(L_{1,\varepsilon}^{(1)} + B_{1,\varepsilon}^{(1)} \Big) \frac{1}{\rho_\varepsilon} \Big(L_{2,\varepsilon}^{(1)} + B_{2,\varepsilon}^{(1)} \Big) - L_\varepsilon =\\
& = \sum_{j=1}^{2} \Gamma_{j,\varepsilon}^{(1)} \partial_z^{2-j} + \Big(\partial_z + A_{1,\varepsilon}^{(1)} \Big) \frac{1}{\rho_\varepsilon} B_{2,\varepsilon}^{(1)} + B_{1,\varepsilon}^{(1)} \frac{1}{\rho_\varepsilon} \Big(\partial_z + A_{2,\varepsilon}^{(1)} \Big) + B_{1,\varepsilon}^{(1)} \frac{1}{\rho_\varepsilon} B_{2,\varepsilon}^{(1)} =
\\
& = \sum_{j=1}^{2} \Gamma_{j,\varepsilon}^{(1)} \partial_z^{2-j} + \Big( B_{1,\varepsilon}^{(1)} \frac{1}{\rho_\varepsilon} + \frac{1}{\rho_\varepsilon} B_{2,\varepsilon}^{(1)} \Big) \partial_z + \partial_z \Big(\frac{1}{\rho_{\varepsilon}} B_{2,\varepsilon}^{(1)} \Big) + \\
& \hspace{4cm}+ B_{1,\varepsilon}^{(1)} \frac{1}{\rho_\varepsilon} A_{2,\varepsilon}^{(1)} + A_{1,\varepsilon}^{(1)} \frac{1}{\rho_\varepsilon} B_{2,\varepsilon}^{(1)} + B_{1,\varepsilon}^{(1)} \frac{1}{\rho_\varepsilon}B_{2,\varepsilon}^{(1)}
\end{split}
\end{equation*}
We now specify $\Big(b_{1,\varepsilon}^{(1)}\Big)_\varepsilon$, $\Big(b_{2,\varepsilon}^{(1)}\Big)_\varepsilon$ as follows: as $\Big(a_{j,\varepsilon}^{(1)}\Big)_\varepsilon = \mp \Big(i \sqrt{a_{\varepsilon}}\Big)_\varepsilon$ ($j=1,2$) is non-vanishing on $I_{\theta}'$ the matrix 
\begin{equation*}
\frac{1}{\rho_\varepsilon} \begin{pmatrix} 1 & 1 \\ a_{2,\varepsilon}^{(1)} & a_{1,\varepsilon}^{(1)} \end{pmatrix}
\end{equation*}
is invertible on $I_{\theta}'$. 

Then on $I_{\theta}'$ 
\begin{equation}\label{sol_lot}
-\rho_\varepsilon \frac{1}{a_{1,\varepsilon}^{(1)}-a_{2,\varepsilon}^{(1)}}  \begin{pmatrix} a_{1,\varepsilon}^{(1)} & -1 \\ -a_{2,\varepsilon}^{(1)} & 1 \end{pmatrix} \begin{pmatrix} \gamma_{1,\varepsilon}^{(1)} \\ \gamma_{2,\varepsilon}^{(1)} \end{pmatrix} = \begin{pmatrix} b_{1,\varepsilon}^{(1)} \\ b_{2,\varepsilon}^{(1)} \end{pmatrix}
\end{equation}
where the net $(\rho_\varepsilon)_\varepsilon$ has the same properties as the net $( 1/\rho_\varepsilon)_\varepsilon$.

Recall that $\exists C>0, \exists \eta \in (0,1]$ such that
\begin{equation*}
|a_{1,\varepsilon}^{(1)}(x,z,\tau,\xi)| \ge C |(\tau,\xi)| \quad \text{ on } I'_{\theta}, \ \varepsilon \in (0,\eta]
\end{equation*}
and $\forall \alpha, \beta \in \mathbb{N}^n, \exists (\omega_\varepsilon)_\varepsilon \in \Pi_{lsc}, \forall (x,z,\tau,\xi) \in I'_{\theta}$ 
\begin{equation*}
|\partial^{\alpha}_{(\tau,\xi)} \partial^{\beta}_{(x,z)} a_{1,\varepsilon}^{(1)}(x,z,\tau,\xi)| = \mathcal{O} (\omega_\varepsilon) |(\tau,\xi)|^{1-|\alpha|} \qquad \text{for } |(\tau,\xi)| \ge 1 \text{ as } \varepsilon \to 0.
\end{equation*}
So $\Bigl(a_{1,\varepsilon}^{(1)}{}^{ -1 }\Bigr)_\varepsilon \in S^{-1}_{cl,lsc}$ on $I'_{\theta}$.
Then, modulo  $S^{-\infty}_{lsc}(I_{\theta}')$ there are uniquely determined symbols $\Big(b_{1,\varepsilon}^{(1)}\Big)_\varepsilon$, $\Big(b_{2,\varepsilon}^{(1)}\Big)_\varepsilon \in S^{0}_{lsc}$ on $I_{\theta}'$ homogeneous for $|(\tau,\xi)| \ge 1$ solving the system
\begin{equation*}
\begin{split}
-\gamma_{1,\varepsilon}^{(1)} & = b_{1,\varepsilon}^{(1)} \frac{1}{\rho_\varepsilon} + \frac{1}{\rho_\varepsilon} b_{2,\varepsilon}^{(1)}\\
-\gamma_{2,\varepsilon}^{(1)} & = b_{1,\varepsilon}^{(1)} \frac{1}{\rho_\varepsilon} a_{2,\varepsilon}^{(1)} + a_{1,\varepsilon}^{(1)} \frac{1}{\rho_\varepsilon} b_{2,\varepsilon}^{(1)}
\end{split}
\end{equation*}
on the set $I_{\theta}'$ where $\Big(\gamma_{j,\varepsilon}^{(1)}\Big)_\varepsilon$ denotes the principal symbol of $\Gamma_{j}^{(1)}$, $j=1,2$. Note that $\Bigl(b_{j,\varepsilon}^{(1)}\Bigr)_\varepsilon \in S^{0}_{lsc} (I_{\theta}')$ and is homogeneous for $|(\tau,\xi)| \ge 1$. Since $(a_{j,\varepsilon}^{(1)})_\varepsilon$ is logarithmic slow scale elliptic and $(b_{j,\varepsilon}^{(1)})_\varepsilon$ is of lower order, we conclude that $(a_{j,\varepsilon}^{(2)})_\varepsilon$ is also logarithmic slow scale elliptic.

We write $B_{j,\varepsilon}^{(1)}:= \mbox{Op}(b_{j,\varepsilon}^{(1)})$ and obtain
\begin{equation*}
L_{1}^{(2)} \frac{1}{\rho} L_{2}^{(2)} - L = \sum_{j=1}^{2} \Gamma_{j}^{(2)} \partial_z^{2-j}
\end{equation*}
with $\Gamma_{j}^{(2)} = \Gamma_j^{(2)}(x,z,D_t,D_x)$ having symbols in $S^{j-2}_{cl,lsc}$ on $I_{\theta}'$, $j=1,2$. 

For $N \ge 1$ assume $\Big(b_{j,\varepsilon}^{(\mu)}\Big)_\varepsilon \in S^{1-\mu}_{lsc}(I_{\theta}')$ is homogeneous of order $1-\mu$ for $|(\tau,\xi)| \ge 1$ and determined for all $\mu < N$, $j=1,2$. For $\varepsilon \in (0,1]$ we set
\begin{equation*}
\begin{array}{ccccc}
a_{j,\varepsilon}^{(N)} := \sum_{\mu = 0}^{N-1} b_{j,\varepsilon}^{(\mu)}, \ && \ A_{j,\varepsilon}^{(N)}:= \mbox{Op} ( a_{j,\varepsilon}^{(N)} ), \ && \ L_{j,\varepsilon}^{(N)}:= \partial_z + A_{j,\varepsilon}^{(N)}.
\end{array}
\end{equation*}
Again, here $A_{j}^{(N)}$ is the polyhomogeneous generalized pseudodifferential operator with symbol $\Bigl(a_{j,\varepsilon}^{(N)}\Bigr)_\varepsilon$ in $S^{1}_{cl,lsc}$. Furthermore, we assume
\begin{equation*}
L_{1}^{(N)} \frac{1}{\rho} L_{2}^{(N)} - L = \sum_{j=1}^{2} \Gamma_{j}^{(N)} \partial_z^{2-j}
\end{equation*}
with $\Gamma_{j}^{(N)} = \Gamma_j^{(N)}(x,z,D_t,D_x)$ having symbols in $S^{j-N}_{cl,lsc}$ on $I_{\theta}'$, $j=1,2$. 

For the induction step we specify $b_{1}^{(N)}, b_{2}^{(N)}$ as follows: We denote by $\Big( \gamma_{j,\varepsilon}^{(N)} \Big)_{\varepsilon}$ the top order symbol of $\Gamma_j^{(N)}$. Since $\Big(a_{j,\varepsilon}^{(1)} \Big)_\varepsilon$ is logarithmic slow scale elliptic on $I_{\theta}'$, $j=1,2$, the matrix 
\begin{equation*}
\frac{1}{\rho_\varepsilon} \begin{pmatrix} 1 & 1 \\ a_{2,\varepsilon}^{(1)} & a_{1,\varepsilon}^{(1)} \end{pmatrix}
\end{equation*}
is invertible on $I_{\theta}'$ and the system
\begin{equation*}
\begin{split}
-\gamma_{1,\varepsilon}^{(N)} & = b_{1,\varepsilon}^{(N)} \frac{1}{\rho_\varepsilon} + \frac{1}{\rho_\varepsilon} b_{2,\varepsilon}^{(N)}\\
-\gamma_{2,\varepsilon}^{(N)} & = b_{1,\varepsilon}^{(N)} \frac{1}{\rho_\varepsilon} a_{2,\varepsilon}^{(1)} + a_{1,\varepsilon}^{(1)} \frac{1}{\rho_\varepsilon} b_{2,\varepsilon}^{(N)}
\end{split}
\end{equation*}
is uniquely solvable for $\Big( b_{1,\varepsilon}^{(N)} \Big)_\varepsilon, \Big( b_{2,\varepsilon}^{(N)} \Big)_\varepsilon \in S^{1-N}_{lsc}$ on $I_{\theta}'$. Moreover, $b_{2,\varepsilon}^{(N)}$ are homogeneous for $|(\tau,\xi)| \ge 1$. We write $B_{j,\varepsilon}^{(N)}:= \mbox{Op}(b_{j,\varepsilon}^{N})$ and indeed we get 
\begin{align*}
L_{1,\varepsilon}^{(N+1)} &\frac{1}{\rho_\varepsilon} L_{2,\varepsilon}^{(N+1)} - L_{\varepsilon} = \Big(L_{1,\varepsilon}^{(N)} + B_{1,\varepsilon}^{(N)} \Big) \frac{1}{\rho_\varepsilon} \Big(L_{2,\varepsilon}^{(N)} + B_{2,\varepsilon}^{(N)} \Big) - L_\varepsilon =\\
& = \sum_{j=1}^{2} \Gamma_{j,\varepsilon}^{(N)} \partial_z^{2-j} + \Big(\partial_z + A_{1,\varepsilon}^{(N)} \Big) \frac{1}{\rho_\varepsilon} B_{2,\varepsilon}^{(N)} + B_{1,\varepsilon}^{(N)} \frac{1}{\rho_\varepsilon} \Big(\partial_z + A_{2,\varepsilon}^{(N)} \Big) + B_{N,\varepsilon}^{(1)} \frac{1}{\rho_\varepsilon} B_{2,\varepsilon}^{(N)} =\\
& = \sum_{j=1}^{2} \Gamma_{j,\varepsilon}^{(N)} \partial_z^{2-j} + \Big(B_{1,\varepsilon}^{(N)} \frac{1}{\rho_\varepsilon} + \frac{1}{\rho_\varepsilon} B_{2,\varepsilon}^{(N)} \Big) \partial_z + \partial_z \Big(\frac{1}{\rho_{\varepsilon}} B_{2,\varepsilon}^{(N)} \Big) + \\
& \hspace{4cm}+ B_{1,\varepsilon}^{(N)} \frac{1}{\rho_\varepsilon} A_{2,\varepsilon}^{(N)} + A_{1,\varepsilon}^{(N)} \frac{1}{\rho_\varepsilon} B_{2,\varepsilon}^{(N)} + B_{1,\varepsilon}^{(N)} \frac{1}{\rho_\varepsilon}B_{2,\varepsilon}^{(N)} =\\
& =: \sum_{j=1}^2 \Gamma_{j,\varepsilon}^{(N+1)} \partial_z^{2-j}
\end{align*}
where $\Gamma_{j}^{(N+1)} =  \Gamma_{j}^{(N+1)}(x,z,D_t,D_x)$ have symbols in $ S^{j-(N+1)}_{cl,lsc}$ on $I_{\theta}'$. 

Then with $a_{j}$, $j=1,2$, such that
\begin{equation*}
a_j \sim \sum_{\mu=0}^{\infty} b_{j}^{(\mu)} \quad \text{on } I_{\theta}'
\end{equation*}
we found the desired operator $A_j$ that solves the problem.
\end{proof}
\begin{rem}
Theorem~\ref{thm:Fact} allows for two distinct factorizations, i.e.
\begin{equation*}
L = (\partial_z + A_{11})\frac{1}{\rho}(\partial_z + A_{21}) + R_1 =(\partial_z + A_{12})\frac{1}{\rho}(\partial_z + A_{22}) + R_2 \quad \text{ on } I_{\theta}'.
\end{equation*}
where the principal symbol of $(A_{j1,\varepsilon})_\varepsilon$ is equal to $(\mp i \sqrt{a_\varepsilon})_\varepsilon$, $j=1,2$ and that of $(A_{j2,\varepsilon})_\varepsilon$ is equal to $(\pm i \sqrt{a_\varepsilon})_\varepsilon$, $j=1,2$. $R_1$ and $R_2$ denote the corresponding remainders of the factorizations.

For completeness, we will compute the zeroth-order term for the symbols of $(A_{11,\varepsilon})_\varepsilon$ and $(A_{12,\varepsilon})_\varepsilon$ explicitly as they will be used later in section~\ref{sec:microlocal_diag}. First, we calculate $(A_{11,\varepsilon})_\varepsilon$. Recall from (\ref{first_approx}) that
\begin{equation*}
\gamma_{1,\varepsilon}^{(1)} =  -i \sum_{j=1}^{n-1} \frac{\partial a_{1,\varepsilon}^{(1)}}{\partial {\xi_j}} \frac{\partial \rho_\varepsilon^{-1}}{\partial {x_j}} = - a_\varepsilon^{-1/2} \frac{1}{\rho_\varepsilon^2} \sum_{j=1}^{n-1} \frac{\partial \rho_\varepsilon}{\partial x_j} \xi_j
\end{equation*}
and 
\begin{equation*}
\begin{split}
\gamma_{2,\varepsilon}^{(1)} & \equiv \partial_z \Bigl( \frac{1}{\rho_\varepsilon} a_{2,\varepsilon}^{(1)} \Bigr) - a_{1,\varepsilon}^{(1)} \# (\frac{1}{\rho_\varepsilon} a_{2,\varepsilon}^{(1)}) - a_{\rho,\varepsilon} - i \sum_{j=1}^{n-1} \Bigl( \partial_{x_j} \frac{1}{\rho_\varepsilon} \Bigr) \xi_j \quad \mod S^0_{lsc}.
\end{split}
\end{equation*}
Hence
\begin{equation*}
\begin{split}
\gamma_{2,\varepsilon}^{(1)} 
& = \partial_z \Bigl( \frac{1}{\rho_\varepsilon} a_{2,\varepsilon}^{(1)} \Bigr) - \frac{i}{\rho_\varepsilon} \sum_{j=1}^{n-1} \frac{\partial a_{1,\varepsilon}^{(1)}}{\partial {\xi_j}} \frac{\partial a_{2,\varepsilon}^{(1)}}{\partial {x_j}} = \\
& = \frac{i}{2} \frac{a_{\rho,\varepsilon}^{1/2}}{\rho_\varepsilon^{1/2}} \Bigl (\frac{\partial a_{\rho,\varepsilon}}{\partial z} \frac{1}{a_{\rho,\varepsilon}}  - \frac{\partial \rho_\varepsilon}{\partial z} \frac{1}{\rho_\varepsilon} \Bigr) + \frac{i}{\rho_\varepsilon} \sum_{j=1}^{n-1} \frac{\partial a_{\varepsilon}^{1/2}}{\partial {\xi_j}} \frac{\partial a_{\varepsilon}^{1/2}}{\partial {x_j}}.
\end{split}
\end{equation*}
We set $b_\varepsilon=a_\varepsilon^{1/2}$ and obtain by (\ref{sol_lot})
\begin{equation}\label{A_11}
\begin{split}
a_{11,\varepsilon} = -i b_\varepsilon &+ \frac{1}{2 b_\varepsilon} \sum_{j=1}^{n-1} \frac{\partial b_\varepsilon}{\partial {\xi_j}} \frac{\partial b_\varepsilon}{\partial {x_j}} - \frac{1}{4} \Bigl (\frac{\partial a_{\rho,\varepsilon}}{\partial z} \frac{1}{a_{\rho,\varepsilon}}  - \frac{\partial \rho_\varepsilon}{\partial z} \frac{1}{\rho_\varepsilon} \Bigr) +\\
&+ \frac{1}{2} \frac{1}{\rho_\varepsilon b_\varepsilon}  \sum_{j=1}^{n-1} \xi_j \frac{\partial \rho_\varepsilon}{\partial x_j} + \text{order}(-1). 
\end{split}
\end{equation}
Using the same arguments, we obtain for the operator $A_{12}$ the following expression.
\begin{equation}\label{A_12}
\begin{split}
a_{12,\varepsilon} =\phantom{-}i b_\varepsilon &- \frac{1}{2 b_\varepsilon} \sum_{j=1}^{n-1} \frac{\partial b_\varepsilon}{\partial {\xi_j}} \frac{\partial b_\varepsilon}{\partial {x_j}} - \frac{1}{4} \Bigl (\frac{\partial a_{\rho,\varepsilon}}{\partial z} \frac{1}{a_{\rho,\varepsilon}} -\frac{\partial \rho_\varepsilon}{\partial z} \frac{1}{\rho_\varepsilon} \Bigr) - \\
&- \frac{1}{2} \frac{1}{\rho_\varepsilon b_\varepsilon} \sum_{j=1}^{n-1} \xi_j \frac{\partial \rho_\varepsilon}{\partial x_j} + \text{order}(-1).
\end{split}
\end{equation}
\end{rem}
\section{The generalized wave front set}\label{sec:lscWaveFrontSet}
Given a pseudodifferential equation with smooth symbols, regularity properties can be used to describe the behavior of the solution. To study propagation of singularities the notion of the wave front set was introduced.
We recall that the complement of the classical wave front set of a distribution $u$ is the set $WF(u)^c$ and measures smoothness near a point in the sense that the Fourier transform of a localized piece of $u$ is rapidly decreasing in an open cone.

Another version is the Sobolev-based wave front set, where one studies subsets $\mbox{WF}^m(u)$ of $T^*\mathbb{R}^n \setminus 0$, $m \in \mathbb{N}$ such that
\begin{equation*}
(x_0,\xi_0) \notin \mbox{WF}^m(u) 
\end{equation*} 
if there exists a $P \in \Psi^m$ elliptic at $(x_0,\xi_0)$ such that $Pu \in L^2(\mathbb{R}^n)$. In this sense, wave front sets give a description of local smoothness of a distribution since $\mbox{WF}^m(u)=\emptyset$ iff $u \in H^m_{loc}(\mathbb{R}^n)$. As pointed out in \cite[Section 4]{Wunsch:08}, there is the following relation between these two versions:
\begin{thm}
Let $u \in \mathcal{D}'(\mathbb{R}^n)$. Then 
\begin{equation*}
\mbox{WF}(u) = \overline{\cup_{k} \mbox{WF}^k(u)}.
\end{equation*}
\end{thm}
Relating to the notion of wave front sets to pseudodifferential operators with smooth symbols, it is well known that for any $P \in \Psi^m$ with homogeneous principal symbol $P_m$ one has the following inclusion (on $T^*\mathbb{R}^n \setminus 0$)
\begin{equation*}
\mbox{WF}(u) \subseteq \mbox{WF}(Pu) \cup \text{Char}(P)
\end{equation*}
where $\text{Char}(P)=P^{-1}_m(0) \cap T^*\mathbb{R}^n \setminus 0$ is the characteristic set depending on the principal symbol of the operator. Moreover, pseudodifferential operators do not increase the wave front sets. For more information, we refer the reader to \cite[Section 18]{Hoermander:3}.

In the framework of Colombeau generalized functions in $\GLtwo (\mathbb{R}^n)$, we follow this idea and measure regularity by considering rapid decay on cones in the frequency domain after localization in space. We refer to \cite{NPS:98, GarettoHoermann:05, GaHoe:06, HO:04} for more details on the commonly used notion of a generalized wave front set based on  $\G^{\infty}$-regularity.
\subsection{Generalized Microlocal Analysis}
Recall that a function $u \in \GLtwo (\mathbb{R}^{n})$ is regular, denoted by $u \in \GregLtwo (\mathbb{R}^{n})$, if and only if there exists a representative $(u_\varepsilon)_\varepsilon$ of $u$ such that
\begin{equation*}
\exists N \in \mathbb{N} \ \forall \alpha \in \mathbb{N}^n: \quad \lVert D^{\alpha} u_\varepsilon \rVert_{L^2(\mathbb{R}^n)} = \mathcal{O} (\varepsilon^{-N}) \quad \text{ as } \varepsilon \to 0.
\end{equation*}
As already indicated above, we introduce the following definition which is similar to \cite{GarettoHoermann:05}:
\begin{defn}
A generalized function $u \in \GLtwo(\mathbb{R}^n)$ is said to be microlocally regular at $(x_0,\xi_0) \in T^* \mathbb{R}^n \setminus 0$ if there exist $\phi \in \mathcal{C}^{\infty}_c (\mathbb{R}^n)$ with $\phi(x_0) = 1$ and a conic neighborhood $\Gamma \subseteq \mathbb{R}^n \setminus 0$ of $\xi_0$ such that
\begin{equation}\label{eqn:MicroReg}
\exists N \in \mathbb{N} \ \forall l \in \mathbb{N}: \quad \lVert \langle\xi\rangle^{l} \mathcal{F}(\phi u_\varepsilon) \rVert_{L^2(\Gamma)} = \mathcal{O}(\varepsilon^{-N}) \quad \text{ as } \varepsilon \to 0.
\end{equation}
For $u \in \GLtwo$ we denote by $\mbox{WF}_g(u) \subseteq T^*\mathbb{R}^n \setminus 0$ the generalized wave front set of $u$ which is defined as the complement of all points in phase space where u is microlocally regular.

Moreover, we say that two generalized functions $u, v \in \GLtwo(\mathbb{R}^n)$ are microlocally equivalent at $(x_0,\xi_0) \in T^* \mathbb{R}^n \setminus 0$ if and only if there are $\phi \in \mathcal{C}^{\infty}_c (\mathbb{R}^n)$ with $\phi(x_0) = 1$ and a conic neighborhood $\Gamma \subseteq \mathbb{R}^n \setminus 0$ of $\xi_0$ such that
\begin{equation*}
\exists N \in \mathbb{N} \ \forall l \in \mathbb{N}: \quad \lVert \langle\xi\rangle^{l} \mathcal{F}\big(\phi (u_\varepsilon-v_\varepsilon)\big) \rVert_{L^2(\Gamma)} = \mathcal{O}(\varepsilon^{-N}) \quad \text{ as } \varepsilon \to 0.
\end{equation*}
\end{defn}
As we pointed out in subsection~\ref{subsec:parametrix}, one can construct a generalized parametrix for an logarithmic slow scale elliptic operator $A(x,D) \in \Psi^m_{lsc}$ modulo some regularizing error in $\Psi^{-\infty}_{lsc}$ that maps $\GLtwo$ into $\GregLtwo$. This fact can be deduced by the $L^2$-boundedness theorem for operators with symbol of class $S^0$, see \cite[Theorem 4.1]{Kumano-go:81}.
\begin{rem}
Let $P \in \Psi^{-\infty}_{rg}(\mathbb{R}^n)$, i.e. $p \in \widetilde{S}_{rg}^{-\infty} = \mathcal{M}_{S^{-\infty}}^{\infty} / \mathcal{N}_{S^{-\infty}}$, and $u \in \GLtwo(\mathbb{R}^n)$. Then $Pu$ is in $\GregLtwo(\mathbb{R}^n)$ and hence microlocally regular on $T^* \mathbb{R}^n \setminus 0$. This follows from the fact that the symbol of $D^{\gamma} P(x,D)$ is in $S^{-\infty}_{rg} = \mathcal{M}_{S^{-\infty}}^{\infty} \subseteq S^{0}_{rg}$ uniformly in $\gamma \in \mathbb{N}^n$. More precisely we have the estimate: 
\begin{equation*}
\begin{split}
\exists N \in \mathbb{N} \ & \forall \gamma \in \mathbb{N}^n \ \exists \ l \in \mathbb{N} \ \exists \ c >0:\\
&\lVert D^{\gamma} P_\varepsilon u_\varepsilon \rVert_{L^2} \leq c \max_{|\alpha+\beta| \leq l} \sup_{(x,\xi) \in \mathbb{R}^{2n}} |\partial_{\xi}^{\alpha} \partial_x^{\beta} \Big( \xi^{\gamma} \# p_\varepsilon(x,\xi)\Big) | \langle \xi \rangle^{|\alpha|} \lVert u_\varepsilon \rVert_{L^2} = \mathcal{O}(\varepsilon^{-N}) 
\end{split}
\end{equation*}
as $\varepsilon \to 0$. In particular, an operator $P \in \Psi^{-\infty}_{lsc}$ maps $\GLtwo(\mathbb{R}^n)$ into $\GregLtwo(\mathbb{R}^n)$.
\end{rem}
We now introduce the generalized microsupport of a symbol as follows: Let $p \in S^{m}_{rg}$ and $(x_0,\xi_0) \in T^* \mathbb{R}^n \setminus 0$. Then the symbol $p$ is $\GregLtwo$-smoothing at $(x_0,\xi_0)$ if there exist a representative $(p_\varepsilon)_\varepsilon \in p$, a relatively compact open neighborhood $U$ of $x_0$ and a conic neighborhood $\Gamma$ of $\xi_0$ such that
\begin{equation*}
\begin{split}
\exists N \in \mathbb{N} \ \forall m \in \mathbb{R} \ \forall \alpha, \beta \in \mathbb{N}^n & \ \forall (x,\xi) \in U \times \Gamma: \\
&|\partial_{\xi}^{\alpha} \partial_x^{\beta} p_\varepsilon(x,\xi)| = \mathcal{O}( \varepsilon^{-N}) \langle \xi \rangle^{m-|\alpha|} \quad \text{as } \varepsilon \to 0.
\end{split}
\end{equation*}
The generalized microsupport of $p$, denoted by $\mbox{$\mu$supp}_{g}(p)$, is defined as the complement of the set of points $(x_0,\xi_0)$ where $p$ is $\GregLtwo$-smoothing.

In the sequel we give an essential overview of the concepts in [19]
which will be employed in this section more frequently, referring to [19] for the
proofs of the main results and for further explanations.

In the sequel we will follow the ideas made in \cite{GarettoHoermann:05} where the authors measured regularity of a generalized function in $u \in \G(\Omega)$ and $\Omega$ is some open subset of $\mathbb{R}^n$ when acting on generalized pseudodifferential operators with slow scale symbols. Our definition is (for $u \in \GLtwo$) is now the following:
\begin{defn}
For any $u \in \GLtwo$ we define
\begin{equation*}
W_{\text{sc}}(u) := \underset{{\substack{
   p(x,D) \in \leftidx{_{pr}}{\Psi}{^{0}_{sc}}\\
   p(x,D)u \in \GregLtwo
  }}}{\mathbin{\scalebox{2}{\ensuremath{\cap}}}} \text{Ell}_{sc}(p)^c
\end{equation*}
where $\text{Ell}_{sc}(p)$ denotes the set of points where $p$ is slow-scale elliptic. Note that in the manner described in \cite{GarettoHoermann:05} we have that for any slow scale elliptic symbol $p \in \widetilde{S}^m_{sc}$ there exists a parametrix $q \in \widetilde{S}^{-m}_{rg}$ such that $p \# q = q \# p = 1$ in $\widetilde{S}^{0}_{rg}$.
\end{defn}
As in \cite[Proposition 2.8]{{GarettoHoermann:05}} one can show the following result: Let $\pi : T^*(\mathbb{R}^n) \setminus 0 \to \mathbb{R}^n : (x, \xi) \to x$. Then for any $u \in \GLtwo$ we have
\begin{equation*}
\pi(W_{sc} (u)) = \text{singsupp}_g (u)
\end{equation*}
where $\mathbb{R}^n \setminus \text{singsupp}_{g} (u) := \{ x \in \mathbb{R}^n: \exists U_x \subseteq \mathbb{R}^n \text{ open such that } u|_{U_x} \in \GregLtwo \}$.

Before we proceed, let us briefly recall the three main theorems in \cite[Theorem 3.6, Theorem 3.11, Theorem 4.1]{GarettoHoermann:05} but applied to the symbol classes and function spaces we use. We renounce to give the proofs as they can be obtained by minor changes in the arguments.

\begin{thm}\cite{GarettoHoermann:05}\label{thm:microlocal}
For any $P=p(x,D) \in \leftidx{_{pr}}{\Psi}{^{m}_{rg}}$ and $u \in \GLtwo$ we have
\begin{equation*}
W_{sc}(p(x,D)u) \subseteq W_{sc}(u) \cap \mbox{$\mu$supp}_g(p).
\end{equation*}
\end{thm}
Moreover, we have the following theorem.
\begin{thm}\cite{GarettoHoermann:05}
For $u \in \GLtwo(\mathbb{R}^n)$ we have
\begin{equation}\label{eqn:gen-WF}
\mbox{WF}_g(u) = W_{sc}(u) = 
\underset{\substack{
   P \in \leftidx{_{pr}}{\Psi}{^{0}}\\
   Pu \in \GregLtwo
  }}{\mathbin{\scalebox{2}{\ensuremath{\cap}}}} Char(P)
\end{equation}
\end{thm}
The proof of this theorem follows the same lines as in \cite[Theorem 3.9]{GarettoHoermann:05} but with slight changes. For completeness we give the proof.
\begin{proof}
We first show that if $(x_0,\xi_0) \notin \mbox{WF}_g(u)$ then $(x_0,\xi_0) \notin \mbox{W}_{sc}(u)$. So, suppose that \eqref{eqn:MicroReg} holds. By \cite[Remark 3.5]{GarettoHoermann:05} there exists $p(\xi) \in S^0(\mathbb{R}^n \times \mathbb{R}^n)$ such that $\mbox{\text{supp}}(p) \subseteq \Gamma$, $p=1$ in a conic neighborhood $\Gamma '$ of $\xi_0$, $|\xi| \ge 1$. Taking a typical proper cut-off $\chi$ one can write the properly supported pseudodifferential operator with amplitude $\chi(x,y) p(\xi) \phi(y)$ in the form $\sigma(x,D) \in \leftidx{_{pr}}{\Psi}{^{0}}_{sc}(\mathbb{R}^n)$ such that $\sigma(x,D) v - p(D)(\phi v) \in \GregLtwo(\mathbb{R}^n)$ for all $v \in \GLtwo$.   Then
\begin{equation*}
p(D) \phi(x,D) u = \left[ \left( \int_{\Gamma} e^{ix \xi} p(\xi) \widehat{\phi u_\varepsilon}(\xi) \, \dbar \xi \right)_{\! \varepsilon}  \right] \in \GregLtwo
\end{equation*}
Since $\sigma(x,D)u - p(D) \phi u \in \GregLtwo$ it follows that $\sigma(x,D)u \in \GregLtwo$ and hence $(x_0,\xi_0) \notin \mbox{W}_{sc}(u)$.

To show the converse assume that $(x_0,\xi_0) \notin \mbox{W}_{sc}(u)$. Then there exists an open neighborhood $U$ of $x_0$ such that $(x, \xi_0) \in \mbox{W}_{sc}(u)^c$ for all $x \in U$. We choose $\phi \in \mathcal{C}_c^{\infty} (U)$ and define the closed conic set $\Sigma'$ by
\begin{equation*}
\Sigma' := \{ \xi \in \mathbb{R}^n \setminus 0 \ | \ \exists x \in \mathbb{R}^n : (x, \xi) \in \mbox{W}_{sc}(\phi u) \}.
\end{equation*}
By Theorem~\ref{thm:microlocal} we obtain $\mbox{W}_{sc}(\phi u) \subseteq \mbox{W}_{sc}(u) \cap ( \mbox{supp}(\phi) \times \mathbb{R}^n )$. Since there is a $p \in S^0(\mathbb{R}^n \times \mathbb{R}^n)$, $0 \le p \le 1$ such that $p=1$ in a conic neighborhood $\Gamma'$ of $\xi_0$, $|\xi| \ge 1$ and $p(\xi)=0$ in a conic neighborhood $\Sigma_0'$ of $\Sigma'$ we get that $\mbox{$\mu$supp}(p) \subseteq \mathbb{R}^n \times \Sigma_0'$ and $\mbox{W}_{sc}( \phi u) \subseteq \mathbb{R}^n \times \Sigma'$. Therefore $\mbox{W}_{sc}(p(D) \phi u) \subseteq \mbox{W}_{sc}( \phi u) \cap \mbox{$\mu$supp}(p) = \emptyset$. Hence $p(D)\phi u \in \GregLtwo$ and 
\begin{equation*}
\int e^{i x \xi}p(\xi) \widehat{\phi u_\varepsilon}(\xi) \, \dbar \xi = \left( \phi u_\varepsilon \ast \check{p} \right)(x) \in \mathcal{M}_{H^\infty}^{\infty}.
\end{equation*}
Since $\check{p}$ is a Schwartz function outside the origin we distinguish between two cases. First, assume that $\mbox{dist}(x,\mbox{supp}(\phi)) > \delta$ for some $\delta >0$. Then we have that $\partial^{\alpha} (\phi u_\varepsilon \ast \check{p})(x) = \int_{|y| > \delta} \phi u_\varepsilon(x-y) \partial^{\alpha} \check{p} (y) \, d y$ and therefore for every $l \ge 0$:
\begin{equation*}
\begin{split}
\langle x \rangle^l |\partial^{\alpha} (\phi u_\varepsilon \ast \check{p})(x)| & \le c \int_{|y| \ge \delta} |\phi u_\varepsilon(x-y)| \langle x-y \rangle^l \langle y \rangle^l |\partial^{\alpha} \check{p} (y)| \, d y \le \\
& \le c \int_{\mbox{\tiny{supp}}(\phi)} |\phi u_\varepsilon(z)| \langle z \rangle^l \le \\
& \le c \lVert \phi(z) \langle z \rangle^l \rVert_{L^2(\mbox{\tiny{supp}}(\phi))} \lVert u_\varepsilon \rVert_{L^2(\mbox{\tiny{supp}}(\phi))}
\end{split}
\end{equation*}
where in the last inequality we used the H\"{o}lder inequality. Note that there is an $N \in \mathbb{N}$ such that for every $l\ge 0$ and every $\alpha \in \mathbb{N}^n$ the last expression is uniformly bounded by $\mathcal{O}(\varepsilon^{-N})$ as $\varepsilon \to 0$.

Therefore we have: $\exists N \in \mathbb{N}, \forall l \ge 0, \forall \alpha \in \mathbb{N}^n$
\begin{equation*}
\begin{split}
\lVert \langle x & \rangle^l \partial^{\alpha} (\phi u_\varepsilon \ast \check{p})(x) \rVert_{ L^2(\{ {\scriptscriptstyle x} : \mbox{\tiny{dist}}({\scriptscriptstyle x},\mbox{\tiny{supp}}(\phi)) > \delta \})} = \\
& = \mathcal{O}(\varepsilon^{-N} ) \lVert \langle x \rangle^l \partial^{\alpha} (\phi u_\varepsilon \ast \check{p})(x) \rVert_{L^{\infty}(\{ {\scriptscriptstyle x}: \mbox{\tiny{dist}}({\scriptscriptstyle x},\mbox{\tiny{supp}}(\phi)) > \delta \})}  \int \langle x \rangle^{-n-1}
 \, d x = \\ 
 & = \mathcal{O}(\varepsilon^{-2N}) \qquad \text{as }\varepsilon \to 0.
\end{split}
\end{equation*}

Next, consider the compact set $\mbox{dist}(x,\mbox{supp}(\phi)) \le \delta$. Since $(\phi u_\varepsilon \ast \check{p})(x) \in \mathcal{M}_{H^{\infty}}^{\infty}$ we obtain $\exists N \in \mathbb{N}, \forall l \ge 0 , \forall \alpha \in \mathbb{N}^n$:
\begin{equation*}
\lVert \langle x \rangle^l \partial^{\alpha} (\phi u_\varepsilon \ast \check{p})(x) \rVert_{ L^2(\{ {\scriptscriptstyle x} : \mbox{\tiny{dist}}({\scriptscriptstyle x},\mbox{\tiny{supp}}(\phi)) \le \delta \})} = \mathcal{O}(\varepsilon^{-2N})
\end{equation*}
as $\varepsilon \to 0$.

Putting this together we get $\exists N \in \mathbb{N}, \forall \alpha \in \mathbb{N}^n$, so that
\begin{equation*}
\lVert \partial^{\alpha} (\phi u_\varepsilon \ast \check{p})(x) \rVert_{ L^2(\mathbb{R}^n)} = \mathcal{O}(\varepsilon^{-N}) \quad \text{as } \varepsilon \to 0.
\end{equation*}
Taking the Fourier transform we obtain that $\exists N \in \mathbb{N}, \forall \alpha \in \mathbb{N}^n$
\begin{equation*}
\lVert \langle \xi \rangle^{|\alpha|} (\phi u_\varepsilon \ast \check{p}) \widehat{\phantom{x}} (\xi) \rVert_{ L^2(\mathbb{R}^n)} = \mathcal{O}(\varepsilon^{-N}) \quad \text{as } \varepsilon \to 0.
\end{equation*}
Hence
\begin{equation*}
\lVert \langle \xi \rangle^{|\alpha|} p(\xi) \widehat{\phi u_\varepsilon}(\xi) \rVert_{ L^2(\mathbb{R}^n)} = \mathcal{O}(\varepsilon^{-N}) \quad \text{as } \varepsilon \to 0.
\end{equation*}
In particular, since $p(\xi) = 1$ in a neighborhood $\Gamma$ of $\xi_0$, $|\xi| \ge 1$, we get: $\exists N \in \mathbb{N}, \forall \alpha \in \mathbb{N}^n$
\begin{equation*}
\lVert \langle \xi \rangle^{|\alpha|} \widehat{\phi u_\varepsilon}(\xi) \rVert_{ L^2(\Gamma)} = \mathcal{O}(\varepsilon^{-N}) \text{ as } \varepsilon \to 0,
\end{equation*}
as desired.
\end{proof}

%
\begin{thm}\cite{GarettoHoermann:05}
Let $P=p(x,D) \in \leftidx{_{pr}}{\Psi}{^{m}_{sc}}$ and $u \in \GLtwo$. Then 
\begin{equation*}
\mbox{WF}_{g}(Pu) \subseteq \mbox{WF}_{g}(u) \subseteq \mbox{WF}_{g}(Pu) \cup \text{Ell}_{\text{sc}}(p)^c.
\end{equation*}
\end{thm}
In particular, this relation holds for any $P \in \leftidx{_{pr}}{\Psi}^m_{lsc}$.
\section{Microlocal decomposition of the wave equation}\label{sec:microlocal_diag}
The purpose of this section is to establish a microlocal diagonalization of the operator $L$. In what follows, we will restrict our analysis to the following subset $I_{\theta_2}$ of the phase space
\begin{equation*}
I_{\theta_2} := \{ (t,x,z,\tau,\xi;\zeta) \in T^*\mathbb{R}^{n+1} \setminus 0 \ | \ (x,z,\tau,\xi) \in I_{\theta_2}', \ |\zeta| < C |\tau| \},
\end{equation*}
assuming that $WF_g(U) \subseteq I_{\theta_2}$ (cf. \cite{Stolk:04}). Recall that if $LU = 0$, then by section~\ref{sec:ell} and \ref{sec:lscWaveFrontSet} we have that  $\mbox{WF}_{g}(U) \subseteq \text{Ell}_{\text{sc}}(L)^c = \text{Ell}_{\text{lsc}}(L)^c = \Sigma$. The inequality $|\zeta| \le \frac{1}{c^*}(x,z)|\tau|$ on $\Sigma$ implies $|\zeta| \le c_0^{-1} |\tau|$ on $\Sigma$ and explains the inequality $|\zeta| < C |\tau|$ on $I_{\theta_2}$. The microlocal diagonalization will then be stated on the set $I_{\theta_2}$. 

In order to decompose $LU=F$ microlocally on $I_{\theta_2}$ into a system of two first-order components, we will consider $(u_+,u_-)$ that are obtained from $(U,\frac{1}{\rho} \partial_z U)$ by an logarithmic slow scale elliptic 2$ \times$2 pseudodifferential operator matrix $Q = Q(x,z,D_t,D_x)$, i.e. there exists $P = P(x,z,D_t,D_x)$ the generalized parametrix such that $PQ = QP = I$ modulo an operator with symbol in $S^{-\infty}_{lsc}$. Then, with the change of variables
\begin{equation}\label{eqn:wavefields}
\begin{pmatrix} u_{+} \\ u_{-} \end{pmatrix} := Q^{-1} \begin{pmatrix} U \\ \frac{1}{\rho} \partial_z U \end{pmatrix} , \hspace{25pt} \begin{pmatrix} f_{+} \\ f_{-} \end{pmatrix} := Q^{-1} \begin{pmatrix} 0 \\ F \end{pmatrix}
\end{equation}
we search for an equivalent model to the equation
\begin{equation*}
LU=F  \hspace{35pt} \mbox{microlocally on} \ I_{\theta_2}
\end{equation*}
with $U,F \in \GLtwo(\mathbb{R}^{n+1})$ in terms of two first-order generalized pseudodifferential equations of the form 
\renewcommand{\arraystretch}{1.6}
\begin{equation}\label{eqn:DecoupledWaves}
\begin{array}{ccc}
\big( \partial_z - i B_{+}(x,z,D_t,D_x) \big) u_{+} \hspace{-5pt} & = & \hspace{-4pt} f_{+} \hspace{25pt} \mbox{microlocally on} \hspace{2pt} I_{\theta_2}\\
\big( \partial_z - i B_{-}(x,z,D_t,D_x) \big) u_{-} \hspace{-5pt} & = & \hspace{-4pt} f_{-} \hspace{25pt} \mbox{microlocally on} \hspace{2pt} I_{\theta_2},
\end{array}
\end{equation}
where $u_{\pm}, f_{\pm} \in \GLtwo(\mathbb{R}^{n+1})$ and $B_{\pm} = B_{\pm}(x,z,D_t,D_x)$ are logarithmic slow scale pseudodifferential operators of order 1. Note that the operators $B_{\pm}$ are acting in $(t,x)$ and depend on the parameter $z$. Moreover, we will show that there is a choice of the normalization $Q$ of the wave field such that the operators $B_\pm$ become self-adjoint.

We start by rewriting the homogeneous equation $LU = 0$ in $\GLtwo$ into a first-order system with evolution parameter $z$: 
\begin{equation} \label{eqn:Matrix}
\bigg[ I \partial_z - \begin{pmatrix} 0 & \rho  \\ -A_{\rho} & 0  \end{pmatrix} \bigg] \begin{pmatrix} U \\ \frac{1}{\rho} \partial_z U \end{pmatrix} = \vec{0}
\end{equation}
with $U \in \GLtwo(\mathbb{R}^{n+1})$, $I$ the 2$\times$2 identity matrix and 
\begin{equation*}
A_{\rho} = A_{\rho}(x,z,D_t,D_x) = -\frac{1}{\rho} \frac{1}{c^2} \partial^2_t + \sum_{j=1}^{n-1} \partial_{x_j} \frac{1}{\rho} \partial_{x_j}.
\end{equation*}
Notice that $A_{\rho}$ is a logarithmic slow scale elliptic pseudodifferential operator of order 2 on $I_{\theta_2}'$. For brevity, we will drop the identity matrix $I$ in equation (\ref{eqn:Matrix}).

\begin{lem}\label{lem:Diag}
For suitable chosen $Q$ and $B_{\pm}$ we can write
\begin{equation}\label{eqn:DecoupleSystem}
P \bigg[ \partial_z - \begin{pmatrix} 0 & \rho  \\ -A_\rho & 0  \end{pmatrix} \bigg] Q = \partial_z - \begin{pmatrix} i B_{+} & 0 \\ 0 & i B_{-} \end{pmatrix} + R 
\end{equation}
where $R=R(x,z,D_t,D_x)$ is a 2$\times$2 pseudodifferential operator matrix with entries in $\widetilde{S}^{1}_{lsc}$ and that are of order $-\infty$ on $I_{\theta_2}'$. 
\end{lem}
For the proof we introduce the following notation.
Let $R$ be a pseudodifferential operator valued 2$\times$2 error-matrix with entries of the form 
\begin{equation}\label{eqn:Error1}
\sum_{j=1}^2 R_j(x,z,D_t,D_x) \partial_z^{2-j}
\end{equation}
and the symbol of $R_j = R_j(x,z,D_t,D_x)$ is in $\widetilde{S}^{-\infty}_{lsc}$, $j = 1,2$.  In the following we will sometimes write $\Psi^{-\infty}_{1,lsc}$ to denote an operator of the form (\ref{eqn:Error1}). Similarly we will write $\Psi_{2,lsc}^{-\infty}$ for an operator of the form
\begin{equation*}
\sum_{j=0}^2 R_j(x,z,D_t,D_x) \partial_z^{2-j}
\end{equation*}
with $R_j= R_j(x,z,D_t,D_x)$ having symbols in $S^{-\infty}_{lsc}$ for $j = 0,1,2$.
\begin{proof}
We will now search for appropriate choices of the operators $P,Q,R$ and $B_{\pm}$, where $Q,P,R$ are as above and $B_{\pm}$ as already indicated in (\ref{eqn:wavefields}) and (\ref{eqn:DecoupledWaves}) such that equation (\ref{eqn:DecoupleSystem}) holds on $I_{\theta_2}'$, i.e. $R$ is of order $-\infty$ on $I_{\theta_2}'$.

To start with, we set $P = \begin{pmatrix} P_{11} & P_{12} \\ P_{21} & P_{22} \end{pmatrix}$. We first apply the left-hand side of (\ref{eqn:DecoupleSystem}) to $P \begin{pmatrix} 1 \\ \frac{1}{\rho} \partial_z  \end{pmatrix}$ and obtain 
\begin{equation}\label{three1}
P \begin{pmatrix} \partial_z & -\rho \\ A_\rho & \partial_z \end{pmatrix} QP \begin{pmatrix} 1 \\ \frac{1}{\rho} \partial_z \end{pmatrix} = \begin{pmatrix} P_{12} (\partial_z \frac{1}{\rho} \partial_z + A_\rho) + S^{(1,+)} \\ P_{22} (\partial_z \frac{1}{\rho} \partial_z + A_\rho) + S^{(1,-)} \end{pmatrix}
\end{equation}
on $I_{\theta_2}'$, with $S^{(1,\pm)}$ in $\Psi^{-\infty}_{2,lsc}$ on $I_{\theta_2}'$. Similarly, we compute for the right-hand side
\begin{multline}\label{three2}
\bigg[ \begin{pmatrix} \partial_z \hspace{-2pt} - \hspace{-1pt} i B_{+} & 0 \\ 0 & \partial_z \hspace{-2pt} - \hspace{-1pt} i B_{-} \end{pmatrix} + R \bigg] P \begin{pmatrix} 1 \\ \frac{1}{\rho} \partial_z \end{pmatrix} = \\ = \begin{pmatrix} \enspace (\partial_z \hspace{-1pt} - \hspace{-1pt} iB_{+})(P_{12} \frac{1}{\rho} \partial_z + P_{11}) + S^{(2,+)} \enspace \\ (\partial_z \hspace{-2pt} - \hspace{-1pt} iB_{-})(P_{22} \frac{1}{\rho} \partial_z + P_{21}) + S^{(2,-)} \end{pmatrix} \quad \mbox{on} \ I_{\theta_2}'
\end{multline}
where $S^{(2,\pm)} \in \Psi^{-\infty}_{2,lsc}$ on $I_{\theta_2}'$. Summarizing this, allows us to write
\begin{align*}
\begin{array}{cccc}
P_{12} (\partial_z \frac{1}{\rho} \partial_z + A_\rho) \hspace{-5pt} &=  (\partial_z - iB_+)(P_{12} \frac{1}{\rho} \partial_z + P_{11}) + \displaystyle \sum_{j=0}^2 R_j^{(+)} \partial_z^{2-j}\\
P_{22} (\partial_z \frac{1}{\rho} \partial_z + A_\rho) \hspace{-5pt} &= (\partial_z - iB_-)(P_{22} \frac{1}{\rho} \partial_z + P_{21}) + \displaystyle \sum_{j=0}^2 R_j^{(-)} \partial_z^{2-j}
\end{array}
\end{align*}
on $I_{\theta_2}'$ where $R_j^{(\pm)} = R_j^{(\pm)}(x,z,D_t,D_x)$ have symbols in $S_{lsc}^{-\infty}$ on $I_{\theta_2}'$, $j=0,1,2$. 

Otherwise, in view of Theorem~\ref{thm:Fact}, the operator $L=\partial_z \frac{1}{\rho} \partial_z + A_\rho$ can be written in the following form
\begin{equation}\label{fac_1.1}
L = \big( \partial_z + A_{11} \big) \frac{1}{\rho} \big( \partial_z + A_{21} \big) + \sum_{j=1}^{2} R_{j,1}^{(\infty)} \partial_z^{2-j} \quad \mbox{on} \ I_{\theta_2}'
\end{equation}
with the symbols of $A_{j1} = A_{j1}(x,z,D_t,D_x)$ in $S^1_{lsc}$ on $I_{\theta_2}'$, $j=1,2$. Here $A_{11}$ and $-A_{21}$ have the same principal symbol equal to $-i (\sqrt{a_\varepsilon})_\varepsilon$, where $a_\varepsilon := \frac{\tau^2}{c_\varepsilon^2(x,z)} - |\xi|^2$ is the principal symbol of $\rho A_\rho$. Also we obtain
\begin{equation}\label{fac_1.2}
L = \big( \partial_z + A_{12} \big) \frac{1}{\rho} \big( \partial_z + A_{22} \big) + \sum_{j=1}^{2} R_{j,2}^{(\infty)} \partial_z^{2-j} \quad \mbox{on} \ I_{\theta_2}'
\end{equation}
where the symbols of $A_{j2} = A_{j2}(x,z,D_t,D_x)$ are in $S^1_{lsc}$ on $I_{\theta_2}'$, $j=1,2$. At this point, $A_{12}$ and $-A_{22}$ have the same principal symbol equal to $i (\sqrt{a_\varepsilon})_\varepsilon$. Expansion of the right-hand side of (\ref{fac_1.1}), respectively (\ref{fac_1.2}), results in
\begin{equation*}
\partial_z \frac{1}{\rho} \partial_z + A_\rho = \partial_z \frac{1}{\rho} \partial_z + \big( A_{1k} \frac{1}{\rho} + \frac{1}{\rho} A_{2k} \big) \partial_z
 + \partial_z \big( \frac{1}{\rho} A_{2k}\big) + A_{1k} \frac{1}{\rho} A_{2k} + \sum_{j=1}^{2} R_{j,k}^{(\infty)} \partial_z^{2-j}
\end{equation*}
on $I_{\theta_2}'$, $k=1,2$. Equating coefficients gives $A_{2k} = -\rho A_{1k} \frac{1}{\rho}$ modulo an operator $S^{-\infty}_{lsc}$ on $I_{\theta_2}'$, $k=1,2$. Using this, equations (\ref{fac_1.1}) and (\ref{fac_1.2}) now read
\begin{align}
\partial_z \frac{1}{\rho} \partial_z + A_\rho \hspace{-3pt} &= \big( \partial_z + A_{11} \big) \big( \frac{1}{\rho} \partial_z - A_{11} \frac{1}{\rho} \big) \hspace{8pt} \pmod{\Psi_{1,lsc}^{-\infty}} \tag{7.9'}\label{fac_2.1}\\
& = \big( \partial_z + A_{12} \big) \big( \frac{1}{\rho} \partial_z  - A_{12} \frac{1}{\rho} \big) \hspace{8pt} \pmod{\Psi_{1,lsc}^{-\infty}}\tag{7.10'}\label{fac_2.2}
\end{align}
on $I_{\theta_2}'$. At this point we choose $P_{12}$ and $P_{22}$ such that the requirements of Theorem~\ref{thm:parametrix} are satisfied on $I_{\theta_2}'$ and we denote by $P_{12}^{-1}$ and ${P}_{22}^{-1}$ the corresponding parametrixes. Inserting ${P}_{12}^{-1} P_{12}$ into (\ref{fac_2.1}) and ${P}_{22}^{-1} P_{22}$ into (\ref{fac_2.2}) yields
\begin{eqnarray*}
\partial_z \frac{1}{\rho} \partial_z + A_\rho \hspace{-3pt} & = & \hspace{-3pt} \big( \partial_z + A_{11} \big) {P}_{12}^{-1} \big( P_{12} \frac{1}{\rho} \partial_z - P_{12} A_{11} \frac{1}{\rho} \big) \hspace{8pt} \pmod{\Psi_{2,lsc}^{-\infty}}\\
\hspace{-3pt} & = & \hspace{-3pt} \big( \partial_z + A_{12} \big) {P}_{22}^{-1} \big( P_{22} \frac{1}{\rho} \partial_z - P_{22} A_{12} \frac{1}{\rho} \big) \hspace{8pt} \pmod{\Psi_{2,lsc}^{-\infty}}.
\end{eqnarray*}
on $I_{\theta_2}'$. We define
\renewcommand{\arraystretch}{1.6}
\begin{equation}\label{eqn:solutionsB}
\begin{array}{ccc}
\displaystyle -iB_{+}  \hspace{-5pt} & := & \hspace{-3pt} P_{12} A_{11} {P}_{12}^{-1} + P_{12} \partial_z \big( {P}_{12}^{-1} \big) \qquad \ \text{ on} \ I_{\theta_2}'\\
\displaystyle -iB_{-}  \hspace{-5pt} & := & \hspace{-3pt} P_{22} A_{12} {P}_{22}^{-1} + P_{22} \partial_z \big( {P}_{22}^{-1} \big) \qquad \ \text{ on} \ I_{\theta_2}'
\end{array}
\end{equation}
to denote the uniquely determined solutions to 
\begin{eqnarray*}
\big( \partial_z + A_{11} \big) {P}_{12}^{-1} \hspace{-3pt} &=& \hspace{-3pt} {P}_{12}^{-1} \big( \partial_z - i B_+ \big) \\
\big( \partial_z + A_{12} \big) {P}_{22}^{-1} \hspace{-3pt} &=& \hspace{-3pt} {P}_{22}^{-1} \big( \partial_z - i B_- \big),
\end{eqnarray*}
modulo logarithmic slow scale smoothing operators on $I_{\theta_2}'$. Consequently, $B_{\pm}$ are operators with symbols in $S^1_{lsc}$ on $I_{\theta_2}'$ with real principal symbols. Note that $B_{\pm}$ are only prescribed on $I_{\theta_2}'$.
Thus
\begin{eqnarray*}
\partial_z \frac{1}{\rho} \partial_z + A_\rho \hspace{-3pt} & = & \hspace{-3pt} {P}_{12}^{-1} \big( \partial_z - iB_{+} \big) \big( P_{12} \frac{1}{\rho} \partial_z - P_{12} A_{11} \frac{1}{\rho} \big) \hspace{8pt} \pmod{\Psi_{2,lsc}^{-\infty}}\\
\hspace{-3pt} & = & \hspace{-3pt} {P}_{22}^{-1} \big( \partial_z - iB_{-} \big) \big( P_{22} \frac{1}{\rho} \partial_z - P_{22} A_{12} \frac{1}{\rho} \big) \hspace{8pt} \pmod{\Psi_{2,lsc}^{-\infty}}
\end{eqnarray*}
on $I_{\theta_2}'$. Further, if we choose for $P$ from above $P_{11}:= -P_{12} A_{11} \frac{1}{\rho}$ and $P_{21}:= -P_{22} A_{12} \frac{1}{\rho}$ we obtain
\begin{equation}\label{eqn:solutionQ}
P = \begin{pmatrix} -P_{12} A_{11} \frac{1}{\rho} & P_{12} \\ -P_{22} A_{12} \frac{1}{\rho} & P_{22} \end{pmatrix} = \begin{pmatrix} P_{12} & 0 \\ 0 & P_{22} \end{pmatrix} \begin{pmatrix} -A_{11} \frac{1}{\rho} & 1 \\ -A_{12} \frac{1}{\rho} & 1 \end{pmatrix}  \qquad \text{on } I'_{\theta_2}
\end{equation}
where $P_{12}$ and $P_{22}$ are logarithmic slow scale elliptic pseudodifferential operators (in the sense of Theorem~\ref{thm:parametrix}) on $I_{\theta_2}'$. 

Since the principal symbol of $A_{11}$ is $-i(\sqrt{a_\varepsilon})_\varepsilon$ and the principal symbol of $A_{12}$ is $i(\sqrt{a_\varepsilon})_\varepsilon$, it follows that $P$ is logarithmic slow scale elliptic (in the above sense) whenever $P_{12}$ and $P_{22}$ satisfy the requirements of Theorem~\ref{thm:parametrix}. Furthermore, the approximative inverse matrix $Q$ of $P$ is given by
\begin{equation*}
Q = \begin{pmatrix} -C & C \\ -A_{12} \frac{1}{\rho} C & 1 + A_{12} \frac{1}{\rho} C \end{pmatrix} \begin{pmatrix} {P}_{12}^{-1} & 0 \\ 0 & {P}_{22}^{-1} \end{pmatrix} \qquad \text{on } I'_{\theta_2}
\end{equation*}
where $C$ is the generalized parametrix of $A_{11} \frac{1}{\rho} - A_{12} \frac{1}{\rho}$ in the sense of Theorem~\ref{thm:parametrix}. We have therefore found appropriate operator-valued matrices $P,Q,R$ and operators $B_{\pm}$ solving (\ref{eqn:DecoupleSystem}) on $I_{\theta_2}'$.

Outside $I_{\theta_2}'$, we choose $P$ logarithmic slow scale elliptic and of the same order as in (\ref{eqn:solutionQ}) on $I_{\theta_2}'$. Then Lemma~\ref{lem:Diag} is satisfied also on the complement of $I_{\theta_2}'$.
\end{proof}
We note that the derived one-way wave equations are not unique. Only the principal symbol remains the same for any choice of $P_{12}$ and $P_{22}$. Recall that the principal symbol of the operator $B_\pm$ is directly related to the generalized wave front set of the full wave equation. In the smooth setting it turns out, that in order to also describe the wave amplitudes, the subprincipal symbols of $B_\pm$ are needed (cf. \cite{Zhang:03}, \cite{OptRootStolk:10}).

Moreover, by (\ref{A_11}), (\ref{A_12}) and (\ref{eqn:solutionsB}) the zeroth-order terms of $B_\pm$ also depend on the partial derivatives with respect to $z$ of the coefficients, i.e. $\partial_z c_\varepsilon(x,z)$ and $\partial_z \rho_\varepsilon(x,z)$. In the following, we will show that for an appropriate choice of the operator $Q$ such zeroth-order terms can be eliminated. In particular, we have the following result.
\begin{lem}\label{lem:selfadj_B}
Let $Q$, $P$, $B_\pm$ and $R$ as in Lemma~\ref{lem:Diag}. Then there is a choice of $Q$ such that the operators $B_\pm$ become self-adjoint on $I_{\theta_2}'$. Moreover, the symbols of the operators $B_\pm$ are of the form
\begin{equation}\label{eqn:solutionsBselfadj}
b_\pm(x,z,\tau,\xi) = \pm \ \bigg(b + \frac{i}{2b} \sum_{j=1}^{n-1} \frac{\partial b}{\partial \xi_j} \frac{\partial b}{\partial x_j} \bigg) + \text{order}(-1) \qquad \ \text{ on} \ I_{\theta_2}'.
\end{equation}
\end{lem}
\begin{proof}
In the following, we will restrict ourself to the set $I_{\theta_2}'$ and show that $B_{\pm}$ can be chosen self-adjoint. The result follows by an argument used in \cite{Stolk:04}.

We will tackle the problem by choosing generalized classical pseudodifferential operators $P_{12}, P_{22}$ such that 
\begin{equation}\label{eqn:selfadj}
Q^{-1} = P = \begin{pmatrix} 1 & 0  \\ 0 & -1  \end{pmatrix} Q^{\ast} \begin{pmatrix} 0 & -i  \\ i & 0  \end{pmatrix}
\end{equation}
is valid modulo a smoothing operator matrix and where $Q^*$ denotes the adjoint of $Q$. We therefore explicitly compute the right hand side and obtain
\begin{equation*} 
\begin{pmatrix} 1 & 0  \\ 0 & -1  \end{pmatrix} Q^{\ast} \begin{pmatrix} 0 & -i  \\ i & 0  \end{pmatrix} = 
\begin{pmatrix} -i (P_{12}^{-1})^* C^* \frac{1}{\rho} A_{12}^* & i (P_{12}^{-1})^* C^*  \\ -i (P_{22}^{-1})^*-i (P_{22}^{-1})^* C^* \frac{1}{\rho} A_{12}^* & i (P_{22}^{-1})^* C^*  \end{pmatrix}.
\end{equation*}
Then (\ref{eqn:selfadj}) is equivalent to the following two systems of equations
\begin{align*}
& \begin{cases}  -P_{12} A_{11} \frac{1}{\rho} = -i (P_{12}^{-1})^* C^* \frac{1}{\rho} A_{12}^*  \\ \phantom{-}P_{12} = i (P_{12}^{-1})^* C^* \end{cases} \\
\intertext{and}
& \begin{cases}  -P_{22} A_{12} \frac{1}{\rho} = -i (P_{22}^{-1})^* - i (P_{22}^{-1})^* C^* \frac{1}{\rho} A_{12}^*  \\ \phantom{-}P_{22} = i (P_{22}^{-1})^* C^*. \end{cases}
\end{align*}
These systems are equivalent to solving the following equations
\begin{eqnarray*}
i ( A_{11} \frac{1}{\rho} - \frac{1}{\rho}A_{11}^* ) = P_{12}^{-1}(P_{12}^{-1})^* && -i (A_{12} \frac{1}{\rho} - \frac{1}{\rho}A_{12}^*) = P_{22}^{-1}(P_{22}^{-1})^*
\end{eqnarray*}
where we used the fact that $\frac{1}{\rho} A_{11}^* - \frac{1}{\rho}A_{12}^* = (C^{-1})^* = (C^*)^{-1}$. 

As a next step, we will choose $P_{12}^{-1}$ and $P_{22}^{-1}$ equal to the self-adjoint square root of $i ( A_{11} \frac{1}{\rho} - \frac{1}{\rho}A_{11}^* )$ and $-i (A_{12} \frac{1}{\rho} - \frac{1}{\rho}A_{12}^*)$ respectively. We show that $P_{12}^{-1}$ can be chosen in such a way. The result for $P_{22}^{-1}$ follows by the same method. We first note that the operator $T:=i( A_{11} \frac{1}{\rho} - \frac{1}{\rho}A_{11}^* )$ is self-adjoint. As we are searching for the square root, we are interested in an asymptotic expansion $X = \sum_{j=0}^{\infty} X^{(j)}$ with $X^{(j)} \in \Psi_{lsc}^{1/2-j}$ with $X \sim \pm T^{1/2}$ in $\widetilde{S}_{cl,lsc}^1$. 

Therefore, let $X^{(0)}_\varepsilon := \bigl( (b_\varepsilon \frac{1}{\rho_\varepsilon})^{1/2} + (b_\varepsilon \frac{1}{\rho_\varepsilon})^{1/2})^* \bigr)/2$ be self-adjoint in $S_{lsc}^{1/2}$ where $b_\varepsilon$ is the principal symbol of $iA_{11, \varepsilon}$. Then 
\begin{equation*}
X^{(0) 2}_\varepsilon = T_\varepsilon + R^{(0)}_\varepsilon \qquad \text{with }  R^{(0)}_\varepsilon \in S^{0}_{lsc}.
\end{equation*}
Moreover, the remainder $R_\varepsilon^{(0)}$ is self-adjoint. Suppose we have $X^{(j)}_\varepsilon$, $j=0, \ldots , N$, such that 
\begin{equation}\label{eqn:self-adj-square-root}
\sum_{j=0}^{N} X^{(j)}_\varepsilon = T_\varepsilon + R_\varepsilon^{(N)}
\end{equation}
with $R^{(N)}_\varepsilon \in S^{-N}_{lsc}$ self-adjoint. Then with $X_\varepsilon^{(N+1)}:= - \frac{1}{2\sqrt{b_\varepsilon/\rho_\varepsilon}} R^{(N)}_\varepsilon$, $X^{(N+1)}_\varepsilon$ is self-adjoint since $b_\varepsilon/\rho_\varepsilon$ is real-valued and (\ref{eqn:self-adj-square-root}) holds with $N+1$ instead of $N$.

We recall from (\ref{eqn:DecoupleSystem}) that 
\begin{equation*}
P \begin{pmatrix} \partial_z & -\rho  \\ A_\rho & \partial_z  \end{pmatrix}  Q = \partial_z + P \begin{pmatrix} 0 & -\rho  \\ A_\rho & 0  \end{pmatrix}  Q + P \frac{\partial Q}{\partial z} = \partial_z - i \ \begin{pmatrix} B_+ & 0  \\ 0 & B_-  \end{pmatrix} + R. 
\end{equation*}
Using (\ref{eqn:selfadj}) we obtain for the second term in the middle expression
\begin{equation*}
P \begin{pmatrix} 0 & -\rho  \\ A_\rho & 0  \end{pmatrix}  Q = 
-\begin{pmatrix} 1 & 0  \\ 0 & -1  \end{pmatrix} \Bigl[ P \begin{pmatrix} 0 & -\rho  \\ A_\rho & 0  \end{pmatrix}  Q \Bigr]^* \begin{pmatrix} 1 & 0  \\ 0 & -1  \end{pmatrix}.
\end{equation*}
Now this term is the sum of an anti-self-adjoint diagonal matrix and a self-adjoint off-diagonal matrix. Also we have
\begin{equation*}
P \frac{\partial Q}{\partial z} = 
-\begin{pmatrix} 1 & 0  \\ 0 & -1  \end{pmatrix} \Bigl[ \ P \frac{\partial Q}{\partial z} \ \Bigr]^* \begin{pmatrix} 1 & 0  \\ 0 & -1  \end{pmatrix}.
\end{equation*}
So this expression is again the sum of a self-adjoint off-diagonal part and an anti-self-adjoint diagonal part. 

Hence $-iB_{\pm}$ are anti-self-adjoint and $B_{\pm}$ are self-adjoint.

It remains to show that $B_{\pm}$ has the form (\ref{eqn:solutionsBselfadj}). In order to compute the zeroth order term of $B_+$, we need to incorporate the principal and the subprincipal symbol of $P_{12}^{-1}$. We define the operator $P_{12}^{-1} = \mbox{Op}(\sqrt{2} a_\rho^{1/4} \rho^{-1/4}) + R_1$ of order $1/2$ and the symbol of $R_1$ is in $S^{-1/2}_{lsc}$. Then the parametrix $P_{12}$ has the following asymptotic expansion
\begin{equation*}
\begin{split}
P_{12,\varepsilon} &= \psi \Bigl[ \frac{1}{\sqrt{2}} a_{\rho,\varepsilon}^{-1/4} \rho_\varepsilon^{1/4} - \frac{1}{2} R_1 a_{\rho,\varepsilon}^{-1/2} \rho_\varepsilon^{1/2} +\\
& + \frac{1}{\sqrt{2}} \Bigl( \frac{i}{8} a_\varepsilon^{-3/2} \sum_{j=1}^{n-1}\frac{\partial a_\varepsilon}{\partial x_j} \xi_j - \frac{i}{4} \frac{1}{\rho_\varepsilon} a_\varepsilon^{-1/2} \sum_{j=1}^{n-1} \frac{\partial \rho_\varepsilon}{\partial x_j} \xi_j \Bigr) \rho_\varepsilon^{1/2} a_\varepsilon^{-3/4} \Bigr] + \\
&\phantom{=.}+\text{order}(-5/2).
\end{split}
\end{equation*}
where $\psi$ is a cut-off function as in the proof of Theorem~\ref{thm:parametrix}.
We obtain
\begin{equation*}
P_{12,\varepsilon} \frac{\partial P_{12,\varepsilon}^{-1}}{\partial z} = \frac{1}{4} \bigg( \frac{\partial a_{\rho,\varepsilon}}{\partial z} \frac{1}{a_{\rho,\varepsilon}} - \frac{\partial \rho_\varepsilon}{\partial z} \frac{1}{\rho_\varepsilon} \bigg) + \text{order}(-1)
\end{equation*}
and after some calculation
\begin{equation*}
P_{12,\varepsilon} A_{11,\varepsilon} P_{12,\varepsilon}^{-1} = A_{11,\varepsilon} - \frac{1}{2} \frac{1}{\rho_\varepsilon} a_\varepsilon^{-1/2} \sum_{j=1}^{n-1} \frac{\partial \rho_\varepsilon}{\partial x_j} \xi_j + \text{order}(-1).
\end{equation*}
Using (\ref{A_11}) and (\ref{A_12}) from section ~\ref{sec:Factorization} we get
\begin{equation*}
-ib_{+,\varepsilon} = -ib_\varepsilon + \frac{1}{2b_\varepsilon} \sum_{j=1}^{n-1} \frac{\partial b_\varepsilon}{\partial \xi_j} \frac{\partial b_\varepsilon}{\partial x_j} + \text{order}(-1)
\end{equation*}
where $b_\varepsilon := \sqrt{a_\varepsilon}$.
Similarly, if we write $P_{22}^{-1}:=\mbox{Op}(-\sqrt{2}  a_\rho^{1/4} \rho^{-1/4}) + R_2$, $R_2 \in S_{lsc}^{-1/2}$, then one can show that
\begin{equation*}
-ib_{-,\varepsilon} = ib_\varepsilon - \frac{1}{2b_\varepsilon} \sum_{j=1}^{n-1} \frac{\partial b_\varepsilon}{\partial \xi_j} \frac{\partial b_\varepsilon}{\partial x_j} + \text{order}(-1).
\end{equation*}
So the proof is complete.
\end{proof}
\begin{rem}
We derived the decoupled system (\ref{eqn:DecoupledWaves}) by using the two different exact factorizations of the full wave equation. Once factorized the equation, we were able to write down the normalization matrix explicitly. As a result, we obtained the decomposition in two first-order hyperbolic wave equations. Here, the factorizations were shown by induction and we got a direct connection between the exact factorizations of the wave equation and the decoupled first-order system.

This method is different to the one used in \cite{Stolk:04}. There, the first-order wave equations were obtained by recursively defining a normalization matrix which diagonalizes the matrix in (\ref{eqn:Matrix}) in the sense of (\ref{eqn:DecoupleSystem}). Moreover, no exact factorization of the wave equation is directly used as input data.
\end{rem}

With this in mind, we are now in the position to state the following main theorem.
\begin{thm}\label{thm:MicrolocalDecomposition}
Let $L = \partial_z \frac{1}{\rho} \partial_z + A_\rho(x,z,D_t,D_x)$ and $U,F \in \GLtwo(\mathbb{R}^{n+1})$. Then there are first-order pseudodifferential operators $B_{\pm}$ and a globally elliptic pseudodiffential operator matrix $Q$ as in Lemma~\ref{lem:Diag} such that the equation
\begin{equation}\label{startingpoint}
LU \equiv F \hspace{10pt} \mbox{microlocally on} \hspace{3pt} I_{\theta_2}
\end{equation}
holds if and only if
\begin{equation}\label{eqn_ml_plus}
L_{0,\pm} u_{\pm}  := \big( \partial_z - iB_{\pm}(x,z,D_t,D_x) \big) u_{\pm} \equiv f_{\pm} \hspace{5pt} \mbox{microlocally on} \hspace{3pt} I_{\theta_2}
\end{equation}
where
\begin{equation*}
\begin{pmatrix} u_{+} \\ u_{-} \end{pmatrix} := Q^{-1} \begin{pmatrix} U \\ \frac{1}{\rho} \partial_z U \end{pmatrix} , \hspace{25pt} \begin{pmatrix} f_{+} \\ f_{-} \end{pmatrix} := Q^{-1} \begin{pmatrix} 0 \\ F \end{pmatrix}.
\end{equation*}
Here $B_{\pm}$ is as in (\ref{eqn:solutionsB}) on $I_{\theta_2}'$ and $Q$ as in (\ref{eqn:solutionQ}) on $I_{\theta_2}'$. Moreover, by Lemma~\ref{lem:selfadj_B} the matrix $Q$ can be chosen such that $B_{\pm}$ become self-adjoint on $I_{\theta_2}'$.
\end{thm}
\begin{proof}
We define a microlocal cut-off $\psi=\psi(x,z,D_t,D_x,D_z)$ around $\xi=\tau=0$. In particular let $\psi(x,z,\tau,\xi,\zeta)$ be a homogeneous symbol of degree $0$ in $(\tau,\xi,\zeta)$ outside a compact set that satisfies
\begin{equation*}
\psi(x,z,\tau,\xi,\zeta) = \begin{cases} 0 & |\zeta| > 3 C |\tau| \\ 1 & |\zeta| < 2 C |\tau|, \ |\zeta| > 1. \end{cases}
\end{equation*}
We note that $\psi$ satisfies the requirements of [16, Theorem 18.1.35] since it vanishes if $|\zeta| > 3 C |\tau|$.

So given an operator $W \in \Psi^{m}_{lsc}$ on $I_{\theta_2}'$ the operator given by $\psi W$ is in $\Psi^m_{lsc}$ on $I_{\theta_2}$ by \cite[Theorem 18.1.35]{Hoermander:3}. Also, the symbol of $\psi W$ is equal to the symbol of $W$ modulo $\Psi^{-\infty}_{lsc}$ on $I_{\theta_2}$. 

As already announced in section~\ref{sec:Factorization}, the operators $L_{j}$, $j=1,2$ in (\ref{StateFact}) are not pseudodifferential operators on $\mathbb{R}^{n+1}$. But, when applied to the cut-off function $\psi=\psi(x,z,D_t,D_x,D_z)$, we see that $\psi L_{j}$ equals $L_{j} \psi$ and  $L_{j} \psi$ equals $L_{j}$ modulo $\Psi^{-\infty}_{lsc}$ on $I_{\theta_2}$, $j=1,2$.

We may regard (\ref{three1}) and (\ref{three2}) as microlocal equations on $I_{\theta_2}$. 

Then $LU \equiv F$ microlocally on $I_{\theta_2}$ if and only if 
\begin{equation*}
\psi P \begin{pmatrix} \partial_z & -\rho \\ A_\rho & \partial_z \end{pmatrix} QP \begin{pmatrix} U \\ \frac{1}{\rho} \partial_z  U \end{pmatrix} \equiv P \begin{pmatrix} \partial_z & -\rho \\ A_\rho & \partial_z \end{pmatrix} \begin{pmatrix} U \\ \frac{1}{\rho} \partial_z  U \end{pmatrix} \qquad \text{microlocally on} \ I_{\theta_2},
\end{equation*} 
$U \in \GLtwo(\mathbb{R}^{n+1})$. On the other hand we achieve
\begin{multline*}
\psi \bigg[\hspace{2pt} \begin{pmatrix} \partial_z \hspace{-2pt} - \hspace{-1pt} i B_{+} & 0 \\ 0 & \partial_z \hspace{-2pt} - \hspace{-1pt} i B_{-} \end{pmatrix} + R \bigg] P \begin{pmatrix} U \\ \frac{1}{\rho} \partial_z  U \end{pmatrix} \equiv \\ \equiv  \begin{pmatrix} \partial_z \hspace{-2pt} - \hspace{-1pt} i B_{+} & 0 \\ 0 & \partial_z \hspace{-2pt} - \hspace{-1pt} i B_{-} \end{pmatrix} P \begin{pmatrix} U \\ \frac{1}{\rho} \partial_z  U \end{pmatrix} \qquad \text{microlocally on} \ I_{\theta_2}
\end{multline*}
for $U \in \GLtwo(\mathbb{R}^{n+1})$. Thus
\begin{equation*}
P \begin{pmatrix} \partial_z & -\rho \\ A_\rho & \partial_z \end{pmatrix} \begin{pmatrix} U \\ \frac{1}{\rho} \partial_z  U \end{pmatrix} \equiv \begin{pmatrix} \partial_z \hspace{-2pt} - \hspace{-1pt} i B_{+} & 0 \\ 0 & \partial_z \hspace{-2pt} - \hspace{-1pt} i B_{-} \end{pmatrix} P \begin{pmatrix} U \\ \frac{1}{\rho} \partial_z  U \end{pmatrix}
\end{equation*}
holds microlocally on $I_{\theta_2}$. Therefore equation (\ref{startingpoint}) is equivalent to
\begin{equation*}
P \begin{pmatrix} 0 \\ F \end{pmatrix} \equiv \begin{pmatrix} \partial_z \hspace{-2pt} - \hspace{-1pt} i B_{+} & 0 \\ 0 & \partial_z \hspace{-2pt} - \hspace{-1pt} i B_{-} \end{pmatrix} P \begin{pmatrix} U \\ \frac{1}{\rho} \partial_z  U \end{pmatrix} \qquad \mbox{microlocally on} \ I_{\theta_2}.
\end{equation*}
\end{proof}
\section{Approximated first-order wave equations}
One particular technique for wave modeling is the one-way wave equation which is often used since it requires only one boundary condition, as it is a first-order evolution equation. In the case that the background data is smooth, it also provides a way to propagate in a predetermined direction (\cite{OptRoot2011}). Our background data involves generalized functions and due to the lack of an available theory connecting propagation of singularities with a generalized Hamiltonian flow we need a substitute.

With view to Theorem~\ref{thm:MicrolocalDecomposition}, we now introduce approximated first-order equations to (\ref{eqn_ml_plus}) of the form
\begin{equation}\label{approx_oneway}
L_\pm u_\pm := (\partial_z - iB_\pm(x,z,D_t,D_x) + C(x,z,D_t,D_x)) u_\pm = 0 \qquad z>z_0
\end{equation}
where the operator $C$ has a symbol in $S_{lsc}^1$ and serves as a correction term which should suppress certain singularities of the solutions. In detail, we choose the principal symbol $c=[(c_\varepsilon)_\varepsilon]$ of $C$ such that
\begin{alignat*}{2}
c_\varepsilon(x,z,\tau,\xi) &= 0 && \qquad \text{ on } I'_{\theta_1} \\
c_\varepsilon(x,z,\tau,\xi) &\ge \eta(\tau^2 + |\xi|^2)^{\frac{1}{2}} && \qquad \text{ outside } I'_{\theta_2}.
\end{alignat*}
for some positive constant $\eta \in \mathbb{R}$, all $\varepsilon \in (0,1]$ and where $0 < \theta_1 < \theta_2 < \pi/2$. The precise definition of the damping operator will be given in the subsection below. In this case, the operator $C$ is chosen logarithmic slow scale elliptic outside $I'_{\theta_2}$ and therefore $C$ is also slow scale elliptic there. Without the operator $C$ in (\ref{approx_oneway}) the operator $L_\pm$ reduces to $L_{0,\pm}$ and is a standard first-order hyperbolic operator on $I'_{\theta_2}$. In the region where $c$ is logarithmic slow scale elliptic the principal symbol of $L_\pm$ is logarithmic slow scale elliptic and therefore $c$ leads to a suppression of singularities of the solution to $L_\pm u_\pm = 0$, $z > z_0$. In detail, we have for the singularities of the solution to $L_\pm u_\pm=0$ , $z > z_0$ that
\begin{equation*}
\begin{split}
\mbox{WF}_g(u_\pm) & \subseteq WF_g(L_\pm u_\pm) \cup \text{Ell}_{lsc}^c(L_\pm)  \subseteq\\
& \subseteq \text{Ell}_{lsc}^c(L_{0,\pm}) \cap \text{Ell}_{lsc}^c(C) \subseteq\\
& \subseteq \{ (t,x,z,\tau,\xi;\zeta) \in T^*\mathbb{R}^{n+1} \setminus 0 \ | \ (x,z,\tau,\xi) \in I'_{\theta_2}, |\zeta| < C|\tau| \}  \qquad z > z_0
\end{split}
\end{equation*}
since $L_\pm u_\pm = 0$ for $z > z_0$. So, $C$ suppresses singularities outside $I'_{\theta_2}$.
Moreover, having solutions $u_\pm$ to $L_\pm u_\pm = 0$, $z > z_0$ with initial conditions $u_\pm(z_0) = u_{\pm,0}$ then they are also approximations to $L_{0,\pm} u_\pm \equiv 0$ microlocally on $I_{\theta_1}$ for $z > z_0$ with the same initial conditions. This can be seen by 
\begin{equation*}
\mbox{WF}_g(L_\pm u_\pm - L_{0,\pm} u_\pm) = \mbox{WF}_g(C u_\pm) \subseteq \mbox{WF}_g(u_\pm) \cap \mbox{$\mu$supp}_g(C) \qquad z > z_0
\end{equation*}
where $\mbox{$\mu$supp}_g(C) \subseteq I'^c_{\theta_1}$ and therefore
\begin{equation*}
\mbox{WF}_g(L_\pm u_\pm - L_{0,\pm} u_\pm) \cap I_{\theta_1} = \emptyset \qquad z > z_0.
\end{equation*}
We conclude that $L_\pm u_\pm \equiv L_{0,\pm} u_\pm$ microlocally on $I_{\theta_1}$, $z > z_0$. Then
\begin{equation*}
\psi \begin{pmatrix} \partial_z \hspace{-2pt} - \hspace{-1pt} i B_{+} & 0 \\ 0 & \partial_z \hspace{-2pt} - \hspace{-1pt} i B_{-} \end{pmatrix} \begin{pmatrix} u_+ \\ u_- \end{pmatrix} \equiv \vec{0} \qquad \text{microlocally on} \ I_{\theta_1},\ z>z_0
\end{equation*}
and by the proof of Theorem~\ref{thm:MicrolocalDecomposition} this is equivalent to 
\begin{equation}\label{eqn:ml_diag_reverse}
\psi \bigg[\hspace{2pt} P \begin{pmatrix} \partial_z & -\rho \\ A_\rho & \partial_z \end{pmatrix} Q - R \bigg] \begin{pmatrix} u_+ \\ u_- \end{pmatrix} \equiv \vec{0} \qquad \text{microlocally on} \ I_{\theta_1},\ z>z_0.
\end{equation} 
Since the entries of the pseudodifferential matrix $R$ are operators with symbol in $S^{-\infty}_{lsc}$ on $I_{\theta_2}'$, we have
\begin{equation*}
R  \begin{pmatrix} u_+ \\ u_- \end{pmatrix} \equiv \vec{0} \qquad \text{microlocally on} \ I_{\theta_1},\  z>z_0.
\end{equation*}
We now define $U$ being microlocally equivalent to $Q_+ u_+ + Q_- u_-$ on  $I_{\theta_1}, z>z_0$ where we have set $Q_+ := Q_{11}$ and $Q_- := Q_{12}$, i.e.
\begin{equation*}
U:=Q_+ u_+ + Q_- u_- + \widetilde{f}
\end{equation*}
for some $\widetilde{f} \equiv 0$ microlocally on $\ I_{\theta_1},\ z>z_0$. Then
\begin{equation*}
\begin{pmatrix} U \\ \frac{1}{\rho} \partial_z  U \end{pmatrix} \equiv Q \begin{pmatrix} u_+ \\ u_- \end{pmatrix} \qquad \text{microlocally on} \ I_{\theta_1}, \ z>z_0
\end{equation*}
since $u_+ \equiv Q_+ u_+ \equiv Q_{21} u_+$ and $u_- \equiv Q_- u_- \equiv Q_{22} u_-$ microlocally on $\ I_{\theta_1},\ z>z_0$. Hence (\ref{eqn:ml_diag_reverse}) holds if and only if 
\begin{equation*}
\psi P \begin{pmatrix} \partial_z & -\rho \\ A_\rho & \partial_z \end{pmatrix}  \begin{pmatrix} U \\ \frac{1}{\rho} \partial_z  U \end{pmatrix} \equiv \vec{0} \qquad \text{microlocally on} \ I_{\theta_1},\ z>z_0
\end{equation*} 
which means that $LU \equiv 0$ microlocally on $\ I_{\theta_1},\ z>z_0$.

In summary, we get the following. If $u_\pm$ solves the problem
\begin{alignat}{2}
L_\pm u_\pm &= 0 && \qquad \text{ on } (z_0,Z) \times \mathbb{R}^{n}\label{eqn_CP1_+} \\
u_\pm(z_0) &= u_{\pm,0} && \qquad \text{ on } \mathbb{R}^n \label{eqn_CP2_+}
\end{alignat}
then $U:=Q_+ u_+ + Q_- u_- + \widetilde{f}$ for some $\widetilde{f} \equiv 0$ microlocally on $\ I_{\theta_1},\ z \in (z_0, Z)$ solves 
\begin{equation*}
LU \equiv 0 \qquad \text{microlocally on} \ I_{\theta_1},\ z \in (z_0, Z).
\end{equation*}
\subsection{Requirements on the damping operator}
We now give sufficient conditions on the damping operator $C$ for the operator $L_\pm$ such that the Cauchy problems in (\ref{eqn_CP1_+})-(\ref{eqn_CP2_+}) are well-posed. The well-posedness will be the content of the next section. 

We define the principal symbol of $c$ of $C$ by
\begin{equation*}
c_\varepsilon(x,z,\tau,\xi) := \omega(x,z,\tau,\xi) (1-\chi_\varepsilon(x,z,\tau,\xi))
\end{equation*}
with $\chi_\varepsilon$ as in subsection~\ref{subsec:TechnicalPreliminaries} and $\omega$ is a smooth symbol homogeneous of degree 1 in $(\tau,\xi)$ and bounded below by some constant times $(\tau^2+|\xi|^2)^{1/2}$. Then $c_\varepsilon \in S^{1}_{lsc}(\mathbb{R}^n \times \mathbb{R}^n)$. Now, since $c_\varepsilon(x,z,\tau,\xi)$ is real-valued and homogeneous of degree 1, there exists a self-adjoint symbol $C_\varepsilon \in S^1_{lsc}(\mathbb{R}^n \times \mathbb{R}^n)$ with principal symbol equal to $c_\varepsilon$. This can easily seen by the following.

Defining 
\begin{equation*}
C_\varepsilon := \frac{c_\varepsilon + c^{ *}_\varepsilon}{2},
\end{equation*}
then $C_\varepsilon$ is the self-adjoint symbol as desired (cf. next section).  

In the following, we will work with such a self-adjoint damping operator $C_\varepsilon = C_\varepsilon^{(0)}+C^{(1)}_\varepsilon$ of order 1 with parameter $z$ and $C^{(k)}_\varepsilon \in S^{1-k}_{cl,lsc}$ ($k=0,1$). We set $C^{(0)}_\varepsilon := c_\varepsilon$ and get 
\begin{equation*}
\begin{split}
|\partial_{(x,z)}^{\beta} \partial_{(\tau,\xi)}^{\alpha} C_\varepsilon^{(k)}(x,z,\tau,\xi)| &\le C \omega_\varepsilon (1+|(\tau,\xi)|)^{1-k-|\alpha|} \le\\
&\le C \omega_\varepsilon (1+|(\tau,\xi)|)^{-|\alpha|-k+\frac{|\alpha|+|\beta|}{L}}(1+c_\varepsilon(x,z,\tau,\xi))^{1-\frac{|\alpha|+|\beta|}{L}} 
\end{split}
\end{equation*}
on the support of $C^{(0)}_\varepsilon = c_\varepsilon$ since we have $1+c_\varepsilon(x,z,\tau,\xi) \ge 1 + \eta(\tau^2+|\xi|^2)^{1/2}$ there. We thus get,
\begin{multline*}
\forall \alpha, \beta \in \mathbb{N}^n \ \forall L > |\alpha|+|\beta|+2k \ \exists \omega_\varepsilon \in \Pi_{lsc} \ \exists C>0 \ \exists \eta \in (0,1]:\\
|\partial_{(x,z)}^{\beta} \partial_{(\tau,\xi)}^{\alpha} C_\varepsilon^{(k)}(x,z,\tau,\xi)| 
\le C \omega_\varepsilon (1+|(\tau,\xi)|)^{-|\alpha|-k+\frac{|\alpha|+|\beta|}{L}}(1+c_\varepsilon(x,z,\tau,\xi))^{1-\frac{|\alpha|+|\beta|}{L}} 
\end{multline*}
for all $\varepsilon \in (0,\eta]$, $k=0,1$.
\section{Well-posedness of the approximated first-order equations}
To avoid an overkill through parameters, we will reduce the problem (\ref{eqn_CP1_+})-(\ref{eqn_CP2_+}) and show well-posedness for the following generalized Cauchy problem
\begin{alignat}{2}
P u &:= (\partial_z -iA(z,x,D_x)+B(z,x,D_x)) u= f && \qquad \text{ on } (0,Z) \times \mathbb{R}^{n}\label{eqn_CP1_general} \\
u(0,.) &\phantom{:}= u_0 && \qquad \text{ on } \mathbb{R}^n \label{eqn_CP2_general}.
\end{alignat}
where $u=u(z,x) \in \GLtwo(\mathbb{R}^{n+1})$ and $f=f(z,x) \in \GLtwo(\mathbb{R}^{n+1})$
under the following assumptions:
\begin{itemize}[leftmargin=.35in]
\setlength\itemsep{1em}
\item[(i)]
$A$ is a generalized pseudodifferential operator with symbol in $\widetilde{S}^{1}_{lsc}(\mathbb{R}^{n+1} \times \mathbb{R}^n)$.

\item[(ii)]
$B$ is a generalized pseudodifferential operator of order $\gamma>0$ with non-negative real homogeneous principal symbol $b=B^{(0)}=B^{(0)}(z,x,\xi)$, i.e. $b$ is homogeneous for $|\xi| \ge 1$ and in $\widetilde{S}^{-\infty}_{lsc}$ for $|\xi| \le 1$.

\item[(iii)]
Moreover, there are representatives $(B^{(0)}_\varepsilon)_\varepsilon$ and $(B^{(1)}_\varepsilon)_\varepsilon$ of $B^{(0)}$ and $B^{(1)} := B - B^{(0)}$ respectively such that the derivatives of $(B^{(0)}_\varepsilon)_\varepsilon$ and $(B^{(1)}_\varepsilon)_\varepsilon$ can be estimated as follows: $\forall \alpha \in \mathbb{N}^{n}$, $\forall \beta \in \mathbb{N}^{n+1}$, $\exists \omega_{0,\varepsilon},  \omega_{1,\varepsilon} \in \Pi_{lsc}$ such that
\begin{equation*}
|\partial^{\beta}_{(x,z)} \partial^{\alpha}_{\xi} B^{(0)}_\varepsilon(x,z,\xi)| =
\mathcal{O} (\omega_{0,\varepsilon}) (1+|\xi|)^{-|\alpha|+\frac{|\alpha|+|\beta|}{L} \gamma} (1+B^{(0)}_\varepsilon(x,z,\xi))^{1-\frac{|\alpha|+|\beta|}{L}}
\end{equation*}
for $|\alpha|+|\beta| < L$ as $\varepsilon \to 0$ and
\begin{equation*}
|\partial^{\beta}_{(x,z)} \partial^{\alpha}_{\xi} B^{(1)}_\varepsilon(x,z,\xi)| =
\mathcal{O} (\omega_{1,\varepsilon}) (1+|\xi|)^{-|\alpha|-1+\frac{|\alpha|+|\beta|+2}{L} \gamma} (1+B^{(0)}_\varepsilon(x,z,\xi))^{1-\frac{|\alpha|+|\beta|+2}{L}}
\end{equation*}
for $2+|\alpha|+|\beta| < L$ as $\varepsilon \to 0$.
\end{itemize}
\begin{rem}
In this section we construct a square root for the operator $1+B$ modulo a smoothing operator with symbol in $\widetilde{S}^{-\infty}_{lsc}(\mathbb{R}^{n+1} \times \mathbb{R}^n)$. Using this square root we show that the Cauchy problem (\ref{eqn_CP1_general})-(\ref{eqn_CP2_general}) has a unique solution that satisfies an energy estimate.
\end{rem}
\begin{lem}\label{lem:sqr}
Assume that $B(z,x,D_x)$ is a self-adjoint generalized pseudodifferential operator with symbol in $\widetilde{S}^{\gamma}_{lsc}(\mathbb{R}^{n+1} \times \mathbb{R}^n)$, $\gamma  > 0$, that satisfies (ii), (iii) and $2\gamma<L$.
Then, there exists $Q(z,x,D_x) \in \Psi^{\gamma/2}_{1-\frac{\gamma}{L},\frac{\gamma}{L},lsc}(\mathbb{R}^n)$ with $\partial^j_z Q(z,x,D_x) \in \Psi^{\gamma/2-j\gamma/L}_{1-\frac{\gamma}{L},\frac{\gamma}{L},lsc}(\mathbb{R}^n)$ such that $Q$ is self-adjoint and $Q^2 = 1+B+R$ where $R$ is a generalized pseudodifferential operator with symbol in $\widetilde{S}^{-\infty}_{lsc}(\mathbb{R}^{n+1} \times \mathbb{R}^n)$.
\end{lem}
\begin{rem}
In the proof we will make use of the following fact.

Let $f:[0,\infty] \to \mathbb{R}$ be a classical symbol of order $\delta$, i.e. $\forall k \in \mathbb{N}$, $\exists C_k>0$ such that $|f^{(k)}(y)| \le C_k (1+y)^{\delta-k}$, then $f \circ B^{(0)}$ is a symbol in $\widetilde{S}^{\gamma \max(\delta,0)}_{1-\frac{\gamma}{L},\frac{\gamma}{L},lsc}$. 

First note that $|f(B^{(0)})| \le C (1+B^{(0)})^{\delta} \le C \omega_\varepsilon (1+|\xi|)^{\gamma \max(\delta,0)}$.

So let $(\alpha,\beta) \neq \vec{0}$. Using Faà di Bruno's formula, we get that the expression $\partial^{\beta}_{(x,z)} \partial^{\alpha}_\xi f(B^{(0)})$ is a sum of terms of the form
\begin{equation}\label{eqn_sq}
cf^{(k)}(B^{(0)}) \prod_{j=1}^k \partial^{\beta_j}_{(x,z)} \partial^{\alpha_j}_\xi B^{(0)}
\end{equation}
for some constant $c$ and such that $\sum_j \alpha_j = \alpha$, $\sum_j \beta_j = \beta$ and $(\alpha_j,\beta_j) \neq 0$. Since $B^{(0)}$ satisfies (iii) one can estimate the expression in (\ref{eqn_sq}) by
\begin{equation*}
C \omega_\varepsilon (1+|\xi|)^{-|\alpha|+\frac{|\alpha|+|\beta|}{L}\gamma} (1+B^{(0)})^{\delta -\frac{|\alpha|+|\beta|}{L}} \le C \omega_\varepsilon (1+|\xi|)^{\gamma \max(\delta,0)-|\alpha|+\frac{|\alpha|+|\beta|}{L}\gamma}
\end{equation*}
where the constant $C$ and the logarithmic slow scale net $\omega_\varepsilon$ depend on $\alpha$ and $\beta$.
\end{rem}
\begin{proof}
The proof of Lemma~\ref{lem:sqr} follows the same lines as in \cite[Lemma 3]{Stolk:05}.
\end{proof}
In order to get the energy estimates, we will adapt a version of the sharp G\r{a}rding inequality theorem of Hörmander as stated in \cite[Theorem 18.1.14]{Hoermander:3}.
\begin{lem}[Sharp G\r{a}rding inequality]\label{SharpGardingInequality}
Suppose that $a \in \widetilde{S}^1_{lsc}(\mathbb{R}^n \times \mathbb{R}^n)$ has a representative $(a_\varepsilon)_\varepsilon$ such that $\exists \eta \in (0,1]$
\begin{equation*}
\operatorname{\rm{Re}} a_\varepsilon(x,\xi) \ge 0 \qquad \forall \varepsilon \in (0,\eta].
\end{equation*}
Then there exists a logarithmic slow scale net $(\omega_\varepsilon)_\varepsilon \in \Pi_{lsc}$ and a constant $C>0$ such that
\begin{equation}\label{eqn:sharpgarding}
\operatorname{\rm{Re}}( a_\varepsilon(x,D_x)u, u ) \ge -C\omega_\varepsilon \lVert u \rVert^2_{L^2}
\end{equation}
for all $u \in \mathscr{S}(\mathbb{R}^n)$ and $\varepsilon \in (0,\eta]$.
\end{lem}
The proof follows the same lines as in \cite[Theorem 18.1.14]{Hoermander:3}.
\begin{proof}
Let $r \in \widetilde{S}^0_{lsc}$ and $(r_\varepsilon)_\varepsilon \in r$ be a representative. Then, there exist a logarithmic slow scale net $(\omega_{\varepsilon})_\varepsilon \in \Pi_{lsc}$, a constant $C>0$ and an $\eta \in (0,1]$ such that 
\begin{equation*}
|(r_\varepsilon(x,D)u,u)| \le C \omega_{\varepsilon} \lVert u \rVert^2_{L^2} \qquad \varepsilon \in (0,\eta]
\end{equation*}
and
\begin{equation*}
\operatorname{\rm{Re}} (r_\varepsilon(x,D)u,u) \ge -C \omega_{\varepsilon} \lVert u \rVert^2_{L^2} \qquad \varepsilon \in (0,\eta]
\end{equation*}
Now let $a_\varepsilon \in S^1_{lsc}$ and write
\begin{equation*}
a_\varepsilon = \operatorname{\rm{Re}} a_\varepsilon + i \operatorname{\rm{Im}} a_\varepsilon
= \frac{a_\varepsilon+a_\varepsilon^*}{2} + \frac{a_\varepsilon-a_\varepsilon^*}{2}.
\end{equation*}
Then, since the second term on the right side is anti-self-adjoint it follows that
\begin{equation*}
\operatorname{\rm{Re}} \Bigl( \frac{a_\varepsilon-a_\varepsilon^*}{2} u, u \Bigr) = 0.
\end{equation*}
Moreover, since
\begin{equation*}
\operatorname{\rm{Re}} (a_\varepsilon) - \frac{a_\varepsilon+a_\varepsilon^*}{2} \in S^0_{lsc}
\end{equation*}
it suffices to show (\ref{eqn:sharpgarding}) with $a$ replaced by $\operatorname{\rm{Re}} a$.

As a next step, we let $\phi \in \Smooth_0(\mathbb{R}^{2n})$ be an even function with $L^2$-norm equal to $1$. Also we define $\psi \in \mathscr{S}$ by $\psi(x,D) := \phi(x,D)^* \phi(x,D)$. Then, $\psi$ is even and $\iint \psi(y,\eta) \,d y \,d \eta = 1$.

The aim is to write $a_\varepsilon$ as a superposition of two terms $a_{1,\varepsilon}$ and $a_{0,\varepsilon}$ where $a_{1,\varepsilon}$ is of the form
\begin{equation*}
a_{1,\varepsilon}(x,\xi) := \iint \psi((x-y)q(\eta),(\xi-\eta)/q(\eta)) a_\varepsilon(y,\eta) \,d y \,d \eta \qquad \varepsilon \in (0,1]
\end{equation*}
with $q(\eta) = (1+|\eta|^2)^{1/4}$.
Then, one can show that there exists an $\eta \in (0,1]$ such that
\begin{equation*}
(a_{1,\varepsilon}(x,D)u,u) \ge 0 \qquad u \in \mathscr{S}, \ \varepsilon \in (0,\eta]
\end{equation*} 
and $D^{\alpha}_{\xi} D^{\beta}_x a_{0,\varepsilon}(x,\xi) = D^{\alpha}_{\xi} D^{\beta}_x (a_\varepsilon(x,\xi)- a_{1,\varepsilon}(x,\xi))$ is a finite sum of terms
\begin{multline*}
b_{0,\varepsilon}(x,\xi) = \iint \psi_1((x-y)q(\eta),(\xi-\eta)/q(\eta)) b_\varepsilon(y,\eta) \,d y \,d \eta \\
 -b_\varepsilon(x,\xi) \iint \psi_1(y,\eta) \,d y \,d \eta \qquad \varepsilon \in (0,1]
\end{multline*}
where $\psi_1 \in \mathscr{S}$ is even and $(b_\varepsilon)_\varepsilon \in S^{1-|\alpha|}_{lsc}$. Then, by the same arguments as in \cite[Theorem 18.1.14]{Hoermander:3}, one can show that if $(b_\varepsilon)_\varepsilon \in S^{\mu}_{lsc}$ then there exists a logarithmic slow scale net $(\omega_\varepsilon)_\varepsilon$  such that $|b_{0,\varepsilon}(x,\xi)| = \mathcal{O}(\omega_\varepsilon) (1+|\xi|)^{\mu-1}$ as $\varepsilon \to 0$. This finishes the proof.
\end{proof}
\begin{rem}\label{rem_SharpGardingInequality}
Suppose that $r \in \widetilde{S}^0_{lsc}$. Then there exist a representative $(r_\varepsilon)_\varepsilon \in r$, a logarithmic slow scale net $(\omega_{1,\varepsilon})_\varepsilon \in \Pi_{lsc}$, a constant $C_1>0$ and an $\eta_1 \in (0,1]$ such that $|r_\varepsilon(x,\xi)| \le C_1 \omega_{1,\varepsilon}$ for all $\varepsilon \in (0,\eta_1]$. Hence, the symbol $(r_\varepsilon(x,\xi) + C_1 \omega_{1,\varepsilon})_\varepsilon$ satisfies the requirements of Theorem~\ref{SharpGardingInequality}, i.e.
\begin{equation*}
\operatorname{\rm{Re}} (r_\varepsilon(x,\xi) + C_1 \omega_{1,\varepsilon}) \ge 0 \qquad \varepsilon \in (0,\eta_1].
\end{equation*}
We therefore can conclude that there exist $(\omega_{2,\varepsilon})_\varepsilon \in \Pi_{lsc}$, a constant $C_2 >0$ and an $\eta_2 \in (0,1]$ such that for any $u \in \mathscr{S}$ and $\varepsilon \in (0,\eta_2]$ we have
\begin{equation*}
\operatorname{\rm{Re}}( (r_\varepsilon(x,D_x) + C_1 \omega_{1,\varepsilon})u, u ) \ge -C_2 \omega_{2,\varepsilon} \lVert u \rVert^2_{L^2} 
\end{equation*}
which is equivalent to the following
\begin{equation*}
\exists C>0 \ \exists (\omega_\varepsilon)_\varepsilon \in \Pi_{lsc}: \quad 
\operatorname{\rm{Re}}( r_\varepsilon(x,D_x) u, u ) \ge -C \omega_{\varepsilon} \lVert u \rVert^2_{L^2}
\end{equation*}
for all $u \in \mathscr{S}$ and $\varepsilon$ small enough.
\end{rem}
Concerning the well-posedness of the Cauchy problem (\ref{eqn_CP1_general})-(\ref{eqn_CP2_general}), we will first show the following energy estimate as shown in \cite{Stolk:05}, \cite{Hoermander:3}. A slight change in the proof shows that there is also an energy estimate when the symbols depend on $\varepsilon$.
\begin{lem}
Let $A$ and $B$ satisfy $(i), (ii), (iii)$ with $2\gamma < L$. If $s \in \mathbb{R}$ and $(\lambda_\varepsilon)_\varepsilon \in \Pi_{lsc}$ is generalized real and larger than some logarithmic slow scale net depending on $s$, then for every $u \in \mathcal{C}^1([0,Z]; H^{s}) \cap \mathcal{C}^0([0,Z]; H^{s+\hat{\gamma}})$ with $\hat{\gamma} := \max(\gamma,1)$ and every $p \in [1,\infty]$ we have:
\begin{equation}\label{eqn:energy-estimate}
\Big( \frac{1}{2} \int_0^Z \lVert e^{-\lambda_\varepsilon z} u(z,.) \rVert_{H^s}^p \lambda_\varepsilon \,d z \Big)^{\frac{1}{p}} \le \lVert u(0,.) \rVert_{H^s} +2 \int_0^Z e^{-\lambda_\varepsilon z}\lVert  P_\varepsilon u(z,.) \rVert_{H^s} \,d z
\end{equation}
with the interpretation of the maximum when $p=\infty$.
\end{lem}
\begin{proof}
Let $(A_\varepsilon)_\varepsilon$ be a representative of $A$. Since the principal symbol of $A$ is real-valued, we have that
\begin{equation*}
\operatorname{\rm{Re}}(-iA_\varepsilon(x,z,\xi)) \ge -C_1 \omega_{1,\varepsilon}
\end{equation*}
for some constant $C_1>0$ and some $(\omega_{1,\varepsilon})_\varepsilon \in \Pi_{lsc}$, for all $\varepsilon$ sufficiently small. Using the sharp G\r{a}rding inequality we obtain
\begin{equation}\label{eqn:Garding_iA}
\begin{split}
\exists (\omega_{2,\varepsilon})_\varepsilon \in \Pi_{lsc} \ \exists C_2 > 0 & \ \exists \eta_2 \in (0,1] : \\
&\operatorname{\rm{Re}} (-iA_\varepsilon(x,z,D_x) v, v) \ge -C_2 \omega_{2,\varepsilon} \lVert v \rVert^2_{L^2} \quad v \in H^1 
\end{split}
\end{equation}
for $z \in [0,Z]$, $\varepsilon \in (0,\eta_2]$.  

Note that by Remark~\ref{rem_SharpGardingInequality} the property (\ref{eqn:Garding_iA}) is invariant under zeroth-order perturbations of $-iA_\varepsilon(x,z,D_x)$. Now using Lemma~\ref{lem:sqr} we can write $B_\varepsilon=Q_\varepsilon^2-1-R_\varepsilon$ for some self-adjoint $Q_\varepsilon \in S^{\gamma/2}_{1-\frac{\gamma}{L},\frac{\gamma}{L},lsc} (\mathbb{R}^{n+1} \times \mathbb{R}^{n})$ and some $R_\varepsilon \in S^{-\infty}_{lsc} (\mathbb{R}^{n+1} \times \mathbb{R}^{n})$. We therefore obtain: $\exists (\omega_{\varepsilon})_\varepsilon \in \Pi_{lsc} \ \exists C > 0 \ \exists \eta \in (0,1]$ such that
\begin{align}
\operatorname{\rm{Re}}((-iA_\varepsilon+B_\varepsilon)(x,z,D_x) v, v) &= \operatorname{\rm{Re}}((-iA_\varepsilon-1-R_\varepsilon)(x,z,D_x) v, v) + (Q^2_\varepsilon v,v) =\nonumber \\
&= \operatorname{\rm{Re}}((-iA_\varepsilon-1-R_\varepsilon)(x,z,D_x) v, v) + (Q_\varepsilon v,Q_\varepsilon v) \ge \nonumber \\
&\ge -C \omega_\varepsilon \lVert v \rVert^2_{L^2} \qquad v \in H^{\hat{\gamma}} \label{eqn:Garding}
\end{align}
for any $z \in [0,Z]$ and $\varepsilon \in (0,\eta]$.

We first consider the corresponding $L^2$-energy estimates, i.e. $s=0$. Therefore, set $f_\varepsilon := \partial_z u - (iA_\varepsilon-B_\varepsilon)(x,z,D_x) u$. 
Then, taking the scalar products with respect to $z$ we 
\begin{equation*}
\frac{\partial}{\partial z} e^{-2 \lambda_\varepsilon z} \lVert u(z) \rVert^2_{L^2} = -2 \lambda_\varepsilon e^{-2 \lambda_\varepsilon z} \lVert u(z) \rVert^2_{L^2} + e^{-2 \lambda_\varepsilon z} 2 \operatorname{\rm{Re}}\bigl(\frac{\partial}{\partial z} u(z),u(z)\bigr)
\end{equation*}
and hence
\begin{equation*}
\begin{split}
2 \operatorname{\rm{Re}}(f_\varepsilon(z) & ,u(z)) e^{-2 \lambda_\varepsilon z} =\\
&= \frac{\partial}{\partial z} e^{-2 \lambda_\varepsilon z} \lVert u(z) \rVert^2_{L^2} + 2 \operatorname{\rm{Re}}\Bigl(\bigl((-iA_\varepsilon+B_\varepsilon)(x,z,D_x) +\lambda_\varepsilon \bigr) u(z),u(z)\Bigr) e^{-2 \lambda_\varepsilon z}\\
&\ge \frac{\partial}{\partial z} e^{-2 \lambda_\varepsilon z} \lVert u(z) \rVert^2_{L^2}
\end{split}
\end{equation*}
under the requirement that $(\lambda_\varepsilon)_\varepsilon \in \Pi_{lsc}$ with $\lambda_\varepsilon \ge C \omega_\varepsilon$, where $C$ and $\omega_\varepsilon$ are as in (\ref{eqn:Garding}). Integration from $0$ to $z$, $\bar{z} \le z \le Z$ then gives
\begin{equation*}
e^{-2 \lambda_\varepsilon z} \lVert u(z) \rVert^2_{L^2} - \lVert u(0) \rVert^2_{L^2} \le 2 \int_0^z e^{-\lambda_\varepsilon \bar{z}} \lVert f_\varepsilon(\bar{z}) \rVert_{L^2} e^{-\lambda_\varepsilon \bar{z}} \lVert u(\bar{z}) \rVert_{L^2} \,d \bar{z}.
\end{equation*}
Setting 
\begin{equation*}
M_\varepsilon(z):= \sup_{0 \le \bar{z} \le z} e^{-\lambda_\varepsilon \bar{z}} \lVert u(\bar{z}) \rVert_{L^2}
\end{equation*}
we get
\begin{equation*}
M_\varepsilon(z)^2 \le \lVert u(0) \rVert^2_{L^2} + 2 M_\varepsilon(z) \int_0^z e^{-\lambda_\varepsilon \bar{z}} \lVert f_\varepsilon(\bar{z}) \rVert_{L^2} \,d \bar{z}.
\end{equation*}
Hence
\begin{equation*}
\Bigl(M_\varepsilon(z)- \int_0^z e^{-\lambda_\varepsilon \bar{z}} \lVert f_\varepsilon(\bar{z}) \rVert_{L^2} \,d \bar{z}\Bigr)^2 \le \Bigl(\lVert u(0) \rVert_{L^2} + \int_0^z e^{-\lambda_\varepsilon \bar{z}} \lVert f_\varepsilon(\bar{z}) \rVert_{L^2} \,d \bar{z}\Bigr)^2
\end{equation*}
which yields
\begin{equation}\label{eqn:energy_L2}
e^{-\lambda_\varepsilon z} \lVert u(z) \rVert_{L^2} \le \lVert u(0) \rVert_{L^2} + 2 \int_0^z e^{-\lambda_\varepsilon \bar{z}} \lVert f_\varepsilon(\bar{z}) \rVert_{L^2} \,d \bar{z}.
\end{equation}
for all $(\lambda_\varepsilon)_\varepsilon \in \Pi_{lsc}$ with $\lambda_\varepsilon \ge C \omega_\varepsilon$. In (\ref{eqn:energy_L2}) we set $\lambda_\varepsilon = C \omega_\varepsilon$ and multiply the obtained result by the factor $e^{(C\omega_\varepsilon-\lambda_\varepsilon)z}$. Then 
\begin{equation*}
e^{-\lambda_\varepsilon z} \lVert u(z) \rVert_{L^2} \le e^{(C\omega_\varepsilon-\lambda_\varepsilon)z} \lVert u(0) \rVert_{L^2} + 2 \int_0^z e^{-\lambda_\varepsilon \bar{z}} \lVert f_\varepsilon(\bar{z}) \rVert_{L^2} e^{(C\omega_\varepsilon-\lambda_\varepsilon)(z-\bar{z})} \,d \bar{z}.
\end{equation*}
If we choose $(\lambda_\varepsilon)_\varepsilon \in \Pi_{lsc}$ with $\lambda_\varepsilon > 2C \omega_\varepsilon$, then $C\omega_\varepsilon-\lambda_\varepsilon \le -\frac{\lambda_\varepsilon}{2}$ and hence
\begin{equation}\label{eqn:energy_L2_1}
e^{-\lambda_\varepsilon z} \lVert u(z) \rVert_{L^2} \le e^{-\frac{\lambda_\varepsilon}{2}z} \lVert u(0) \rVert_{L^2} + 2 \int_0^z e^{-\lambda_\varepsilon \bar{z}} \lVert f_\varepsilon(\bar{z}) \rVert_{L^2} e^{-\frac{\lambda_\varepsilon}{2}(z-\bar{z})} \,d \bar{z}.
\end{equation}
This shows the desired energy estimate in the case that $p=\infty$. For $1 \le p < \infty$ we observe for the $L^p$-norm of the first term on the right-hand side of (\ref{eqn:energy_L2_1}):
\begin{equation*}
\int_0^Z e^{-\frac{\lambda_\varepsilon}{2} p z} \,d z = \frac{2}{\lambda_\varepsilon p} \Bigl(1-e^{-\frac{\lambda_\varepsilon}{2} p Z} \Bigr) \le \frac{2}{\lambda_\varepsilon p}.
\end{equation*}
So,
\begin{equation*}
\lVert e^{-\frac{\lambda_\varepsilon}{2}z} \lVert u(0) \rVert_{L^2} \rVert_{L^p} \le \Bigl(\frac{2}{\lambda_\varepsilon} \Bigr)^{\frac{1}{p}} \lVert u(0) \rVert_{L^2}.
\end{equation*}
Concerning the $L^p$-norm of the second term on the right-hand side of (\ref{eqn:energy_L2_1}), we compute
\begin{equation*}
\begin{split}
2^p \int_0^Z  & e^{-\frac{\lambda_\varepsilon}{2} p z}\Bigl( \int_0^z e^{-\lambda_\varepsilon \bar{z}} \lVert f_\varepsilon(\bar{z}) \rVert_{L^2} e^{\frac{\lambda_\varepsilon}{2}\bar{z}} \,d \bar{z} \Bigr)^p \,d z =\\
& = -2^p \frac{2}{\lambda_\varepsilon p} e^{-\frac{\lambda_\varepsilon}{2} p z} \Bigl(\int_0^z e^{-\lambda_\varepsilon \bar{z}} \lVert f_\varepsilon(\bar{z}) \rVert_{L^2} e^{\frac{\lambda_\varepsilon}{2}\bar{z}} \,d \bar{z} \Bigr)^p \Bigr|^Z_0 +\phantom{\Bigr)^{p-1} e^{-\lambda_\varepsilon z} \lVert f_\varepsilon(z) \rVert_{L^2} e^{\frac{\lambda_\varepsilon}{2}z} \,d z} \\
& + 2^p \frac{2}{\lambda_\varepsilon} \int_0^Z  e^{-\frac{\lambda_\varepsilon}{2} p z} \Bigl(\int_0^z e^{-\lambda_\varepsilon \bar{z}} \lVert f_\varepsilon(\bar{z}) \rVert_{L^2} e^{\frac{\lambda_\varepsilon}{2}\bar{z}} \,d \bar{z} \Bigr)^{p-1} e^{-\lambda_\varepsilon z} \lVert f_\varepsilon(z) \rVert_{L^2} e^{\frac{\lambda_\varepsilon}{2}z} \,d z \le\\
& \le 2^p \frac{2}{\lambda_\varepsilon} \int_0^Z \Bigl( \int_0^z e^{-\lambda_\varepsilon \bar{z}} \lVert f_\varepsilon(\bar{z}) \rVert_{L^2} e^{-\frac{\lambda_\varepsilon}{2}(z-\bar{z})} \,d \bar{z} \Bigr)^{p-1} e^{-\lambda_\varepsilon z} \lVert f_\varepsilon(z) \rVert_{L^2} \,d z \le \\
& \le \frac{2}{\lambda_\varepsilon} \Bigr( 2 \int_0^Z e^{-\lambda_\varepsilon z} \lVert f_\varepsilon(z) \rVert_{L^2} \,d z \Bigr)^p.
\end{split}
\end{equation*}
Summarizing and using the Minkowski inequality, we obtain for every $p \in [1,\infty]$
\begin{equation*}
\Bigl( \int_0^Z \lVert e^{-\lambda_\varepsilon z} u(z,.) \rVert^p_{L^2} \,d z \Bigr)^{\frac{1}{p}} \le \Bigl( \frac{2}{\lambda_\varepsilon} \Bigr)^{\frac{1}{p}} \Bigl( \lVert u (0,.) \rVert_{L^2} + 2 \int_0^Z e^{-\lambda_\varepsilon z} \lVert P_\varepsilon u(z,.) \rVert_{L^2} \,d z  \Bigl)
\end{equation*}
which shows the theorem in the case that $s=0$.

The $L^2$-energy estimate can now be used to obtain the $H^s$-energy estimates for $s \in \mathbb{R}$, $s \neq 0$. Therefore, let $u \in \mathcal{C}^1([0,Z];H^{s}(\mathbb{R}^n)) \cap \mathcal{C}^0([0,Z];H^{s+\hat{\gamma}}(\mathbb{R}^n))$. Then, $\langle D_x \rangle^s P_\varepsilon \langle D_x \rangle^{-s}$ satisfies the same assumptions as $P_\varepsilon$ and $\langle D_x \rangle^s u \in \mathcal{C}^1([0,Z];L^{2}(\mathbb{R}^n)) \cap \mathcal{C}^0([0,Z];H^{\hat{\gamma}}(\mathbb{R}^n))$. Note that in particular we have an estimate of the form (\ref{eqn:Garding}) if we replace $-iA_\varepsilon + B_\varepsilon$ by $\langle D_x \rangle^s (-iA_\varepsilon + B_\varepsilon) \langle D_x \rangle^{-s}$ and $v$ by $\langle D_x \rangle^s u$ but now with a lower bound $C(s) \omega_{\varepsilon}(s)$ depending on $s$. Using the $L^2$-energy estimate we obtain
\begin{equation*}
\begin{split}
\Bigl( \frac{1}{2} \int_0^Z \lVert e^{-\lambda_\varepsilon z} & u(z,.) \rVert^p_{H^s} \lambda_\varepsilon \,d z \Bigr)^{\frac{1}{p}} = \Bigl( \frac{1}{2} \int_0^Z \lVert e^{-\lambda_\varepsilon z} \langle D_x \rangle^s u(z,.) \rVert_{L^2})^p \lambda_\varepsilon \,d z \Bigr)^{\frac{1}{p}} \le\\
& \le \lVert \langle D_x \rangle^s u (0, .) \rVert_{L^2} + 2 \int_0^Z e^{-\lambda_\varepsilon z} \lVert \langle D_x \rangle^s P_\varepsilon \langle D_x \rangle^{-s} \langle D_x \rangle^s u(z, .) \rVert_{L^2} \,d z =\\
& = \lVert u (0,.) \rVert_{H^s} + 2 \int_0^Z e^{-\lambda_\varepsilon z} \lVert P_\varepsilon u(z,.) \rVert_{H^s} \,d z  
\end{split}
\end{equation*}
since $P_\varepsilon \langle D_x \rangle^{-s} \langle D_x \rangle^s -P_\varepsilon \in S^{-\infty}_{lsc}$.
This completes the proof.
\end{proof}
\begin{thm}
Let $A$ and $B$ be generalized symbols with parameter $z \in [0,Z]$ that satisfy $(i), (ii)$ and $(iii)$ with $2\gamma < L$. Then, for every $f \in \GLtwo((0,Z) \times \mathbb{R}^n)$ and $u_0 \in \GLtwo(\mathbb{R}^n)$ the Cauchy problem (\ref{eqn_CP1_general})-(\ref{eqn_CP2_general}) has a unique solution $u \in \GLtwo((0,Z) \times \mathbb{R}^n)$ and the energy estimate (\ref{eqn:energy-estimate}) remains valid for this solution.
\end{thm}
\begin{proof}
Let $(g_\varepsilon)_\varepsilon \in g$, $(f_\varepsilon)_\varepsilon \in f$ be representatives. We fix $\varepsilon \in (0,1]$ and consider the smooth Cauchy problem
\begin{alignat}{2}
P_\varepsilon u_\varepsilon &= f_\varepsilon && \qquad \text{ on } (0,Z) \times \mathbb{R}^{n}\label{eqn_CP1_rep} \\
u_\varepsilon(0,.) &= u_{0,\varepsilon} && \qquad \text{ on } \mathbb{R}^n. \nonumber
\end{alignat}
Then, by \cite[Theorem 23.1.2]{Hoermander:3} and \cite[Theorem 5]{Stolk:05}, we get the existence of a solution $u_\varepsilon \in \Smooth([0,Z]; H^{\infty}(\mathbb{R}^n))$. Note that the additional regularity with respect to $z$ variable follows from (\ref{eqn_CP1_rep}).  

Concerning uniqueness, assume that $f_\varepsilon = 0$ and $u_{0,\varepsilon} = 0$. 
Furthermore, suppose that moderateness of the solutions is already shown.
We apply the energy estimate (\ref{eqn:energy-estimate}) and get
\begin{equation*}
e^{-\lambda_\varepsilon Z} \lVert u_\varepsilon(z,.) \rVert_{H^{s}} 
\le \max_{z \in [0,Z]}  \lVert e^{-\lambda_\varepsilon z} u_\varepsilon(z,.) \rVert_{H^{s}} = 0 \qquad \forall s \ge 0
\end{equation*}
Therefore, $u_\varepsilon = 0$.

It remains to show that $(u_\varepsilon)_\varepsilon \in \mathcal{E}_{M,L^2}((0,Z) \times \mathbb{R}^n)$. For moderateness with respect to $z \in [0,Z]$ we note that $\lVert u_\varepsilon \rVert_{L^2((0,Z) \times \mathbb{R}^n)} \le Z \sup_{z \in [0,Z]} \lVert u_\varepsilon \rVert_{L^2(\mathbb{R}^n)}$.
 
Since (\ref{eqn:energy-estimate}) is applicable for $u_\varepsilon \in \Smooth([0,Z]; H^{\infty}(\mathbb{R}^n))$, we obtain
\begin{equation*}
\begin{split}
e^{-\lambda_\varepsilon Z} \lVert u_\varepsilon(z,.) \rVert_{H^{s}} &\le \max_{z \in [0,Z]}  \lVert e^{-\lambda_\varepsilon z} u_\varepsilon(z,.) \rVert_{H^{s}} \le\\
&\le \lVert u_\varepsilon(0,.) \rVert_{H^s} +2 \int_0^Z e^{-\lambda_\varepsilon z}\lVert  P_\varepsilon u_\varepsilon(z,.) \rVert_{H^s} \,d z.
\end{split}
\end{equation*}
Hence
\begin{equation*}
\begin{split}
\lVert u_\varepsilon(z,.) \rVert_{H^{s}} &\le e^{\lambda_\varepsilon Z} \Bigl( \lVert u_\varepsilon(0,.) \rVert_{H^s} +2 \int_0^Z e^{-\lambda_\varepsilon z}\lVert  P_\varepsilon u_\varepsilon(z,.) \rVert_{H^s} \,d z \Bigr).
\end{split}
\end{equation*}
and therefore $u_\varepsilon$ is moderate with respect to $x$ if and only if $\lambda_\varepsilon$ is of log-type, which is indeed satisfied because of the assumption $(\lambda_\varepsilon)_\varepsilon \in \Pi_{lsc}$. Recall that a net $r_\varepsilon$ is of log-type if $|r_\varepsilon|= \mathcal{O} \bigl(\log(\frac{1}{\varepsilon}) \bigr)$ as $\varepsilon \to 0$.

Concerning the $z$-derivatives, we write
\begin{equation*}
\partial_z \langle D_x \rangle^s u_\varepsilon = \langle D_x \rangle^s (iA_\varepsilon-B_\varepsilon)(x,z,D_x) u_\varepsilon + \langle D_x \rangle^s f_\varepsilon
\end{equation*}
and obtain that for every $s \ge 0$ there exists an $N \in \mathbb{N}$ such that $\lVert \partial_z \langle D_x \rangle^s u_\varepsilon \rVert_{L^2} = \mathcal{O}(\varepsilon^{-N})$ uniformly in $z$. For the higher order $z$-derivatives one uses an induction argument when differentiating the equation (\ref{eqn_CP1_rep}).
\end{proof}

\subsection*{Acknowlegdments}
The author is very grateful to G\"{u}nther H\"{o}rmann and Shantanu Dave for many helpful discussions and suggestions during the preparation of the paper.

\newpage

\mbox{}
\setlength\parskip{0pt	plus 1pt minus 0pt}

\bibliographystyle{plain}

\end{document}